\documentclass[12pt]{amsart}

\usepackage{pb-diagram}

\usepackage{xcolor}

\usepackage{amsfonts}
\usepackage{amsmath}
\usepackage{amssymb}

\allowdisplaybreaks

\newcommand{\tr}{\textnormal{tr}}
\newcommand{\ric}{\textnormal{Ric}}

\newcommand{\dbar}{\overline{\partial}}

\newcommand{\ddt}[1]{\frac{\partial #1}{\partial t}}

\newcommand{\cS}{\mathcal{S}}

\newcommand{\ddbar}{\sqrt{-1}\partial\dbar}

\newcommand{\mref}[1]{(\ref{#1})}

\newtheorem{theorem}{Theorem}[section]
\newtheorem{proposition}{Proposition}[section]
\newtheorem{lemma}{Lemma}[section]

\newtheorem{definition}{Definition}[section]
\newtheorem{corollary}{Corollary}[section]
\newtheorem{remark}{Remark}[section]

\renewcommand{\thefootnote}{\fnsymbol{footnote}}
\newcommand{\starttext}{ \setcounter{footnote}{0}
\renewcommand{\thefootnote}{\arabic{footnote}}}

\newcommand{\beq}{\begin{equation}}
\newcommand{\bea}{\begin{eqnarray}}
\newcommand{\eea}{\end{eqnarray}} \newcommand{\ee}{\end{equation}}

\def\ba{\begin{eqnarray}}
\def\ea{\end{eqnarray}}

\def\cK{{\mathcal K}}

\def\ti\tilde

\def\u{\underline}

\def\tr{{\rm tr}}

\def\ti{\tilde}

\begin{document}

\starttext \baselineskip=18pt \setcounter{footnote}{0}

\starttext \baselineskip=18pt \setcounter{footnote}{0}

\title{SOBOLEV INEQUALITIES ON K\"AHLER SPACES}

\author[{Bin Guo, Duong H. Phong, Jian Song and Jacob Sturm}
]{Bin Guo$^*$, Duong H. Phong$^\dagger$, Jian Song$^\ddagger$ and Jacob Sturm$^{\dagger\dagger}$ }

\thanks{Work supported in part by the National Science Foundation under grants DMS-22-03273, DMS-22-03607 and DMS-23-03508, and the collaboration grant 946730 from Simons Foundation.}

\address{$^*$ Department of Mathematics \& Computer Science, Rutgers University, Newark, NJ 07102}

\email{bguo@rutgers.edu}

\address{$^\dagger$ Department of Mathematics, Columbia University, New York, NY 10027}

\email{phong@math.columbia.edu}

\address{$^\ddagger$ Department of Mathematics, Rutgers University, Piscataway, NJ 08854}

\email{jiansong@math.rutgers.edu}

\address{$^{\dagger\dagger}$ Department of Mathematics \& Computer Science, Rutgers University, Newark, NJ 07102}

\email{sturm@andromeda.rutgers.edu}

\begin{abstract} 
We establish a uniform Sobolev inequality for K\"ahler metrics, which only require an entropy bound and no lower bound on the Ricci curvature. We further extend our Sobolev inequality to singular K\"ahler metrics on K\"ahler spaces with normal singularities. This allows us to build a general theory of global geometric analysis on singular K\"ahler spaces including the spectral theorem, heat kernel estimates, eigenvalue estimates and diameter estimates. Such estimates were only known previously in very special cases such as Bergman metrics.  As a consequence, we derive various geometric estimates, such as the diameter estimate and the Sobolev inequality, for K\"ahler-Einstein currents on projective varieties with definite or vanishing first Chern class. 
{\footnotesize }

\end{abstract}

\maketitle

{ \footnotesize  \tableofcontents}

\baselineskip=15pt
\setcounter{equation}{0}
\setcounter{footnote}{0}

\section{Introduction}
\setcounter{equation}{0}

The Sobolev inequality is a fundamental tool in global geometric analysis. The optimal Sobolev constant reflects the geometric structure of the underlying Riemannian manifold, and geometric curvature bounds will in turn control the Sobolev constant. The classical works of Yau and his collaborators have built a wide range of geometric estimates for compact or complete Riemannian manifolds  assuming a lower bound for the Ricci curvature (c.f. \cite{ScYa, CL, CLY}). In \cite{GPSS}, we were able to establish uniform Green's formula and diameter estimates for a large family of K\"ahler metrics on a K\"ahler manifold without such a Ricci curvature lower bound.
Rather, these estimates only depended on an entropy bound. 
One of the first main results in the present paper
is to obtain a uniform Sobolev inequality for K\"ahler metrics, again assuming only an entropy bound and no Ricci curvature lower bound. As a consequence, under similar assumptions, we also obtain a wide range of geometric estimates such as heat kernel estimate and eigenvalue estimates.
  
\smallskip  
In his celebrated work \cite{Y} on the Calabi conjecture, Yau initiated the theory of complex Monge-Amp\`ere equations with singular data, in the hope of understanding canonical K\"ahler metrics with singularities and anticipating future applications to algebraic geometry. The analytic theory was subsequently developed by Ko\l{}odziej \cite{K} using pluripotential theory, and later further extended to complex Monge-Amp\`ere equations on singular K\"ahler varieties \cite{EGZ, DPali}. However, very little geometric information can be extracted by pluripotential theory for singular K\"ahler metrics, and particularly for K\"ahler-Einstein metrics on projective varieties with log terminal singularities. Very recently, an alternative proof of Ko\l{}odziej’s estimates has been found by Guo, Phong, and Tong in \cite{GPT}, using PDE methods. These PDE methods have had many important consequences \cite{GPa, GP, GPSS}.
Their key features are their flexibility and uniformity, which allow many geometric estimates to be passed on to K\"ahler spaces with normal singularities. 

\smallskip
The next main goal of the present paper is to develop a geometric theory of complex Monge-Amp\`ere equations with {\it singularities}. A theory with sharp geometric estimates had been a well-known challenge in presence of singularities. This is because geometric invariants such as the diameter and the Sobolev inequality are defined in terms of the underlying metric, whose very definition requires two derivatives of the potential. Bounds for such high derivatives of the potential are usually not available in presence of singularities. Another basic and subtle issue right at the start, is to identify suitable K\"ahler spaces with singularities which would contain K\"ahler spaces of interest, such as those carrying canonical metrics such as K\"ahler-Einstein metrics. In this context, we note that the standard notion on smooth K\"ahler metrics on normal varieties, defined as restrictions of smooth K\"ahler metrics on the ambiant space, is inadequate, since such notions would exclude K\"ahler-Einstein metrics.

\smallskip

Nevertheless, building in part on the methods of \cite{GPT, GPS, GPSS} and especially the uniform Sobolev estimates on K\"ahler manifolds obtained in the first part of the paper, we shall identify a natural class of singular K\"ahler spaces (c.f. Definition 3.1) and 
establish for them a whole package of estimates such as Sobolev inequalities, the spectral theorem, heat kernel estimates, diameter estimates and volume non-collapsing estimates. These results provide both conceptual and technical platforms for a geometric regularity theory on singular K\"ahler spaces, which had long been an objective in both analytic and algebraic geometry. As an application, we derive a unique diameter bound for singular K\"ahler-Einstein spaces. 
Applying our methods to the K\"ahler-Ricci flow, we can extend the Sobolev inequality to the metric completion of the singular K\"ahler current and get the diameter/noncollapsing bound from the Sobolev inequality without any curvature assumptions. We note that
diameter bounds for the metric completion $(\hat X,d)$ of normal K\"ahler spaces
have been an open problem for a long time except for special cases when $(X, \omega)$ is a limit of a sequence of K\"ahler manifolds with suitable Ricci curvature bounds (cf. \cite{RZ, S1}).
On the analytic side, our Sobolev inequality allows classical PDE techniques such as the Moser's iteration to work on singular K\"ahler spaces. Regarding the general problem of geometric analysis with singularities, we would like to mention the recent beautiful work of Cao et al. \cite{CGN} on Hermitian-Einstein metrics on spaces with normal singularities in the sense of Grauert, which also relied in an essential manner on the estimates of \cite{GPT, GPS}.

\smallskip
The paper is organized as follows. Sections 2 to 5 are devoted to a precise formulation of our results, with the case of smooth K\"ahler manifolds described in Section \ref{smooth},  and the case of K\"ahler spaces with singularities described in Section \ref{secapp}, which includes applications to singular K\"ahler-Einstein metrics and finite-time limits of the K\"ahler-Ricci flow.
Throughout Sections \ref{section 3} to \ref{section 5}, we will assume that $X$ is a smooth K\"ahler manifold. In Section \ref{section 3}, we will present the proof of Theorem \ref{thm:Sob}, while Theorem \ref{thm:heat} will be proven in Section \ref{section 4}. Additionally, Section \ref{section 5} will contain statements of local versions of Sobolev inequalities. In the remaining sections, we will assume that $X$ is a normal K\"ahler variety. We will consider then the applications of Theorems \ref{thm:Sob} and \ref{thm:heat}, and in particular provide the proof of the results formulated in Section \ref{secapp}.

\section{Uniform bounds for smooth K\"ahler manifolds}
\label{smooth}
\setcounter{equation}{0}

We first state our results for smooth K\"ahler manifolds. Let $(X,\omega_X)$ be a compact K\"ahler manifold without boundary with $\omega_X$ a fixed K\"ahler metric. Without loss of generality we normalize $\omega_X$ such that $\int_X\omega_X^n = 1$. Suppose that the complex dimension of $X$ is $n$. Given a K\"ahler metric $\omega$ on $X$, we denote its volume by 
$$V_\omega=[\omega]^n = \int_X \omega^n$$
 and define the relative volume function by
$$e^{F_\omega} = \frac{1}{V_\omega}\frac{\omega^n}{\omega_X^n}.$$ 
Given $p\ge 1$ we define the $p$-th Nash-Yau entropy of $\omega$ by
$${\mathcal N}_p(\omega) = \frac{1}{V_\omega}\int_X \Big|\log \frac{1}{V_\omega} \frac{ \omega^n}{\omega_X^n} \Big|^p \;{\omega^n} =  \int_X |F_\omega|^p e^{F_\omega} \omega_X^n.$$
For a given {\em nonnegative} {\em continuous} function $\gamma\in C^0(X)$, and given parameters $0< A\le +\infty, K>0$, as in \cite{GPSS}, we consider the following subset of the space of K\"ahler metrics on $X$:
\begin{equation}\label{eqn:the set}
{\mathcal W}:={\mathcal W}(n,p, A, K,\gamma ): = \Big\{ \omega :\; [\omega]\cdot [\omega_X]^{n-1}< A,\, {\mathcal N}_p(\omega) \le K, \, e^{F_\omega} \ge \gamma \Big\}.
\end{equation}
It has been shown in \cite{GPSS} that if $p>n$ and the vanishing locus of $\gamma$ is relatively small (c.f. (\ref{eqn:gamma})), then  the K\"ahler metrics $\omega\in {\mathcal W}(n,p,A, K, \gamma )$ have uniformly bounded diameter and satisfy local noncollapsing.  In this paper, we aim to study the heat kernels of the Laplacian $\Delta_\omega$ of K\"ahler metrics $\omega$ in this subset, and as applications, we give lower bounds of eigenvalues of  $-\Delta_\omega$, by using the approach of Cheng-Li \cite{CL}. The main tool is the following Sobolev-type inequality. Throughout the paper, as in \cite{GPSS} we always assume $\gamma\in C^0(X)$ satisfies 
\begin{equation}\label{eqn:gamma}
\dim_{{\mathcal H}} \{\gamma = 0 \} < 2n -1,\quad \gamma\ge 0,
\end{equation} 
where ${\mathrm {dim}}_{{\mathcal H}}E$ denotes the Hausdorff dimension of a closed subset $E\subset X$. In our earlier work \cite{GPSS}, uniform estimates are established for the Green's function and the diameter associated to K\"ahler metrics in $\mathcal{W}(n, p, A, K, \gamma)$ . We denote 
\begin{equation}\label{norinter}
I_\omega = [\omega]\cdot[\omega_X]^{n-1}
\end{equation}
 to be the intersection of the K\"ahler classes $[\omega]$ and $[\omega_X]$. To state results for possibly unbounded K\"ahler classes, we will write the admissible set of K\"ahler metrics with $A = \infty$ in \mref{eqn:the set} as ${\mathcal W}(n, p, \infty, K, \gamma)$. 


\begin{theorem}\label{thm:Sob} 

Given $p>n$, $K>0$ and $\gamma\in C^0(X)$ satisfying (\ref{eqn:gamma}), there exist   $q = q(n, p)>1$ and $C = C(n, p,   K,\gamma, q)>0$ such that  for any K\"ahler metric $\omega\in {\mathcal W}(n,p,\infty, K, \gamma )$ and any $u\in W^{1,2}(X)$,  we have the following Sobolev-type inequality 
\begin{equation}\label{eqn:Sob}
\Big(\frac{1}{V_\omega}\int_X | u - \overline{u}  |^{2q}\omega^n   \Big)^{1/q}\le C \frac{ I_\omega}{V_\omega} \int_X |\nabla u|_\omega^2 \omega^n,
\end{equation}
where $\overline{u} = \frac{1}{V_\omega}\int_X u \omega^n$ is the average of $u$ over $(X,\omega)$.
\end{theorem}

We remark that the Sobolev constant in Theorem \ref{thm:Sob} does not depend on the curvature of the K\"ahler metric $\omega$. The Sobolev inequality \mref{eqn:Sob} is indeed {\em  scale invariant}, although the exponent $q$ may be  smaller than $\frac{n}{n-1}$. We also stress that if we impose the Ko\l{}odziej-type condition on $e^{F_\omega}$, i.e. its $L^{1+\epsilon'}(X,\omega_X^n)$-norm is uniformly bounded for some $\epsilon'>0$, then the constant $q>1$ can be chosen as close as possible  to the {exponent} $\frac{n}{n-1}$ as in the Euclidean case, with the constant $C_q$ possibly blowing up as $q$ approaches $ \frac{n}{n-1}$. This fact will be clear from the proof of Theorem \ref{thm:Sob} by combining the almost sharp estimates on the Green's functions established in \cite{GPS}.\footnote{In \cite{GPS} the metrics are assumed to be uniformly {\em big}, however, the same proof remains applicable to the degenerating metrics by slight modifications.} However, for the geometric applications in this paper, having $q>1$ suffices.

We note that the Sobolev inequality (\ref{eqn:Sob}) can also be used to prove the diameter estimate of $(X,\omega)$, which has been established in our previous work \cite{GPSS} using the Green's formula. We are informed by Guedj, Guenancia and Zeriahi that they very recently obtained diameter estimates under hypotheses different from ours (c.f. \cite{GGZ}). 

\smallskip

With the Sobolev inequality in Theorem \ref{thm:Sob}, we can now state the heat kernel upper bounds and eigenvalue lower bounds for K\"ahler metrics in $\mathcal{W}(n, p, \infty, K, \gamma)$.  Recall that the heat kernel of the Laplacian $\Delta_\omega$ satisfies the equation
\begin{equation}\label{eqn:heat}
\frac{\partial}{\partial t} H(x,y,t) = \Delta_{\omega, y} H(x,y,t) \quad \mbox {and } \lim_{t\to 0^+} H(x,y,t) = \delta_x(y). 
\end{equation}
It is well-known that $H(x,y,t)$ is a positive smooth function on $X\times X\times (0,\infty)$, and it is symmetric in $x$ and $y$. 
\begin{theorem}\label{thm:heat}
Let $q = q(n,p)>1$ be the constant in Theorem \ref{thm:Sob}. There exists $C=C(n,p,K , \gamma)>0$ such that for any $x,y\in X$
\begin{equation}\label{eqn:heat kernel}H(x,y,t) \le  \left\{\begin{array}{ll}
\frac{C}{V_\omega}\big(\frac{I_\omega}{t} \big)^{\frac{q}{q-1}} \exp\big({-\frac{d_\omega(x,y)^2}{10t}}\big), & \quad \mbox{if }t\in (0,I_\omega]\\
 \frac{C}{V_\omega}  \exp\big({-\frac{d_\omega(x,y)^2}{10t}}\big), & \quad\mbox{if }t \in (I_\omega,\infty).
\end{array}\right.\end{equation}Here $d_\omega(x,y)$ is the geodesic distance between $x$ and $y$ with repsect to the metric $\omega$. 
When $x=y$, a stronger upper bound of the heat kernel on the diagonal of $X\times X$ holds as below,
\begin{equation}\label{eqn:CL}H(x,x,t) \le \frac{1}{V_\omega} + \frac{C}{V_\omega} I_\omega^{\frac{q}{q-1}}t^{-\frac{q}{q-1}}. \end{equation}
Moreover, if $t\ge I_\omega$, we have for some positive constant $c>0$
\begin{equation}\label{eqn:exp heat}
\sup_{x,y\in X}\big| H(x,y,t) - \frac{1}{V_\omega}  \big| \le \frac{C}{V_\omega} e^{- c I_\omega^{-1} t}.
\end{equation}
\end{theorem} 
The classical heat kernel upper bound is achieved in \cite{CLY}. Theorem \ref{thm:heat} does not require any assumption on the Ricci curvature and it will also give uniform lower bounds for the eigenvalues associated to metrics in $\mathcal{W}(n, p, \infty, K, \gamma)$.
Let $0=\lambda_0 < \lambda_1 \le \lambda_2\le \ldots$ (counting multiplicity) be the increasing sequence of eigenvalues of the Laplacian $-\Delta_\omega$. We have the following lower bound on these eigenvalues.
\begin{corollary}\label{cor:CL}
There exists  $c = c(n,p, K, \gamma)>0$ such that for any $k\in {\mathbb N}$,  
$$\lambda_k\ge c \, I_\omega^{-1}  k^{\frac{q-1}{q}}.$$
\end{corollary}
We observe that the uniform positive lower bound of $\lambda_1$ was proved in \cite{GPSS}. The known estimates on higher order eigenvalues required lower bounds on the Ricci curvature \cite{CL}.


\section{Applications to singular K\"ahler spaces}\label{secapp}
\setcounter{equation}{0}

In this section, we will apply the Sobolev inequality and the heat kernel estimates to normal K\"ahler varieties. We will focus on the family of closed positive $(1,1)$-currents whose log volume measure has only log type analytic singularities (c.f. Definition \ref{logtype}).   

\begin{definition} \label{defak} Let $X$ be an $n$-dimensional compact normal K\"ahler space. Let $\pi: Y \rightarrow X$ be a log resolution of singularities and let $\theta_Y$ be a smooth K\"ahler metric  on the nonsingular model $Y$.  We define the set of admissible semi-K\"ahler currents 
$${\mathcal{AK}}(X, \theta_Y, n, p, A, K, \gamma),$$  
 to be the set of any semi-K\"ahler current $\omega$ on $X$ satisfying the following conditions.

\begin{enumerate}

\item $[\omega]$ is a K\"ahler class on $X$ and $\omega$ has bounded local potentials. 

\smallskip

\item $[\pi^*\omega] \cdot [\theta_Y]^{n-1}\leq A$ and $[\omega]^n \geq A^{-1}$.  

\smallskip

\item The $p$-th Nash-Yau entropy is bounded for some $p>n$, i.e.
$${\mathcal{N}}_p (\omega) = \frac{1}{V_\omega} \int_Y \left|\log \frac{1}{V_\omega} \frac{(\pi^*\omega)^n}{(\theta_Y)^n} \right|^p (\pi^*\omega)^n  \leq K, $$
where $V_\omega= [\omega]^n$.

\medskip

\item The log volume measure ratio 
$$\log \left( \frac{(\pi^*\omega)^n}{(\theta_Y)^n} \right)$$ has log type analytic singularities (c.f. Definition \ref{logtype}). There also exists a non-negative $\gamma \in C^0(Y)$ such that $\{\gamma=0\}$ is contained in a proper analytic subvariety of $Y$ and 
 $$\frac{(\pi^*\omega)^n}{ (\theta_Y)^n}\geq \gamma.$$

\end{enumerate}

\end{definition}

The nonsingular model $Y$ is indeed K\"ahler (cf. Lemma 2.2 \cite{CMM}) in Definition \ref{defak}.  The $\mathcal{AK}(X, \theta_Y, n, p, A, K, \gamma)$ space is the extension of $\mathcal{W}(n, p, A, K, \gamma)$ with two key additional conditions: (1) uniform lower bound of the volume for $[\omega]$ to prevent collapsing;  (2) the volume measure ratio has log type analytic singularities to guarantee smoothness of $\omega$ in an open Zariski subset of $X$. 

\smallskip

Smooth K\"ahler metrics on a normal variety $X$ are defined locally as the restriction of smooth K\"ahler metrics in the ambient space (cf. Definition \ref{defn:6.1}). In general, 
canonical metrics on $X$ such as K\"ahler-Einstein metrics are not smooth K\"ahler metrics in this sense. For example, if $X$ is a projective variety with orbifold singularities,  the smooth orbifold K\"ahler metrics are never smooth K\"ahler metrics on $X$. However, the K\"ahler-Einstein metrics on projective varieties belong to the $\mathcal{AK}$-space of Definition \ref{defak} and they are not smooth near the algebraic singularities. This is one of the reasons we consider the   $\mathcal{AK}$-spaces.  

\begin{remark} \label{bignefex}   The space $\mathcal{AK}(X, \theta_Y, n, p, A, K, \gamma)$ in Definition (\ref{defak}) can be easily generalized for positive  currents $\omega$ with minimal singularities in a big and nef class.  
Most of our results still hold for such positive currents $\omega$. We choose to state our results for $\mathcal{AK}(X, \theta_Y, n, p, A, K, \gamma)$ with the K\"ahler assumptions for convenience and simplicity. In fact, the definition for $\mathcal{AK}$-space of currents can be further extended without the nef assumption. 

\end{remark}

For any $\omega \in \mathcal{AK}(X, \theta_Y, n, p, A, K, \gamma)$, we let $\cS_{X, \omega}$ be the union of the singular set of $X$ and the singular set of $\pi_*\left( \log \frac{(\pi^*\omega)^n}{(\theta_Y)^n} \right)$. By definition, $\cS_{X, \omega}$ is an analytic subvariety of $X$. The current $\omega$ is a smooth K\"ahler metric on $X\setminus \cS_{X, \omega}$ (c.f. Proposition \ref{omapp}). 

\begin{definition} Let $X$ be an $n$-dimensional a compact normal K\"ahler variety and 
$$\omega \in \mathcal{AK}(X, \theta_Y, n, p, A, K, \gamma).$$ We define
$$(\hat X, d) = \overline{(X\setminus \cS_{X, \omega}, \omega|_{X\setminus \cS_{X, \omega}})}$$
to be the metric completion of $\left(X\setminus \cS_{X,\omega}, \omega|_{X\setminus \cS_{X,\omega}} \right)$. We also denote the unique metric measure space  associated with $(X, \omega)$ by 
$$(\hat X, d, \omega^n).$$

\end{definition}

We remark that $\omega^n$ extends uniquely to a volume measure on $\hat X$ because both $\omega$ and $\omega^n$ does not carry mass on $\cS_{X, \omega}$.

\smallskip

Our goal is to extend Theorem \ref{thm:Sob} and Theorem \ref{thm:heat}  to the metric measure space $(\hat X, d, \omega^n)$ associated to $(X, \omega)$. We then apply the Sobolev inequality to  establish the volume non-collapsing and diameter bounds for $(\hat X, d, \omega^n)$ in Section \ref{singdiasec}. A spectral theorem and heat kernel estimates will also be proved in Section \ref{section heat} and Section \ref{section 10}. Finally, we will apply our estimates to singular K\"ahler-Einstein spaces in Section \ref{section 11} and limiting spaces of non-collapsing finite time solutions of the K\"ahler-Ricci flow in Section \ref{section 12}.

\subsection{Sobolev inequality, heat kernel and diameter estimates}

$~$ 

\medskip

Let $X$ be an $n$-dimensional normal K\"ahler variety. For any $p>n$ and any $\omega\in \mathcal{AK}(X, \theta_Y, n,p, A, K, \gamma)$, we let $(\hat X, d, \omega^n)$ the metric measure space associated to $(X, \omega)$. We will consider the Sobolev space $W^{1,2}(\hat X,d, \omega^n)$ as in Definition \ref{sobdef} on  the metric measure space $(\hat X, d, \omega^n)$ associated to $\omega$.

\begin{theorem} \label{singsob} Let $X$ be an $n$-dimensional compact normal K\"ahler space. For any $\omega\in \mathcal{AK}(X, \theta_Y, n, p, A, K, \gamma)$, the metric measure space $(\hat X, d, \omega^n)$ associated to $(X, \omega)$ satisfies the following properties. 

\begin{enumerate}

\item There exists $C=C(X, \theta_Y, n, p, A, K, \gamma)>0$ such that  
$${\textnormal{diam}}(\hat X, d) \leq C.$$
In particular, $(\hat X, d)$ is a compact metric space.

\medskip 

\item  There exist $q>1$ and $C_S=C_S(X, \theta_Y, n, p,  A, K, \gamma, q)>0$ such that 
$$
\Big(\int_{\hat X} | u  |^{2q}\omega^n   \Big)^{1/q}\le C_S \left( \int_{\hat X} |\nabla u|^2 ~\omega^n + \int_{\hat X} u^2 \omega^n \right) .
$$
for all $u\in W^{1, 2}(\hat X, d, \omega^n)$. 

\medskip

\item There exists $C=C(X, \theta_Y, n, p, A, K, \gamma)>0$ such that the following trace formula holds for the heat kernel of $(\hat X, d, \omega^n)$
$$H(x,x, t) \leq \frac{1}{V_\omega} + \frac{C}{V_\omega} t^{-\frac{q}{q-1}}. $$ 

\medskip

\item Let $0=\lambda_0 < \lambda_1 \leq \lambda_2 \leq ... $ be the increasing sequence of eigenvalues of the Laplacian $-\Delta_\omega$ on $(\hat X, d, \omega^n)$. Then there exists $c=c(X, \theta_Y, n, p, A, K, \gamma)>0$ such that
$$\lambda_k \geq c k^{\frac{q-1}{q}}. $$

\end{enumerate}

\end{theorem}

To the best of our knowledge,
Theorem \ref{singsob} is the first to provide general formulas for the Sobolev inequality, the heat kernel,  and eigenvalue estimates on complex spaces with singularities. Such estimates were only known  previously to hold in special cases such as the following:
\begin{enumerate}
\item $X$ is projective and $\omega$ is a Bergman metric via projective embedding (cf. \cite{LT, Be}).

\smallskip

\item $(X, \omega)$ is a sequential Gromov-Hausdorff limit of K\"ahler-Einstein manifolds (cf. \cite{DS}). 

\end{enumerate}

\smallskip

Geometric estimates  are traditionally studied for  smooth K\"ahler metrics on singular K\"ahler spaces. The most immediate concern is that, as we pointed out earlier, canonical K\"ahler metrics near algebraic singularities cannot be in general the restriction of smooth metrics via local or global embeddings. In addition, most algebraic singularities are not smoothable and so most singular K\"ahler metrics cannot be approximated by smooth K\"ahler metrics with Ricci curvature uniformly bounded from below. 

\smallskip
Thus the traditional geometric estimates cannot be directly applied to the study of canonical metrics. Nor do the traditional methods of proof, which usually require that the metrics are smooth K\"ahler metrics on normal K\"ahler varieties.
For example, the heat kernel trace formula and Sobolev inequalities are obtained in \cite{Pa, LT} for Bergman metrics on projective varieties. Bergman metrics are restrictions of Fubini-Study metrics via global projective embeddings and so they are smooth K\"ahler metrics on normal K\"ahler varieties. 

\smallskip
Canonical metrics, such as K\"ahler-Einstein metrics, require rather a different and much larger class of singular K\"ahler spaces, such as the classes $\mathcal{AK}$-spaces as in Definition \ref{defak}. These lie beyond the reach of traditional methods, but they can be treated by our Sobolev inequality and trace formula for the heat kernel. We also observe that our Sobolev inequality can be applied to the study of Hermitian-Einstein metrics on stable sheaves over singular  normal K\"ahler spaces as developed by  \cite{BS, CGN}.

\smallskip
 
As we had stressed in the Introduction,
diameter bounds for $(X\setminus \mathcal{S}_{X, \omega})$ or $(\hat X, d)$ have been an open problem for a long time. The uniform diameter bound in Theorem \ref{singsob} is particularly striking, in the sense that it mostly relies on the entropy bound rather than the topological and complex structures of singularities.  

\subsection{Singular K\"ahler-Einstein metrics}

We will apply the diameter estimates and Sobolev estimates of Theorem \ref{singsob} to K\"ahler-Einstein spaces. 

Let $X$  be a projective Calabi-Yau variety of $\dim X = n$, i.e. $X$ is normal and $c_1(X) =0$. It is proved in \cite{EGZ}, for any K\"ahler class $\alpha$ on $X$, there exists a unique positive closed $(1,1)$-current $\omega_{CY} \in \alpha$ with bounded local potential and volume measure if $X$ has at worst log terminal singularities.  In particular, $\omega_{CY}$ is a smooth K\"ahler metric on $X^\circ$, the regular part of $X$. We will apply the Sobolev inequality (\ref{singsob}) to obtain the following geometric estimates. 

\begin{theorem} \label{kecy} Let $X$ be a projective Calabi-Yau variety with log terminal singularities. Then for any K\"ahler class $\alpha$, there exists a unique Ricci-flat current $\omega_{CY} \in \alpha$ satisfying the following.

\begin{enumerate}

\item Let $(\hat X, d_{CY})$ be the metric completion of $(X^\circ, \omega_{CY})$. Then $(\hat X, d_{CY}, (\omega_{CY})^n)$  is a compact metric measure space.

\medskip 

\item For any $q\in (1, \frac{n}{n-1})$, there exists $C_S=C_S(X, n, \alpha, q)>0$ such that 
\begin{equation}\label{anasob}
\Big(\int_{\hat X} | u  |^{2q} (\omega_{CY})^n   \Big)^{1/q}\le C_S \left( \int_{\hat X} |\nabla u|^2_{\omega_{CY} }~(\omega_{CY})^n + \int_{\hat X} u^2 (\omega_{CY} )^n \right) .
\end{equation}
for all $u\in W^{1, 2}(\hat X, d_{CY}, (\omega_{CY})^n)$. 

\medskip 

\item For any $\mu >n$, there exists $\kappa=\kappa(X, n, \alpha, \mu)>0$ such that for any $p\in \hat X$ and $r\in (0, \textnormal{diam}(\hat X, d_{CY})]$, 
$$\frac{{\mathrm{Vol}}(B(p, r))}{ r^{2\mu}} \geq \kappa, $$ where $B(p,r)$ denotes the ball in $(\hat X, d)$ with center $p$ and radius $r$.

\end{enumerate}

\end{theorem}
We note that Theorem \ref{kecy} remains valid when $X$ is only a \emph{K\"ahler} Calabi-Yau variety. The proof remains unchanged.

\smallskip

We will also derive geometric estimates for K\"ahler-Einstein currents on canonical models. Let $X$ be an $n$-dimensional  projective normal variety with ample $K_X$ and log terminal singularities. There exists a unique K\"ahler-Einstein current $\omega_{KE} \in [K_X]$ constructed in \cite{EGZ} such that $\omega_{KE}$ has bounded local potential and volume measure satisfying $${\mathrm{Ric}}(\omega_{KE}) = - \omega_{KE}$$
 smoothly on $X^\circ$, the regular part of $X$. 

\begin{theorem} \label{kegt} Let $X$ be an $n$-dimensional  projective normal variety with ample $K_X$ and log terminal singularities. Then the unique K\"ahler-Einstein current $\omega_{KE} \in [K_X]$ satisfies the following properties:

\begin{enumerate}

\item Let $(\hat X, d_{KE})$ be the metric completion of $(X^\circ, \omega_{KE})$. Then $(\hat X, d_{KE}, (\omega_{KE})^n)$  is a compact metric measure space.

\medskip 

\item For any $q\in (1, \frac{n}{n-1})$, there exists $C_S=C_S(X, n, q)>0$ such that 
$$
\Big(\int_{\hat X} | u  |^{2q}(\omega_{KE})^n   \Big)^{1/q}\le C_S \left( \int_{\hat X} |\nabla u|^2_{\omega_{KE}} (\omega_{KE})^n + \int_{\hat X} u^2 (\omega_{KE})^n \right) .
$$
for all $u\in W^{1, 2}(\hat X, d_{KE}, (\omega_{KE})^n)$. 

\medskip 

\item For any $\mu >n$, there exists $\kappa=\kappa(X, n, \mu)>0$ such that for any $p\in \hat X$ and $r\in (0, \textnormal{diam}(\hat X, d_{KE})]$, 
$$\frac{{\mathrm{Vol}}(B(p, r))}{ r^{2\mu}} \geq \kappa. $$

\end{enumerate}

\end{theorem}

The existence of K\"ahler-Einstein currents on a K-stable $\mathbb{Q}$-Fano variety is established in \cite{LC, LXZ} as a singular version of the Yau-Tian-Donaldson conjecture. The analogue of Theorem \ref{kegt} also holds for such singular K\"ahler-Einstein metrics with positive scalar curvature. 
\begin{theorem} \label{kefano} Let $X$ be an $n$-dimensional $\mathbb{Q}$-Fano variety with log terminal singularities. If there exists a K\"ahler-Einstein current $\omega \in -[K_X]$ with bounded local potentials, then $\omega\in C^\infty(X^\circ)$ and $(\hat X, d_{KE}, (\omega_{KE})^n)$ associated with $(X, \omega_{KE})$ is a compact metric measure space satisfying estimates in (2) and (3) of Theorem \ref{kegt}.

\end{theorem}

We stress that the above diameter and non-collapsing estimates go considerably farther than those obtained in \cite{GPSS}. Indeed, the diameter and non-collapsing estimates 
in \cite{GPSS} are for smooth K\"ahler manifolds, and even though they were uniform in many key parameters, they still do not apply to the present setting of singular spaces $(X,\omega)$ by taking limits. This is because the convergence is only uniform on any compact subset of the regular part of $X$, which is too weak and not sufficient to obtain the convergence of geodesics or geodesic balls. Furthermore, the arguments in \cite{GPSS} cannot be adapted to the singular setting because of the lack of estimates of the Green's functions on the singular space. Instead, we have to rely on the Sobolev inequality to prove these estimates.

\smallskip
The Sobolev inequality has the advantage that it can be passed down to the limit or singular space, thanks to the uniform Sobolev inequality we establish on the smooth resolution. However, this is not obvious for the Green's function because the gradient of Green's function is not in $L^2$. This latter is the technique we used in \cite{GPSS}.  

\smallskip 
It is also the Sobolev inequality which enables us to estimate the heat kernel of the Laplacian, and its eigenvalues, hence obtain a spectral theorem on singular spaces (more precisely on the metric completion of such spaces).

\subsection{Finite time singularities of the K\"ahler-Ricci flow}

We consider the following unnormalized K\"ahler-Ricci flow on a K\"ahler manifold $X$ with an initial K\"ahler metric $g_0$ 
 \begin{equation}\label{krflow}
\left\{
\begin{array}{l}
{ \displaystyle \ddt{g} = -\ric(g),}\\
\\
g|_{t=0} =g_0 \in H^{1,1}(X, \mathbb{R})\cap H^2(X, \mathbb{\mathbb{Q}}).
\end{array} \right.
\end{equation}
Let 
\begin{equation}\label{singT}
T= \sup\{ t >0~|~ [g_0]+t[K_X] >0 \} \in \mathbb{R} \cup\{\infty\}.
\end{equation}
The K\"ahler-Ricci flow has a maximal solution $g(t)$ for $t\in [0, T)$. If $T<\infty$, the flow (\ref{krflow}) must develop singularities at $t=T$ and the curvature tensors must blow up as $t\rightarrow T^-$.  In this case, $\alpha_T= [g_0]+T[K_X]$ is a nef class on $X$.  Since $[g_0]\in H^2(X, \mathbb{Q})$, $\alpha_T$ is semi-ample by Kawamata's base point free theorem and the numerical dimension of $\alpha_T$ coincides with its Kodaira dimension. In particular, $\alpha_T$ induces a surjective holomorphic map 
\begin{equation}\label{kawam}
\Phi: X \rightarrow Y.
\end{equation}
 We consider the  case when $\alpha_T$ is big, which is equivalent to either of the following conditions 
\begin{enumerate}\label{big}
\smallskip

\item  $(\alpha_T)^n= \lim_{t\rightarrow T} {\rm Vol}_{g(t)}(X) >0.$

\smallskip

\item $\dim Y = n$. 

\end{enumerate}
We let $S$ be the exceptional locus of $\Phi$, i.e., where $\Phi$ is not biholomorphic. $S$ is then an analytic subvariety of $X$.

\begin{theorem}\label{thm:main2} Let $(X, g_0)$ be a K\"ahler manifold equipped with a K\"ahler metric $g_0$. If $g(t)$ is the maximal solution of the K\"ahler-Ricci flow (\ref{krflow}) for $t\in [0, T)$ for some $T\in \mathbb{R}^+$ and if the limiting class $\alpha_T$ is big, then the following hold.

\begin{enumerate}

\item $g(t)$ converges smoothly a K\"ahler metric $g(T)$ on $X\setminus S$. There exists $D=D(X, g_0, n)>0$ such that for all $t\in [0, T)$, 
$$diam(X, g(t)) \leq D.$$ 
For any $\mu>n$, there exists $\kappa=\kappa(X, g_0, n, \mu)>0$ such that for all $t\in [0, T)$ and $r\in (0, D]$, 
$$\frac{{\mathrm{Vol}}_{g(t)} (B_{g(t)}(x, r))}{r^{2\mu}} \geq \kappa.$$

\medskip

\item Let $\omega(T)$ be the unique positive current extended by $g(T)$ and $(\hat X, d_T)$ be the metric completion of $(X\setminus S, g(T))$ . Then the metric measure space $(X_T, d_T, \omega(T)^n)$ is a compact metric space. For any $\mu>n$, there exists $\kappa=\kappa(\mu)>0$ such that for any $r \in (0, \textnormal {diam}( X_T, d_T))$,  we have
$$ \frac{{\mathrm{Vol}}_{\omega(T)^n} (B_{d_T} (x, r))}{r^{2\mu}}\geq \kappa $$ 
for any $x\in X_T$.
 
 \end{enumerate}

\end{theorem}

As in the previous section, it is worth pointing out that these results go beyond the corresponding ones in \cite{GPSS}, as they apply directly to the limiting K\"ahler current at the singular time $T$. In \cite{GPSS}, only uniform bounds for the diameter and noncollapsing for the smooth metrics along the KR flow were obtained, but it is not possible to take limits because the convergence on compact subsets of the smooth parts is too weak. Again, it is the Sobolev inequality extended to the metric completion of the singular K\"ahler current that can apply to
get the present diameter and noncollapsing estimates.

\smallskip
The first statement of Theorem \ref{thm:main2} implies that after passing to a sequence $t_j\rightarrow T^-$, $(X, g(t_j))$ always converges in Gromov-Hausdorff  sense to a compact metric space $(\hat X_T, \hat d_T)$ which satisfies the same volume non-collapsing condition. In fact, the identity map between $X\setminus S$ and itself extends to a surjective Lipschitz map from $(X_T, d_T)$ to $(\hat X_T, \hat d_T)$. We conjecture that 
$(\hat X_T, \hat d_T)$ and $(X_T, d_T)$ are isomorphic and they are both homeomorphic to the original variety $Y$. This conjecture holds for K\"ahler surfaces by the work of \cite{SW1, SW2, SY, So13}. We would also like to remark that Theorem \ref{thm:main2} still holds if $[g_0]$ is a general K\"ahler class as a consequence of Remark \ref{bignefex},  since the limiting class $\alpha$ is big and nef.  Theorem \ref{thm:main2} can also be generalized to the unique analytic solution $g(t)$ to the K\"ahler-Ricci flow on projective varieties with log terminal singularities whenever $t>0$, because the semi-K\"ahler current $g(t)$ lies in $\mathcal{AK}(X, \theta_Y, n, p, A, K, \gamma)$ for suitable choices of a log resolution $\pi:Y \rightarrow X$, $\theta_Y$, $A$, $p>n$ and $K>0$ whenever $t$ is uniformly bounded away from $0$. 


\section{The Sobolev-type inequality}\label{section 3}
\setcounter{equation}{0}
\newcommand{\redn}[1]{\textcolor{black}{#1}}

In this section we will prove the Sobolev inequality in Theorem \ref{thm:Sob}. We first let $A>0$ be a finite number and consider the  parameters $p>n, K>0$ and $\gamma\in C^0(X)$ satisfying (\ref{eqn:gamma}). We will fix a K\"ahler metric $\omega \in {\mathcal W}(n,p,A, K, \gamma )$. We recall the {\em Green's function} $G_\omega(x,y)$ of the Laplacian $\Delta_\omega$ on $X$ is defined by 
$$\Delta_\omega G_\omega(x,\cdot) = \frac{1}{V_\omega} - \delta_x(\cdot),\quad \int_X G_\omega(x,\cdot)\omega^n = 0,$$ 
where $\delta_x(\cdot)$ is the Dirac function at the base point $x\in X$. It is well-known that $G_\omega$ always exists and is smooth in $X\times X$ off the diagonal. Moreover, $G_\omega$ is symmetric in the sense that $G_\omega(x, y) = G_\omega(y, x)$ for any $x\neq y$. We first recall the lower bound of $G_\omega$ from \cite{GPSS}.
\begin{lemma}[\cite{GPSS}] 
There exists a constant $C_0 = C_0(n,p, A, K, \gamma)>0$ such that 
$$\inf_{x,y \in X} G_\omega(x,y)\ge -C_0/V_\omega.$$
\end{lemma}
Given this lemma, we denote the positive function 
$${\mathcal G}_\omega = G_\omega +\frac{ C_0 + 1 }{V_\omega}  \ge \frac{1}{V_\omega},$$ 
which still satisfies the Green's formula. We also need the following uniform integral estimates of ${\mathcal G}_\omega$ and $\nabla {\mathcal G}_\omega$:
\begin{lemma}[\cite{GPSS, GPS}] \label{lemma 3.2}
\noindent (i). There exist constants $C_1 = C_1(n,p, A,K,\gamma )>0$  and $\varepsilon_0 = \varepsilon_0(n,p)>0$ such that 
\begin{equation}\label{eqn:useful 1}
\sup_{x\in X} \int_X {\mathcal G}_\omega(x,\cdot)^{1+\varepsilon_0} \omega^n \le \frac{C_1}{V_\omega^{\varepsilon_0}}.
\end{equation}

\smallskip

\noindent (ii). For any $\beta>0$ we have 
\begin{equation}\label{eqn:useful 2}
\sup_{x\in X} \int_X \frac{|\nabla {\mathcal G}_\omega(x,\cdot)|^2_\omega}{ {\mathcal G}_\omega (x,\cdot)^{1+\beta}} \omega^n \le \frac{V_\omega^\beta}{\beta}.
\end{equation}
\end{lemma}
With Lemma \ref{lemma 3.2}, we now give the proof of Theorem \ref{thm:Sob}. 

\medskip

\noindent {\em Proof of Theorem \ref{thm:Sob}}: By approximation, we may assume the function $u\in C^1(X)$.  We apply the Green's formula associated to the metric $\omega$ to obtain
$$u(x) = \frac{1}{V_\omega} \int_X u \omega^n + \int_X \langle \nabla  {\mathcal G}_\omega(x,\cdot) , \nabla u\rangle_\omega \omega^n $$
for any $x\in X$. Denote ${\overline{u}} =  \frac{1}{V_\omega} \int_X u \omega^n$ to be the average of $u$ over $(X,\omega^n)$. By the triangle inequality this gives 
\bea
\nonumber
|u(x) - \overline {u}| & \le & \int_X |\nabla  {\mathcal G}_\omega(x,\cdot)|_\omega |\nabla u|_\omega \omega^n\\
& = &\nonumber \int_X \frac{|\nabla  {\mathcal G}_\omega(x,\cdot)|_\omega}{ {\mathcal G}_\omega (x,\cdot)^{(1+\beta)/2}} \cdot {\mathcal G}_\omega (x,\cdot)^{(1+\beta)/2}|\nabla u|_\omega \omega^n\\
\mbox{(by H\"older's ineq.) }&\le \nonumber& \Big(\int_X \frac{|\nabla  {\mathcal G}_\omega(x,\cdot)|^2_\omega}{ {\mathcal G}_\omega (x,\cdot)^{1+\beta}}\omega^n  \Big)^{1/2} \Big(\int_X {\mathcal G}_\omega (x,\cdot)^{1+\beta}|\nabla u|^2_\omega \omega^n \Big)^{1/2}\\
\mbox{ by (\ref{eqn:useful 2})}&\le & \Big(\frac{V_\omega^\beta}{\beta} \Big)^{1/2} \Big(\int_X {\mathcal G}_\omega (x,\cdot)^{1+\beta}|\nabla u|^2_\omega \omega^n \Big)^{1/2}. \label{eqn:2.3}
\eea
Here we fix the constant $\beta = \varepsilon_0/2$, where $\varepsilon_0 >0$ is as in (\ref{eqn:useful 1}), and \begin{equation}\label{eqn:q}
q = \frac{1+\varepsilon_0}{1+\beta} = \frac{1+\varepsilon_0}{1+( \varepsilon_0/2)} >1.
\end{equation}
Taking $(2q)$-th power on both sides of (\ref{eqn:2.3}), we obtain 
$$|u(x) - \overline{u}|^{2q} \le \frac{V_\omega^{\beta q}}{\beta^q}\Big(\int_X {\mathcal G}_\omega (x,y)^{1+\beta}|\nabla u(y)|^2_\omega \omega^n(y) \Big)^{q}.$$ We first integrate this inequality over $X$ against the measure $\omega^n$, and then take the $q$-th root of the resulted integrals, which reads
{\small
\begin{equation}
\Big( \int_X |u(x) - \overline{u}|^{2q} \omega^n  \Big)^{1/q} \le\label{eqn:2.5}  \frac{V_\omega^\beta}{\beta} \Big\{ \int_{x\in X}  \Big(\int_X {\mathcal G}_\omega (x,y)^{1+\beta}|\nabla u(y)|^2_\omega \omega^n(y) \Big)^{q} \omega^n(x) \Big\}^{1/q}. 
\end{equation}
}
Applying the generalized Minkowski inequality to the right-hand side of (\ref{eqn:2.5}), we get
{\small
\bea
& & \Big\{ \int_{x\in X}  \Big(\int_X {\mathcal G}_\omega (x,y)^{1+\beta}|\nabla u(y)|^2_\omega \omega^n(y) \Big)^{q} \omega^n(x) \Big\}^{1/q} \nonumber \\ & \le &  \label{eqn:2.6} \int_{X} |\nabla u(y)|^2_\omega \Big(\int_{x\in X} {\mathcal G}_\omega (x,y)^{(1+\beta)q} \omega^n(x) \Big)^{1/q} \omega^n(y).
\eea
}
By the symmetry of ${\mathcal G}_\omega$, the integral factor in (\ref{eqn:2.6}) satisfies
\bea
\Big(\int_{x\in X} {\mathcal G}_\omega (x,y)^{(1+\beta)q} \omega^n(x) \Big)^{1/q}& \nonumber = & \Big(\int_{ X} {\mathcal G}_\omega (y,x)^{1+\varepsilon_0} \omega^n(x) \Big)^{1/q}\\
&\le \nonumber &\Big( \sup_{y\in X}\int_{ X} {\mathcal G}_\omega (y,x)^{1+\varepsilon_0} \omega^n(x) \Big)^{1/q}\\
\mbox{ by (\ref{eqn:useful 1})}&\le \nonumber & \Big(\frac{C_1}{V_\omega^{\varepsilon_0}} \Big)^{1/q} .
\eea
Combining this with (\ref{eqn:2.5}) and (\ref{eqn:2.6}), we get
\bea 
\noindent
\Big( \int_X |u(x) - \overline{u}|^{2q} \omega^n  \Big)^{1/q} \le  \frac{V_\omega^\beta}{\beta} \frac{C_1^{1/q}}{V_\omega^{\varepsilon_0/q}} \int_X |\nabla u|_\omega^2 \omega^n = \frac{C_2}{V_\omega^{1 - \frac{1}{q}}} \int_X |\nabla u|_\omega^2 \omega^n, \label{eqn:2.7}
\eea
where $C_2 = C_1^{1/q}/\beta>0$ depends only on $n, p, A, K, \gamma$. From (\ref{eqn:2.7}), it is clear that the Sobolev inequality holds, i.e. 
\bea 
\Big( \frac{1}{V_\omega}\int_X |u(x) - \overline{u}|^{2q} \omega^n  \Big)^{1/q} \le   \frac{C_2}{V_\omega} \int_X |\nabla u|_\omega^2 \omega^n. \label{eqn:2.8}
\eea
To see the scale-invariant inequality (\ref{eqn:Sob}), suppose $\omega\in {\mathcal K}(X)$ satisfies $F_\omega\ge \gamma$ and $\| e^{F_\omega}\|_{L^1(\log L)^p(\omega_X^n)}\le K$. Set $\hat \omega = I_\omega^{-1} \omega$. It is elementary to see that $\hat \omega \in {\mathcal W}(n,p, A , K,\gamma)$ with $A = 1$.  So the inequality (\ref{eqn:2.8}) holds for the metric $\hat \omega$ with a constant $C_2'$ which is independent of $A$ as stated, i.e.,
\bea 
\Big( \frac{1}{V_{\hat \omega}}\int_X |u(x) - \overline{u}|^{2q} \hat \omega^n  \Big)^{1/q} \le   \frac{C_2'}{V_{\hat \omega}} \int_X |\nabla u|_{\hat\omega}^2 \hat\omega^n. \label{eqn:2.9N}
\eea Then the inequality (\ref{eqn:Sob}) follows by observing that $|\nabla u|_{\hat\omega}^2 = I_\omega |\nabla u|_\omega^2$ and $\frac{\hat\omega^n}{V_{\hat\omega}} = \frac{\omega^n}{V_\omega}$.
The proof of Theorem \ref{thm:Sob} is completed. \hfill \qedsymbol

\medskip

By combining (\ref{eqn:2.8}) with the Minkowski inequality, we obtain the following form of the Sobolev inequality, which is better suited for the Moser's iteration scheme:
\begin{equation}\label{eqn:2.9}
\Big( \frac{1}{V_\omega}\int_X |u|^{2q} \omega^n  \Big)^{1/q} \le   \frac{C_3}{V_\omega} \int_X ( u^2+  |\nabla u|_\omega^2 )\omega^n,
\end{equation}
where $C_3 =\max( 2 C_2^2,1)$. 

\medskip
In case the K\"ahler metric $\omega$ satisfies the conditions
$$e^{F_\omega}\ge \gamma,\quad \| e^{F_\omega}\|_{L^{1+\varepsilon'}}\le K,\quad [\omega]\cdot [\omega_X]^{n-1}\le A,$$
where $\gamma$ satisfies the condition (\ref{eqn:gamma}), and $\varepsilon'>0$ is a positive constant, we have the following improved estimates on $(X,\omega)$ which may be of independent interest.
\begin{theorem}\label{lemma additional}
Let the assumptions be as above. For any $1\le p'<n$ and $q'\in (p', \frac{np'}{n-p'})$, there exists a constant $C>0$ which depends on $n, K, A,\gamma, \varepsilon'$ and $p',q'$ such that 
\begin{equation}\label{eqn:improvement}
\Big(\frac{1}{V_\omega} \int_X |u - {\overline{u}}|^{2q'} \omega^n\Big)^{1/q'} \le C \Big(\frac{1}{V_\omega} \int_X |\nabla u|_\omega^{2p'} \omega^n \Big)^{1/p'}
\end{equation}
for all $u\in C^1(X)$. In particular, we have the imbedding $W^{1,2p'}(X,\omega)\hookrightarrow L^{2q'}(X,\omega)$.
\end{theorem}
\begin{proof} We know from  \cite{GPS, GPSS} that the Green's function satisfies $$\int_X {\mathcal G}_\omega(x,y)^s \omega^n(y)\le \frac{C(s)}{V_\omega^{s-1}}$$ for all $s\in [1, \frac{n}{n-1})$,
for some constant $C(s)>0$ depending additionally on $s$ besides the given parameters. Given the constants $p',q'$ as in the statement of the theorem, we define $$r = ({1 + \frac{1}{q'} - \frac{1}{p'}})^{-1} \in (1, \frac{n}{n-1}).$$
We choose the constant $\beta\in (0,1)$ small enough in (\ref{eqn:2.3}) such that $(1 + \beta) r < \frac{n}{n-1}$.
We define a linear operator as follows: (below $r^* = \frac{r}{r-1}$)
$${\mathbf {T}}: L^{r^*}\cap L^{1} \to L^\infty\cap L^r, \quad {\mathbf {T}}f(x) =  \int {\mathcal G}_\omega(x,y)^{1+\beta} f(y) \omega^n(y).$$
We can verify that ${\mathbf {T}}$ is of strong types $(1,r)$ and $(r^*, +\infty)$. Indeed, we calculate as follows.
{\small
\bea \nonumber
\| {\mathbf {T}}f\|_{L^r(\omega^n)}  &= &\Big(\int_X |{\mathbf {T}}f(x)|^r \omega^n(x) \Big)^{\frac{1}{r}}\le \int_X \Big(\int_{x\in X} {\mathcal G}_\omega (x,y)^{r(1+\beta)} \omega^n(x) \Big)^{\frac{1}{r}} |f(y)| \omega^n(y)\\
&\le & C_1 V_\omega^{- (1+\beta) + 1/r} \| f\|_{L^1(\omega^n)}, \nonumber
\eea }%
where in the first inequality we have applied the generalized Minkowski inequality. For any $x\in X$, by H\"older's inequality we have
$$|{\mathbf {T}}f(x)| \le   \Big(\int {\mathcal G}_\omega(x,y)^{r(1+\beta)} \omega^n(y)\Big)^{1/r} \| f\|_{L^{r^*}(\omega^n)}\le C_1 V_\omega^{- (1+\beta) + 1/r} \| f\|_{L^{r^*}(\omega^n)}  ,$$
which implies that $\|{\mathbf {T}} f\|_{L^\infty}\le C_1 V_\omega^{- (1+\beta) + 1/r} \| f\|_{L^{r^*}(\omega^n)}$. Then we apply the Riesz-Thorin Interpolation Theorem to conclude that for any $t\in (0,1)$, and the pair of numbers $p_t, q_t$ given by
$$\frac{1}{p_t} = \frac{1-t}{r^*} + {t},\quad \frac{1}{q_t} = \frac{t}{r}, \, \mbox{respectively, }$$
${\mathbf {T}}$ is of strong type $(p_t,q_t)$. More precisely, the following holds,
\begin{equation}\label{eqn:inter}\| {\mathbf {T}}f\|_{L^{q_t}(\omega^n)}\le C_1 V_\omega^{- (1+\beta) + 1/r} \| f\|_{L^{p_t}(\omega^n)}. \end{equation}
In particular, taking $t = {r}/{q'} \in (0,1)$, we see that $p_t = p'$ and $q_t = q'$. To finish the proof we recall that from (\ref{eqn:2.3}), we have the following
$$|u(x) - {\overline{u}}|^2 \le \frac{V_\omega^\beta}{\beta} {\mathbf {T}}(|\nabla u|_\omega^2)(x)$$ for all $x\in X$. We can now apply the inequality (\ref{eqn:inter}) with $f:=|\nabla u|^2_\omega$ to conclude that 
\begin{equation} \label{eqn:inter1}
\| |u - {\overline{u}}|^2\|_{L^{q'}(\omega^n)}\le \frac{C_1}{\beta}  V_\omega^{\frac{1}{q'} - \frac{1}{p'}} \| |\nabla u|^2_\omega\|_{L^{p'}(\omega^n)}.
\end{equation}
Clearly here the $\beta$ can be chosen depending only on $p'$ and $q'$. The desired inequality (\ref{eqn:improvement}) follows easily from some rearrangement of (\ref{eqn:inter1}).
\end{proof}

\medskip

We next discuss a few examples when the Sobolev inequalities (\ref{eqn:2.8}) and (\ref{eqn:2.9}) hold. 

\medskip

\noindent{\bf Example 4.1} Let $\pi: (X,\omega_X)\to (Y,\omega_Y)$ be a surjective holomorphic map between compact K\"ahler manifolds $X$ and $Y$. Assume that ${\mathrm{dim}}(X) = n \ge {\mathrm {dim}}(Y) = m$. Consider the continuous family of K\"ahler metrics $\omega = \omega_t: = t \omega_X + \pi^*\omega_Y$ for $t\in (0,1]$. It is easy to see that $V_\omega \sim t^{n-m}$, and 
$$e^{F_\omega} = \frac{1}{V_\omega} \frac{\omega^n}{\omega_X^n}\sim \frac{1}{V_\omega} \sum_{k=0}^{m} t^{n-k} \frac{(\pi^*\omega_Y)^k \wedge \omega_X^{n-k}}{\omega_X^n}\le C$$
for some constant $C>0$ depending only on $n,m,\omega_X, \pi^*\omega_Y$ but independent of $t$. The intersection number $[\omega]\cdot[\omega_X]^{n-1}$ is clearly bounded. The expected lower bound $\gamma$ of $e^{F_\omega}$ can be taken as suitable multiple of $\frac{(\pi^*\omega_Y)^m \wedge \omega_X^{n-m}}{\omega_X^n}$, which is equal to the squared Jacobian of the holomorphic map $\pi: (X,\omega_X) \to (Y, \omega_Y)$. The Jacobian of $\pi$ consists of  nontrivial (local) holomorphic functions hence its vanishing locus $\{\gamma = 0\}$ is a proper analytic subvariety in $X$. Hence the assumption (\ref{eqn:gamma}) on $\gamma$ is also satisfied. Given these, we find that the Sobolev inequalities (\ref{eqn:2.8}) and (\ref{eqn:2.9}) hold for this continuous family of K\"ahler metrics, with the constants depending on $n, \pi, \omega_X, \omega_Y$, but not $t\in (0,1]$.

\smallskip

We remark that when $\pi:X\to Y$ is the blown-up of a {\em smooth} center in $Y$, the uniform Sobolev inequality (\ref{eqn:2.9}) is known by the work of Bando-Siu in \cite{BS} with $q = \frac{n}{n-1}$. However, for the case of general surjective holomorphic maps, as far as we know, the uniform Sobolev inequalities (\ref{eqn:2.8}) and (\ref{eqn:2.9}) are new.

\medskip

\noindent {\bf Example 4.2} Consider the finite-time K\"ahler-Ricci flow $\omega = \omega_t$: $\frac{\partial \omega}{\partial t} = -\ric(\omega)$ with $t\in [0,T)$ and $T<\infty$. If the limit current $[\omega_T]$ is big, in the sense that $[\omega_T]^n>0$, then by \cite{GPSS}, the metrics $\omega_t$ lie in ${\mathcal W}(n,p, A, K, \gamma)$ for suitable parameters depending only on $n, \omega_t|_{t=0}$ and $[\omega_T]$. Hence the uniform Sobolev inequalities (\ref{eqn:2.8}) and (\ref{eqn:2.9}) hold for these metrics $\omega_t$ with $t\in [0, T)$.

\medskip

\noindent {\bf Example 4.3} Let $X$ be a compact K\"ahler manifold with nef canonical line bundle $K_X$ and nonnegative Kodaira dimension. Suppose $\omega = \omega_t$ satisfies the normalized K\"ahler-Ricci flow $\tfrac{\partial \omega}{\partial t} = -\ric(\omega) - \omega$. Again by \cite{GPSS}, the metrics $\omega_t$ lie in ${\mathcal W}(n,p, A, K, \gamma)$ for suitable parameters depending only on $n, \omega_t|_{t=0}$ and $K_X$. Hence the uniform Sobolev inequalities (\ref{eqn:2.8}) and (\ref{eqn:2.9}) hold for these metrics $\omega_t$ with $t\in [0, \infty)$.

\newcommand{\nref}[1]{(\ref{#1})}

\section{Upper bound of the heat kernel}\label{section 4}
\setcounter{equation}{0}
Let $\omega\in {\mathcal W}(n,p, A, K, \gamma)$ be a fixed K\"ahler metric, and $H(x,y,t)$ be the heat kernel of $\Delta_\omega$ as defined in the equation (\ref{eqn:heat}). 

\subsection{Off diagonal estimates} Given the Sobolev inequality \nref{eqn:2.9}, we will make use of the ideas of Davies \cite{D} (see also \cite{Z}) and Saloff-Coste \cite{Sc} to prove a Gaussian-type upper bound of the hear kernel $H(x,y,t)$.

Let $u_0\in C^\infty(X)$ be an arbitrary positive smooth function on $X$. For a given $\alpha\in {\mathbb R}$ and an arbitrary Lipschitz function $\phi$ with $\| \phi\|_{L^\infty(\omega)}\le 1$, we consider the function $u(x,t)$ defined by
\begin{equation}\label{eqn:linearheat}u(x,t) = e^{-\alpha \phi(x)} \int_X H(x,y,t) e^{\alpha \phi(y)} u_0(y) \omega^n (y),\end{equation} which is positive and satisfies the equation
\begin{equation}\label{eqn:new3.1}
\frac{\partial u}{\partial t} = e^{-\alpha \phi} \Delta_\omega (u e^{\alpha \phi}),\quad u(x,0) = u_0(x)
\end{equation} 
for all  $x\in X$.
Fix any time $T>0$ and a constant $\beta \in (0,1)$ (for our purpose, it suffices to take $\beta = 1/2$). We define an increasing, piecewise smooth, and continuous  function $r(t)$ of $t\in (0,T)$ by $$r= r(t) =\left\{\begin{array}{ll} 
\frac{2^\beta - 1} { (T/2)^\beta} t^\beta + 1, & \quad \mbox{if } t\in [0,T/2),\\
\frac{T^\beta}{(T-t)^\beta}, & \quad \mbox{if }t\in [T/2, T),
\end{array}
\right.$$
which satisfies $r(0) = 1$ and $r(T^-) = \infty$. Denote $\| u\|_r$ to be the $L^r(X,\omega^n)$-norm of $u$. By straightforward calculations, we have 
\bea
\label{eqn:new3.2}
\frac{\partial}{\partial t} \log \| u\|_r &= &  -\frac{r'}{r}\log \| u\|_r + \frac{1}{\| u\|_r^r} \int_X u^{r-1} e^{-\alpha \phi}\Delta_\omega (u e^{\alpha \phi}) \omega^n
 \\
&&\label{eqn:new3.2N}  \quad + \frac{r'}{r} \frac{1}{\| u\|_r ^r} \int_X u^r (\log u)\,\omega^n.
\eea
By integration by parts, the last term in \nref{eqn:new3.2} is equal to 
\bea
&& \nonumber \frac{1}{\| u\|_r^r} \int_X u^{r-1} e^{-\alpha \phi}\Delta_\omega (u e^{\alpha \phi}) \omega^n \\
 & = &\nonumber  - \frac{1}{\| u\|_r^r} \int_X\langle \nabla (u^{r-1} e^{-\alpha \phi}), \nabla (u e^{\alpha \phi})\rangle_\omega \omega^n\\
 & = &\nonumber \frac{1}{\| u\|_r^r} \int_X -(r-1) u^{r-2} |\nabla u|^2 - \alpha (r-2) u^{r-1} \langle \nabla u, \nabla \phi\rangle_\omega + \alpha^2 u^r |\nabla \phi|^2\\
  & \le &\nonumber \frac{1}{\| u\|_r^r} \int_X -(r-1)(1-\delta) u^{r-2} |\nabla u|^2 + \frac{\alpha^2 (r-2)^2}{4(r-1) \delta} u^r |\nabla \phi|^2+ \alpha^2 u^r |\nabla \phi|^2\\
    & \le &\label{eqn:new3.2O} \frac{1}{\| u\|_r^r} \int_X -\frac{4 (r-1)(1-\delta) }{r^2} |\nabla u^{r/2}|^2 \omega^n  + \frac{\alpha^2 (r-2)^2}{4(r-1) \delta} + \alpha^2,
\eea 
where in the first inequality we apply the Cauchy-Schwarz inequality with $\delta\in (0,1)$ a constant to be chosen later, and in the last line we use the condition that $\| \nabla\phi\|_\infty\le 1$.

For notational convenience we set $v = \frac{u^{r/2}}{\| u^{r/2}\|_2}$, which satisfies $\| v\|_2 = 1$. Note that $\| u^{r/2}\|_2^2 = \| u\|_r^r.$ The  term in \nref{eqn:new3.2N} becomes
$$\frac{r'}{r} \frac{1}{\| u\|_r ^r} \int_X u^r \log u\,\omega^n = \frac{r'}{r^2}\int_X v^2 \log v^2\,\omega^n + \frac{r'}{r} \log \| u\|_r.$$ Therefore, we have from \nref{eqn:new3.2} and \nref{eqn:new3.2O} that
\bea
\label{eqn:new3.3}
\frac{\partial}{\partial t} \log \| u\|_r &\le & - \frac{4(r-1)(1-\delta)}{r^2} \int_X |\nabla v|_\omega^2 \omega^n +  \frac{r'}{r^2}\int_X v^2 \log v^2\,\omega^n\\
&& \quad \nonumber  + \frac{\alpha^2 (r-2)^2}{4(r-1) \delta} + \alpha^2.
\eea
We next apply the Sobolev inequality \nref{eqn:2.9} to deal with the last term in \nref{eqn:new3.3}. 
\bea
\int_X v^2 \log v^2 \omega^n & = & \frac{1}{q-1} \int_X v^2 \log v^{2(q-1)} \omega^n \nonumber \\
\mbox{(by Jensen's inequality)}&\le & \frac{1}{q-1} \log \frac{1}{V_\omega} \int_X v^{2q} \omega^n + \frac{\log V_\omega}{q-1} \nonumber \\
&\le & \frac{q}{q-1} \log \Big(\frac{C}{V_\omega} \int_X (v^2 + |\nabla v|^2) \omega^n \Big)+ \frac{\log V_\omega}{q-1}  \nonumber\\
&= & \frac{q}{q-1} \log \Big( \int_X (v^2 + |\nabla v|^2) \omega^n \Big)- \log V_\omega +  \frac{q}{q-1} \log C  . \nonumber
 \eea
 Combining this with \nref{eqn:new3.3}, we obtain
\bea\nonumber\frac{\partial}{\partial t} \log \| u\|_r & \le & - \frac{4(r-1)(1-\delta)}{r^2} \int_X |\nabla v|_\omega^2 \omega^n +  \frac{r'}{r^2}\frac{q}{q-1} \log \Big(1 + \int_X |\nabla v|^2\omega^n \Big)  \\
&& \label{eqn:middle key}\quad +  \frac{r'}{r^2}\frac{q}{q-1} \log {C} - \frac{r'}{r^2} \log V_\omega  + \frac{\alpha^2 (r-2)^2}{4(r-1) \delta} + \alpha^2.
\eea
{\bf Claim:} The following inequality holds: \bea\nonumber &&  - \frac{4(r-1)(1-\delta)}{r^2} \int_X |\nabla v|_\omega^2 \omega^n +  \frac{r'}{r^2}\frac{q}{q-1} \log \Big(1 + \int_X |\nabla v|^2\omega^n \Big)\\
&\le &\nonumber \frac{r'}{r^2} \frac{q}{q-1} \log \Big( \frac{r' q}{4(r-1) (1-\delta)(q-1)} \Big) +  \frac{4(r-1)(1-\delta)}{r^2} - \frac{r'}{r^2} \frac{q}{q-1}. \eea
This claim follows easily from the calculus inequality that the concave function $$f(x): =  - \frac{4(r-1)(1-\delta)}{r^2} x +  \frac{r'}{r^2}\frac{q}{q-1} \log (1 + x)$$ achieves the maximum at $x:= \frac{r' q}{4(r-1)(1-\delta) (q-1)} - 1$.

\smallskip

This claim together with the inequality \nref{eqn:middle key} implies that
\bea\label{eqn:new3.4}\frac{\partial}{\partial t} \log \| u\|_r & \le &  \frac{r'}{r^2} \frac{q}{q-1} \log\Big( \frac{r'}{r-1} \Big) + \frac{4(r-1)(1-\delta)}{r^2} + \frac{r'}{r^2} \Psi\\
&&\nonumber \quad  + \frac{\alpha^2 (r-2)^2}{4(r-1) \delta} + \alpha^2,
\eea
where the constant $\Psi$ is given by $$\Psi: = \frac{q}{q-1} \log \frac{q}{4(q-1)(1-\delta)} - \frac{q}{q-1} + \frac{q}{q-1} \log {C} - \log V_\omega.$$ 
We need the following elementary identities:
\bea
\int_0^T \frac{r'}{r^2} dt = -\frac{1}{r(t)}\Big|_{t=0}^T = 1.\nonumber
\eea
\bea
&& \nonumber \int_0^T \frac{r'}{r^2}  \log\Big( \frac{r'}{r-1} \Big) dt\\
& = & \nonumber \int_0^{T/2} \frac{r'}{r^2} \log \frac{\beta}{t} dt + \int_{T/2}^T \frac{r'}{r^2}  \log \Big( \frac{\beta T^\beta}{(T-t) (T^\beta - (T-t)^\beta)}   \Big)dt\\
& = &  \nonumber\int_0^{1/2} \frac{r_s'}{r^2} \log \frac{\beta}{Ts} ds + \int_{1/2}^1 \frac{r_s'}{r^2}  \log \Big( \frac{\beta}{T(1-s) (1 - (1-s)^\beta)}   \Big)ds\\
& = & \nonumber \log \frac{1}{T} \int_0^1 \frac{r_s'}{r^2} ds + \int_0^{1/2} \frac{\beta (4^\beta - 2^\beta)s^{\beta - 1}}{( (4^\beta - 2^\beta) s^\beta + 1    )^2} \log \frac{\beta}{s} ds\\
&& \nonumber\quad +\int_{1/2} ^1 \frac{\beta}{(1-s)^{1-\beta}}  \log \Big( \frac{\beta}{(1-s) (1 - (1-s)^\beta)}   \Big)ds\\
& = & \nonumber -\log T + A_\beta,
\eea
where we denote $A_\beta\in {\mathbb R}$ to be the sum of the two integrals which are convergent by the choice of $\beta \in (0,1)$, and we note that $A_\beta \approx 2.455$ if $\beta = 1/2$.
\bea
&&\int_0^T\frac{(r-2)^2}{r-1} dt \nonumber \\
& = &\nonumber \int_0^{T/2} \frac{[ (4^\beta - 2^\beta) (t/T)^\beta  -1   ]^2}{ (4^\beta - 2^\beta) (t/T)^\beta    }dt + \int_{T/2}^T \frac{[T^\beta - 2 (T-t)^\beta]^2}{ (T-t)^\beta [T^\beta - (T-t)^\beta]   }dt\\
& = & \nonumber T \Big( \int_0^{1/2} \frac{[ (4^\beta - 2^\beta) s^\beta  -1   ]^2}{ (4^\beta - 2^\beta) s^\beta    }ds + \int_{1/2}^1 \frac{[1 - 2 (1-s)^\beta]^2}{ (1-s)^\beta [1 - (1-s)^\beta]   }ds\Big)\\
&=&T B_\beta \nonumber,
\eea
where $B_\beta>0$ is the sum of the two integrals. (If $\beta = 1/2$, then $B_\beta \approx 2.008$). 
{\small $$ \int_0^T \frac{r-1}{r^2} dt = T \Big( \int_0^{1/2} \frac{ (4^\beta - 2^\beta) s^\beta   }{[(4^\beta - 2^\beta) s^\beta  + 1]^2} ds + \int_{1/2}^1 (1-(1-s)^\beta) (1-s)^\beta ds\Big) = T C_\beta,$$ }
where $C_\beta \approx 0.192$ if $\beta = 1/2$.

%
Together with these identities, integrating \nref{eqn:new3.4} over $t\in (0,T)$ yields that
$$
\log \frac{\| u_T\|_\infty}{\| u_0\|_1}\le \frac{q}{q-1} (A_\beta- \log T) + 4(1-\delta) C_\beta T +\Psi + \frac{B_\beta}{4\delta} \alpha^2 T + \alpha^2 T
$$
where $u_T (x) = u(x,T)$. This implies that 
\bea
\| u_T\|_\infty & \le &  T^{-\frac{q}{q-1}} \exp\Big( 4(1-\delta) C_\beta T + \frac{A_\beta q}{q-1} + \Psi  +\frac{B_\beta}{4\delta} \alpha^2 T + \alpha^2 T\Big) \| u_0\|_1\nonumber\\
&\le&  C_0  T^{-\frac{q}{q-1}} V_\omega^{-1}  e^{4(1-\delta) C_\beta T+\frac{B_\beta}{4\delta} \alpha^2 T + \alpha^2 T}   \| u_0\|_1,\label{eqn:new3.6}
\eea
where the constant $C_0>0$ depends on $\beta$ and $\delta\in (0,1)$.
Since \nref{eqn:new3.6} holds for {\em any} positive function $u_0$, by duality and the representation formula \nref{eqn:linearheat}, the inequality \nref{eqn:new3.6} implies the pointwise estimate for $H(x,y,t)$, for any two points $x,y\in X$, that
\begin{equation}\label{eqn:Heat1}H(x,y,T) \le C_0  T^{-\frac{q}{q-1}} V_\omega^{-1}  e^{4(1-\delta) C_\beta T+\frac{B_\beta}{4\delta} \alpha^2 T + \alpha^2 T} e^{\alpha (\phi(x) - \phi(y))}.
\end{equation}
For two fixed points $x,y\in X$, if we take $\alpha = -\frac{\phi(x) - \phi(y)}{(2 + B_\beta/2\delta)T}$, then from \nref{eqn:Heat1} we obtain
\begin{equation}\label{eqn:heat2}
H(x,y,T) \le C_0  T^{-\frac{q}{q-1}} V_\omega^{-1}  e^{4(1-\delta) C_\beta T} e^{ -\frac{(\phi(x) - \phi(y))^2}{ (4+B_\beta/\delta)T   }    }.
\end{equation}
Since \nref{eqn:heat2} holds for {\em any} smooth function $\phi$ with $\|\nabla\phi\|_\infty\le 1$, we can take a sequence of such functions which converge uniformly to $z\mapsto  d_\omega(x,z)$, where $d_\omega(x,z)$ is the function of distance of $z\in X$ to the fixed point $x$. Therefore, from \nref{eqn:heat2}, we conclude that for any $T>0$, it holds that 
\begin{equation}\label{eqn:heat3}
H(x,y,T) \le C_0  T^{-\frac{q}{q-1}} V_\omega^{-1}  e^{4(1-\delta) C_\beta T} e^{ -\frac{d_\omega(x,y)^2}{ (4+B_\beta/\delta)T   }    }.
\end{equation}
We next claim that the factor $e^{4(1-\delta) C_\beta T}$ on the right-hand side of \nref{eqn:heat3} can be replaced by a uniform constant independent of $T>0$. 
We have shown that \nref{eqn:heat3} holds for {\em any} metric $\omega\in {\mathcal W}(n, p, A, K, \gamma)$. If we take $A=1$, the constant $C_0$ in \nref{eqn:heat3} can be chosen to be independent of the parameter $A$, which we will assume next.  

Recall that we  denote $I_\omega  = [\omega]\cdot[\omega_X]^{n-1}>0$ to be the intersection number of the K\"ahler classes $[\omega]$ and $[\omega_X]$. Given a metric $\omega\in {\mathcal W}(n, p, \infty, K, \gamma)$, we consider the rescaled metric $\tilde \omega = I_{\omega}^{-1} \omega$, which meets the equations $F_{\tilde \omega} = F_\omega$ and $I_{\tilde \omega} = I_{\omega}^{-1} [\omega]\cdot [\omega_X]^{n-1} = 1$. Hence we have $\tilde \omega \in {\mathcal W}(n, p, 1, K, \gamma)$. Applying \nref{eqn:heat3} to the metric $\tilde\omega$ we get
\begin{equation}\label{eqn:heat4}
H_{\tilde \omega}(x,y,T) \le C_0  T^{-\frac{q}{q-1}} V_{\tilde \omega}^{-1}  e^{4(1-\delta) C_\beta T} e^{ -\frac{d_{\tilde \omega}(x,y)^2}{ (4+B_\beta/\delta)T   }    },
\end{equation}
and we stress that the constant $C_0$ depends only on $n,p, K$, $\gamma$ and $\delta\in (0,1)$. It is a standard fact that the heat kernels of $\omega$ and $\tilde \omega$ are related by 
\begin{equation}\label{eqn:heat5}
H_{\omega}(x,y,T) = I_\omega^{-n} H_{\tilde \omega}(x,y, T/I_\omega).
\end{equation}
Combining \nref{eqn:heat4} and \nref{eqn:heat5} gives
\bea\nonumber
H_{ \omega}(x,y,T)& \le & C_0 I_\omega^{-n}  (T/I_\omega)^{-\frac{q}{q-1}} V_{\tilde \omega}^{-1}  e^{4(1-\delta) C_\beta T/I_\omega} e^{ -\frac{I_\omega d_{\tilde \omega}(x,y)^2}{ (4+B_\beta/\delta)T   }    }\\
& = & C_0  I_\omega^{ \frac{q}{q-1}} T^{-\frac{q}{q-1}} V_{ \omega}^{-1}  e^{4(1-\delta) C_\beta T/I_\omega} e^{ -\frac{d_{\omega}(x,y)^2}{ (4+B_\beta/\delta)T   }    } .\label{eqn:heat6}
\eea
The inequality \nref{eqn:heat6} yields an upper bound of $H_\omega(x,y,T)$ when $T\le I_\omega$ by choosing $\delta = 1/2$, i.e. 
\begin{equation}\label{eqn:heat7} 
H_\omega(x,y,T) \le C_0  I_\omega^{ \frac{q}{q-1}} T^{-\frac{q}{q-1}} V_{ \omega}^{-1}  e^{2 C_\beta } e^{ -\frac{d_{\omega}(x,y)^2}{ (4+2 B_\beta)T   }    } .
\end{equation}
If $T\ge I_\omega$, we choose $1-\delta = \frac{I_\omega}{4T}\le 1/4$, i.e. $1>\delta\ge 3/4 $. Note that the constant $C_0\sim e^{\Psi}\sim (1-\delta)^{-\frac{q}{q-1}}$, hence, the inequality \nref{eqn:heat6} implies that if $T\ge I_\omega$
\begin{equation}\label{eqn:heat8} 
H_\omega(x,y,T) \le C_1  V_{ \omega}^{-1}  e^{ C_\beta } e^{ -\frac{d_{\omega}(x,y)^2}{ (4+2 B_\beta)T   }    } ,
\end{equation}
for some constant $C_1>0$ which depends only on $n, p, K$ and $\gamma$. The upper bound of $H_\omega(x,y,t)$ in \nref{eqn:heat kernel} follows from  \nref{eqn:heat7} and \nref{eqn:heat8}. \hfill \qedsymbol

\subsection{On diagonal estimates} In this subsection, we will derive a more precise upper bound of the heat kernel on the diagonal.  We follow  the classical arguments of Cheng-Li in \cite{CL}. We define $$\tilde H(x,y,t) = H(x,y,t) - \frac{1}{V_\omega},$$
which satisfies $\int_X \tilde H(x,y,t) \omega^n(y) = 0$ for any $t>0$. 

\begin{lemma}\label{lemma3.1}
$\tilde H$ satisfies the following semi-group property
\begin{equation}\label{eqn:3.1}
\tilde H(x,z, t) = \int_X \tilde H(x,y, s) \tilde H(y, z, t-s) \omega^n(y)
\end{equation}
 for all $s\in (0,t)$.
\end{lemma}
This lemma is well-known, but for completeness, we provide a sketched proof.

\smallskip

\noindent{\em Proof of Lemma \ref{lemma3.1}}.  Denote $f(s) = \int_X \tilde H(x,y, s) \tilde H(y, z, t-s) \omega^n(y)$. Taking derivative on $f(s)$, we get
\bea
f'(s) &=\nonumber & \int_X  [\frac{\partial}{\partial s}\tilde H(x,y, s) \tilde H(y, z, t-s) - \tilde  H(x,y, s)  \frac{\partial}{\partial t}\tilde H(y, z, t-s)] \omega^n(y)\\
&=\nonumber & \int_X  [\Delta_{\omega,y}\tilde H(x,y, s) \tilde H(y, z, t-s) - \tilde H(x,y, s)  \Delta_{\omega,y}\tilde H(y, z, t-s)] \omega^n(y)\\
& = & 0,\nonumber
\eea
by integration by parts. Hence $f(s)$ is independent of $s\in (0,t)$. It is not hard to see $f(s)\to \tilde H(x,z,t)$ as $s\to 0^+$. Hence the lemma follows.\hfill \qedsymbol

\medskip
In particular, setting $s = t/2$ and $x = z$ in Lemma \ref{lemma3.1} we get the equation
$$\tilde H(x,x,t) = \int_X \tilde H(x,y,t/2) ^2 \omega^n(y).$$
Taking derivative with respect to $t$, this gives 
\bea
\nonumber
\frac{\partial }{\partial t} \tilde H(x,x,t) & = & \int_X \tilde H(x,y,t/2)\frac{\partial}{\partial s}\Big|_{s=\frac{t}{2}} \tilde H(x,y, s) \omega^n(y)\\
& = & \int_X \tilde H(x,y,t/2)\Delta_\omega \tilde H(x,y, t/2) \omega^n(y)\nonumber\\
& = & \label{eqn:3.2}-  \int_X |\nabla_y \tilde H(x,y,t/2)|^2_\omega \omega^n(y). 
\eea
For a fixed $x\in X$ and $t>0$, we consider the function $u(y): = \tilde H(x, y, t/2)$. By the definition of $\tilde H(x,y,t/2)$, we have $\overline{u} = \frac{1}{V_\omega}\int_X u \omega^n = 0$, hence the Sobolev inequality \nref{eqn:Sob} implies that 
\begin{equation}\label{eqn:3.3}\Big( \frac{1}{V_\omega} \int_X |u|^{2q} \Big)^{1/q} \le \frac{C}{V_\omega} \int_X |\nabla u|_\omega^2 \omega^n.\end{equation}
By H\"older's inequality, we have
\bea
\frac{1}{V_\omega} \int_X u^2 \omega^n &\le &\Big(\frac{1}{V_\omega} \int_X |u|^{2q}\omega^n \Big)^{\frac{1}{2q - 1}} \Big( \frac{1}{V_\omega} \int_X |u| \omega^n \Big)^{\frac{2q - 2}{2q - 1}}\nonumber \\
\mbox{by \nref{eqn:3.3}}&\le & \Big( \frac{C}{V_\omega} \int_X |\nabla u|_\omega^2 \omega^n \Big)^{\frac{q}{2q - 1}} \Big(\frac{2}{V_\omega} \Big)^{\frac{2q - 2}{2q - 1}}\nonumber. 
\eea
Thus it holds that
\bea
0<\tilde H(x,x,t) = \int_X u^2 \omega^n \le \nonumber {C}{V_\omega^{-\frac{q-1}{2q-1}}} \Big(\int_X |\nabla u|_\omega^2 \omega^n \Big)^{\frac{q}{2q - 1}}.
\eea
Combining this with \nref{eqn:3.2}, we obtain the differential inequality:
$$\frac{\partial }{\partial t} \tilde H(x,x,t) + C V_\omega^{\frac{q-1}{q} } \tilde H(x,x,t)^{\frac{2q - 1}{q}} \le 0, $$
which implies that for a possible different constant $C>0$, $\frac{\partial }{\partial t} \tilde H(x,x,t) ^{-\frac{q-1}{q}}\ge C (1- \frac {1}{q}) V_\omega^{1-\frac{1}{q}}$. Note that $\tilde H(x,x,t)\to +\infty$ as $t\to 0^+$. Integrating the inequality above over $t\in (0,T]$ for any arbitrary $T>0$ yields that
$$\tilde H(x,x, T) \le \frac{C}{V_\omega} \frac{1}{T^{\frac{q}{q-1}}},$$ which combined with $\tilde H = H - \frac{1}{V_\omega}$ implies the upper bound of $H(x,x,t)$ (we change $T$ to $t$),
$$H(x,x,t) \le \frac{1}{V_\omega} + \frac{C}{V_\omega} \frac{1}{t^{\frac{q}{q-1}}}.$$ 
A simple scaling argument using \nref{eqn:heat5} shows that the constant $C$ can be chosen to be independent of $A>0$ in the following form:
$$H(x,x,t) \le \frac{1}{V_\omega} + C\frac{I_\omega^{\frac{q}{q-1}}}{V_\omega} \frac{1}{t^{\frac{q}{q-1}}}.$$ 

\smallskip

To see \nref{eqn:exp heat}, we observe that from the expansion formula \nref{eqn:heat expansion} we have  when $t\ge I_\omega$,
\bea\nonumber
|\tilde H(x,y,t)|& \le  & \sum_{k=1}^\infty  e^{-\lambda_k t} |\phi_k(x)| |\phi_k(y)|\le e^{-\lambda_1 t/2} \sum_{k=1}^\infty e^{-\lambda _k t /2} |\phi_k(x)| |\phi_k(y)|\\
&\le & e^{-\lambda_1 t/2} \tilde H(x,x,t)^{1/2}  \tilde H(y,y,t)^{1/2} \nonumber\\
&\le &\nonumber   \frac{C}{V_\omega} e^{-\lambda_1 t/2} I_\omega^{\frac{q}{q-1}} t ^{-\frac{q}{q-1}} \le \frac{C}{V_\omega} e^{-c I^{-1}_\omega t},
\eea
where we have used the estimate $\lambda_1 \ge 2c  I_\omega^{-1}$ for some constant $c>0$. In particular, the heat kernel $H(x,y, t)$ converges to the constant function $1/V_\omega$ exponentially as $t\to \infty$.
\hfill \qedsymbol

\medskip

\noindent {\em Proof of Corollary \ref{cor:CL}.} We turn to the applications of the upper bound of heat kernel \nref{eqn:CL}. Let $0=\lambda_0 < \lambda_1 \le \lambda_2\le \ldots$ (counting multiplicity) be the increasing sequence of eigenvalues of the Laplacian $-\Delta_\omega$. We choose eigenfunctions $\phi_k$ with eigenvalue $\lambda_k$ such that $\{\phi_k\}$ are orthonormal in $L^2(X,\omega^n)$. It is a classical fact that $\{\phi_k\}$ is an orthonormal basis of $L^2(X,\omega^n)$, $\phi_0 = 1/\sqrt{V_\omega}$, and the  heat kernel $H(x,y,t)$ can be expanded as
\begin{equation}\label{eqn:heat expansion}H(x,y,t) = \sum_{j=0}^\infty e^{-\lambda_j t} \phi_j(x) \phi_j(y) = \frac{1}{V_\omega} + \sum_{j=1}^\infty e^{-\lambda_j t} \phi_j(x) \phi_j(y).\end{equation}
The upper bound \nref{eqn:CL} implies that
\begin{equation}\label{eqn:3.4}
\sum_{j=1}^\infty e^{-\lambda_j t} \phi_j(x)^2 \le \frac{CI_\omega^{q/(q-1)}}{V_\omega} \frac{1}{t^{\frac{q}{q-1}}}.
\end{equation}
Integrating \nref{eqn:3.4} over $X$ gives 
\begin{equation}\label{eqn:3.5}
\sum_{j=1}^\infty e^{-\lambda_j t}  \le \frac{C I^{q/(q-1)}_\omega}{t^{\frac{q}{q-1}}}.
\end{equation}
For any given $k\in {\mathbb N}$, setting $t = 1/\lambda_k$, the equation \nref{eqn:3.5} implies 
$$C I_\omega^{q/(q-1)} \lambda_k^{\frac{q}{q-1}} \ge \sum_{j=1}^\infty e^{-\lambda_j /\lambda_k}\ge \sum_{j=1}^k e^{-\lambda_j /\lambda_k}\ge e^{-1} k.$$
Hence we conclude that there is a positive constant $c = 1/ (Ce)^{\frac{q-1}{q}}>0$ such that 
$$\lambda_k \ge  c I_\omega^{-1}\, k^{\frac{q-1}{q}}.$$
This finishes the proof of Corollary \ref{cor:CL}. \hfill \qedsymbol

\medskip
We also have the following estimate on the time-derivative of the heat kernel, which will be useful in section \ref{section heat}. In the following we assume $\omega\in {\mathcal W}(n,p, A, K, \gamma)$ as usual, and $I_\omega\le A$ by definition. 
\begin{lemma}\label{lemma heat} The following inequalities hold: for any $t>0$
\begin{equation}\label{eqn:heat dot}
\int_X \big(\dot H(x,y,t)\big)^2\omega^n(y) \le t^{-2} \tilde H(x,x,t) \le \frac{C}{V_\omega} t^{-2 - \frac{q}{q-1}}, 
\end{equation}
and 
\begin{equation}\label{eqn:heat dot1}
\sup_{x,y\in X}|\dot H(x,y,t)| \le \frac{C}{V_\omega} t^{-1 - \frac{q}{q-1}}, 
\end{equation}
where we denote $\dot H(x,y,t) = \frac{\partial }{\partial t} H(x,y,t)$.
\end{lemma}
\begin{proof}
By the expansion formula \nref{eqn:heat expansion} of the heat kernel $H(x,y,t)$, we have 
$$\dot H(x,y,t) = - \sum_{j=1}^\infty e^{-\lambda_j t} \lambda_j \phi_j(x) \phi_j(y). $$
Then we integrate $\dot H(x,y,t)^2$ and obtain
{\small
$$\int_X \big(\dot H(x,y,t)\big)^2\omega^n(y) = \sum_{j=1}^\infty e^{-2\lambda_j t} \lambda_j^2 \phi_j^2(x)\le \frac{1}{t^2} \sum_{j=1}^\infty e^{-\lambda_j t} \phi_j^2(x)=  \frac{1}{t^2} \tilde H(x,x,t) \le \frac{C}{V_\omega} t^{-2 - \frac{q}{q-1}}, $$%
}%
where we have used the elementary inequality $0< e^{- 2\lambda_j t} \lambda_j^2 = t^{-2} e^{-\lambda_j t} \cdot (\lambda_jt)^2 e^{-\lambda_j t} < t^{-2} e^{-\lambda_j t}$ since the maximum of the function $x^2 e^{-x}|_{x>0}$ is $(2/e)^2$ which is $<1$. 

To see \nref{eqn:heat dot1}, we observe that the function $\dot H(x,y,t)$ as a function of $(y,t)\in X\times (0,\infty)$ satisfies the heat equation, hence by the representation formula for heat equations, we have for a fixed $t_0>0$ and any $t>t_0$
\bea\nonumber
|\dot H(x,y,t)| &= &|\int_X H(z,y, t-t_0) \dot H(x,z,t_0) \omega^n(z)| = |\int_X \tilde H(z,y, t-t_0) \dot H(x,z,t_0) \omega^n(z)|\\
&\le & \Big(\int_X \tilde H(z,y, t-t_0)^2 \omega^n(z) \Big)^{1/2} \Big(\int_X  \dot H(x,z,t_0)^2 \omega^n(z) \Big)^{1/2}\label{eqn:30 1}\\
&= &  \tilde H(y,y, 2t-2t_0)^{1/2} \Big(\int_X  \dot H(x,z,t_0)^2 \omega^n(z) \Big)^{1/2}, \label{eqn:30 2}
\eea
where \nref{eqn:30 1} follows from H\"older's inequality and in \nref{eqn:30 2} we have applied the semi-group property  \nref{eqn:3.1} of $\tilde H$. 
For a fixed $t>0$ we take $t_0 = t/2$. We then get \nref{eqn:heat dot1} by applying the estimates \nref{eqn:heat dot} and \nref{eqn:CL}. Alternatively, the inequality \nref{eqn:heat dot1} also follows from the expansion of $\dot H(x,y,t)$ and Cauchy-Schwarz inequality.
\end{proof}
\medskip
We remark that   $L^\infty$ estimates of the eigenfunctions $\phi_k$ also follow from \nref{eqn:3.4}. In fact,
 setting $t = 1/\lambda_k$ in  \nref{eqn:3.4} yields that
$$\phi_k(x)^2 \le \sum_{j=1}^{\infty}e^{-\lambda_j/\lambda_k} \phi_j^2(x) \le \frac{C I^{q/(q-1)}_\omega}{V_\omega} \lambda_k^{\frac{q}{q-1}}.$$
Choose $x\in X$ such that $\phi_k(x)^2 = \max_X \phi_k^2$. We then get 
$$\max_X |\phi_k|^2 \le \frac{C }{V_\omega} (I_\omega \lambda_k)^{\frac{q}{q-1}}.$$

\section{Local Sobolev inequality and Poincar\'e inequality}\label{section 5}
\setcounter{equation}{0}

In this section, we show that the Sobolev inequality and Poincar\'e inequality continue to hold on domains in $X$. Let $\Omega'\subset X$ be an open domain which for simplicity is assumed connected, and $\Omega\subset \Omega'$ be a relatively compact subdomain of $\Omega'$. We fix a compactly supported function $\eta \in C^\infty_0(\Omega')$ such that $\eta \equiv 1$ in $\Omega$.  The main result is 
\begin{lemma}\label{lemma local Sob}\label{lemma 4}
There is a uniform constant $C=C(n,p,A, K, \gamma)>0$ such that for any $u\in W^{1,2}(\Omega')$, we have:

(1). The local Sobolev inequality:
\begin{equation}\label{eqn:local S}
\Big( \frac{1}{V_\omega}\int_\Omega | u - \bar u_\Omega  |^{2q} \omega^n  \Big)^{1/q} \le \frac{C}{V_\omega} \int_{\Omega'} |\nabla (\eta u)|^2_\omega \omega^n,
\end{equation}
and  (2). The local Poincar\'e inequality:
\begin{equation}\label{eqn:local P}
\int_\Omega | u - \bar u_\Omega  |^{2} \omega^n {} \le {C}\int_{\Omega'} |\nabla (\eta u)|^2_\omega \omega^n,
\end{equation}
where $\bar u_\Omega  = \frac{1}{V_\omega(\Omega)} \int_\Omega u \omega^n$ is the average of $u$ over $\Omega$ under the metric $\omega$ and $V_\omega(\Omega) = \int_\Omega \omega^n$ is the $\omega$-volume of $\Omega$.
\end{lemma}

\begin{proof} By approximation, we may assume $u\in C^1(\Omega')$ and $\int_{\Omega'}( u^2 + |\nabla u|^2_\omega)\omega^n <\infty$. For simplicity we denote $\tilde u = \eta u$ which can be extended trivially to $X$ as a $C^1$ function. By the Green's formula, we have for the function $\tilde u$,
\begin{equation}\label{eqn:4.3}
\tilde u(x)  = \frac{1}{V_\omega}\int_X \tilde u \omega^n + \int_{y\in X} \langle \nabla {\mathcal G}(x, y) ,\nabla \tilde u(y)\rangle _\omega  \omega^n. 
\end{equation}
Applying \nref{eqn:4.3} to different points $x, z\in \Omega$ and taking the difference result in the equation
\begin{equation}\label{eqn:4.4}
\tilde u(x)  - \tilde  u(z)= \int_{y\in X} \langle \nabla {\mathcal G}(x, y) ,\nabla \tilde u(y)\rangle _\omega  \omega^n - \int_{y\in X} \langle \nabla {\mathcal G}(z, y) ,\nabla \tilde u(y)\rangle _\omega  \omega^n. 
\end{equation}
Integrating \nref{eqn:4.4} over $z\in \Omega$ against $\omega^n$ and dividing the obtained equation by $V_\omega (\Omega)$, we get
\begin{equation}\label{eqn:4.5}
\tilde u(x)  - \bar {\tilde  u}_\Omega= \int_{y\in X} \langle \nabla {\mathcal G}(x, y) ,\nabla \tilde u(y)\rangle _\omega  \omega^n - \frac{1}{V_\omega(\Omega)} \int_{z\in \Omega}\omega^n\int_{y\in X} \langle \nabla {\mathcal G}(z, y) ,\nabla \tilde u(y)\rangle _\omega  \omega^n,
\end{equation}
which implies by triangle inequality that 
{\small
\bea
\nonumber |\tilde u(x)  - \bar {\tilde  u}_\Omega| & \le & \int_{ X} |\nabla {\mathcal G}|(x, y) |\nabla \tilde u(y)| _\omega  \omega^n + \frac{1}{V_\omega(\Omega)} \int_{z\in \Omega}\omega^n\int_{y\in X} |\nabla {\mathcal G}|(z, y) |\nabla \tilde u(y)|_\omega  \omega^n\\
&\le & \Big( \int_X \frac{|\nabla {\mathcal G}(x,y)|^2}{{\mathcal G}(x,y)^{1+\beta}} \omega^n \Big) ^{1/2}  \Big( \int_X {\mathcal G}(x,y)^{1+\beta} |\nabla \tilde u|^2\omega^n \Big) ^{1/2} \nonumber \\
&& ~ +  \frac{1}{V_\omega(\Omega)} \int_{z\in \Omega}\omega^n  \Big( \int_X \frac{|\nabla {\mathcal G}(z,y)|^2}{{\mathcal G}(z,y)^{1+\beta}} \omega^n \Big) ^{1/2}  \Big( \int_X {\mathcal G}(z,y)^{1+\beta} |\nabla \tilde u|^2\omega^n \Big) ^{1/2} \nonumber\\
\mbox{by \nref{eqn:useful 2}}&\le &\frac{V_\omega^{\beta/2}}{\beta^{1/2}} \Big[ \Big( \int_X {\mathcal G}(x,\cdot)^{1+\beta} |\nabla \tilde u|^2 \Big) ^{1/2} + \frac{1}{V_\omega(\Omega)}  \int_{\Omega}  \Big( \int_X {\mathcal G}(z,\cdot)^{1+\beta} |\nabla \tilde u|^2\omega^n \Big) ^{1/2} \omega^n\Big] \nonumber\\
&\le &\frac{V_\omega^{\beta/2}}{\beta^{1/2}} \Big[ \Big( \int_X {\mathcal G}(x,\cdot)^{1+\beta} |\nabla \tilde u|^2\Big) ^{1/2} +  \Big(  \frac{1}{V_\omega(\Omega)}  \int_{z\in \Omega}\int_X {\mathcal G}(z,\cdot)^{1+\beta} |\nabla \tilde u|^2 \Big) ^{1/2} \Big], \label{eqn:4.6}
\eea
}
where the last inequality follows from H\"older's inequality. To get the local Poincar\'e inequality \nref{eqn:local P}, we take square on both sides of \nref{eqn:4.6} and integrate the inequality over $x\in \Omega$, then we obtain
\bea
\int_{\Omega} |\tilde u - \bar{\tilde u}_\Omega|^2 & \le &\frac{2V_\omega^\beta}{\beta} \Big[ \int_{x\in \Omega }\int_X {\mathcal G}(x,\cdot)^{1+\beta} |\nabla \tilde u|^2 +   \int_{z\in \Omega}\int_X {\mathcal G}(z,\cdot)^{1+\beta} |\nabla \tilde u|^2 \Big],\nonumber\\
&\le \nonumber&\frac{C}{\beta} \int_X |\nabla \tilde u|_\omega^2 \omega^n,
\eea
where we apply the inequality \nref{eqn:useful 1} in the last line by choosing $\beta = \varepsilon_0$. The inequality \nref{eqn:local P} follows since $\tilde u = \eta u = u$ in $\Omega$ and $\tilde u$ is supported in $\Omega'$.

To prove the local Sobolev inequality \nref{eqn:local S}, we choose $\beta = \varepsilon_0/2$ and $q>1$ as in \nref{eqn:q}. Taking $(2q)$-th power of both sides of \nref{eqn:4.6} and arguing as that in the proof of Theorem \ref{thm:Sob}, the inequality \nref{eqn:local S} follows. 
\end{proof}

\medskip

We denote $W^{1,2}_0(\Omega)$ the $W^{1,2}$-functions with compact support in $\Omega$. For these functions, we have the following local Sobolev and Poincar\'e inequalities:
\begin{lemma}\label{lemma 5}
There is a uniform constant $C=C(n,p,A, K, \gamma)>0$ such that for any $u\in W_0^{1,2}(\Omega)$, we have:

(1). The local Soblev inequality:
\begin{equation}\label{eqn:local S1}
\Big( \frac{1}{V_\omega}\int_\Omega | u|^{2q} \omega^n  \Big)^{1/q} \le C\Big(1+ \frac{V_\omega(\Omega)}{V_\omega(\Omega^c)}\Big)  \frac{1}{V_\omega} \int_{\Omega} |\nabla u|^2_\omega \omega^n,
\end{equation}
and  (2). The local Poincar\'e inequality:
\begin{equation}\label{eqn:local P1}
\int_\Omega | u |^{2} \omega^n {} \le {C} \Big(1+ \frac{V_\omega(\Omega)}{V_\omega(\Omega^c)}\Big)  \int_{\Omega} |\nabla u|^2_\omega \omega^n,
\end{equation}
where $\Omega^c = X\backslash \Omega$ is the complement of $\Omega$ in $X$.
\end{lemma}
\begin{proof} The proof is similar to that of the above lemma. So we just outline the main steps. We may assume $u\in C^1_0(\Omega)$ and $u$ can be viewed as a $C^1$ function in $X$. By the Green's formula, we have for any $x\in X$
\begin{equation}\label{eqn:4.9}
u(x)  = \frac{1}{V_\omega}\int_X  u \omega^n + \int_{y\in X} \langle \nabla {\mathcal G}(x, y) ,\nabla u(y)\rangle _\omega  \omega^n. 
\end{equation}
Applying \nref{eqn:4.9} to $x\in \Omega^c = X\backslash \Omega$ at which $u = 0$ and integrating the equation over $\Omega^c$, we obtain 
\begin{equation}\label{eqn:4.10}
\frac{1}{V_\omega} \int_X u \omega^n = - \frac{1}{V_\omega(\Omega^c)} \int_{\Omega^c} \int_{y\in X} \langle \nabla {\mathcal G}(x, y) ,\nabla u(y)\rangle _\omega  \omega^n.
\end{equation}
The equations \nref{eqn:4.10} and \nref{eqn:4.9} together yield that for any $x\in \Omega$
\begin{equation}\label{eqn:4.11}
u(x)  =  - \frac{1}{V_\omega(\Omega^c)} \int_{z\in \Omega^c} \int_{y\in X} \langle \nabla {\mathcal G}(z, y) ,\nabla u(y)\rangle _\omega  \omega^n + \int_{y\in X} \langle \nabla {\mathcal G}(x, y) ,\nabla u(y)\rangle _\omega  \omega^n. 
\end{equation}
Taking square on both sides of \nref{eqn:4.11} gives that for any $x\in \Omega$
\begin{equation}\label{eqn:4.12}
|u(x)|^2 \le  \frac{C}{V_\omega(\Omega^c)} \int_X |\nabla u|^2_\omega \omega^n + \frac{V_\omega^\beta}{\beta} \int_X {\mathcal G}(x,y)^{1+\beta} |\nabla u|_\omega ^2 \omega^n.
\end{equation}
Integrating \nref{eqn:4.12} over $\Omega$ yields that
$$\int_\Omega |u|^2 \omega^n \le C \frac{ V_\omega(\Omega)}{V_\omega(\Omega^c)} \int_X |\nabla u|^2_\omega \omega^n + C \int_X|\nabla u|_\omega ^2 \omega^n,$$
which is the inequality \nref{eqn:local P1}. The local Sobolev inequality \nref{eqn:local S1} can be proved by taking $q$-th power of both sides of \nref{eqn:4.12} and arguing similarly to that in Lemma \ref{lemma 4}.
 \end{proof} 

\medskip

We next consider a  mean value inequality for sub-solutions of the Laplacian $\Delta_\omega$ on geodesic balls in $X$. We fix a point $x\in X$ and denote the geodesic ball $B_\omega(x,R)$ by $B_R$ for any $R>0$.  Suppose $u\in W^{1,2}(B_R)$ satisfies 
\begin{equation}\label{eqn:4.13}
\Delta_\omega u \ge f, \quad \mbox{in }B_R. 
\end{equation} 
By the classical Moser's iteration argument (e.g. \cite{HL}), we can derive the following geometric mean value inequality.
\begin{lemma}\label{lemma 6} Given a constant $N>\frac{q}{q-1}$, 
there exists a constant $C=  C(n, p, A, K, \gamma, N)>0$ such that for any $s>0$
{\small
\begin{equation}\label{eqn:4.14}
\sup_{B_r} u^+ \le  \frac{C}{(R-r)^{\frac{1}{q-1}}}\big( \frac{1}{V_\omega} \int_{B_R} |f|^N \big)^{\frac{1}{N}} + \frac{C}{(R-r)^{\frac{2}{s(q-1)}}} \big (\frac{1}{V_\omega}\int_{B_R} (u^+)^s \big)^{1/s} ,
\end{equation}
}
for any $0< r <R$. 
\end{lemma}

\begin{proof} We denote $k = (R-r)^\beta \| f\|_{L^N(B_{R}, \omega^n)} = (R-r)^\beta(\frac{1}{V_\omega} \int_{B_R} |f|^N \omega^n)^{1/N}$ if $|f|\not\equiv 0$, and $k>0$ is an positive number if $f\equiv 0$, and here $\beta = \frac{2(q-1)N - 2q}{(q-1)N}>0$. Replacing $u$ by  $u^++k \ge k$, we may assume $u\ge k$. Fix any  constants $0< r<R_2 < R_1 < R$. Choose a cut-off function $\eta$ that depends on $d_\omega(x,\cdot)$, the distance function to the center $x\in B_R$, such that $\eta \equiv 1$ in $B_{R_2}$, supp$(\eta)\subset B_{R_1}$, and 
$$\sup |\nabla \eta|_\omega \le \frac{4}{R_1 - R_2}.$$ 
Multiplying both sides of \nref{eqn:4.13} by $\eta^2 u^m$ for some $1\le m$, and applying integration by parts, we obtain (we omit the integral domain $B_R$ and the measure $\omega^n$ in the following calculations.)
\bea\nonumber
\int m \eta^2 u^{m-1} |\nabla u|^2_\omega & \le&  \int 2 \eta u^m |\nabla \eta| |\nabla u|  + \int f \eta^2 u^m\\
&\le \nonumber & \frac{1}{2} \int m \eta^2 u^{m-1} |\nabla u|^2 + \frac{2}{m} \int u^{m+1} |\nabla \eta|^2 + \int f \eta^2 u^m,
\eea
hence 
$$\frac{2m}{(m+1)^2} \int \eta^2 |\nabla u^{\frac{m+1}{2}}|^2 \le \frac{2}{m} \int  u^{m+1} |\nabla \eta|^2 + \int f \eta^2 u^m,$$
which is equivalent  to
$$ \int |\nabla( \eta u^{\frac{m+1}{2}}) |^2 \le 8 \int  u^{m+1} |\nabla \eta|^2 + 2m  \int f \eta^2 u^m.$$
By the Sobolev inequality \nref{eqn:Sob}, this implies 
\begin{equation}\label{eqn:4.15} \Big( \frac{1}{V_\omega}  \int \eta^{2q} u^{(m+1) q} \Big)^{1/q}\le \frac{C }{V_\omega}\Big(  \int  (\eta^2 u^{m+1}+ u^{m+1} |\nabla \eta|^2) + m  \int f \eta^2 u^m\Big).\end{equation}
We observe that the last term in \nref{eqn:4.15} satisfies
\bea\nonumber
\frac{m }{V_\omega} \int f \eta^2 u^m & \le & \frac{m}{V_\omega} \int \frac{|f|}{k} \eta ^2 u^{m+1}\le  \frac{m}{k} \| f\|_{L^N} \| \eta^2 u^{m+1}\|_{L^{N^*}} \\
&\le&  {m} (R-r)^{-\beta}\| \eta^2 u^{m+1}\|_{L^{N^*}}.\label{eqn:4.15m}
\eea
Applying H\"older's inequality to the right-hand side of \nref{eqn:4.15m}, we obtain
{\small
\bea
&&  {m} (R-r)^{-\beta}\| \eta^2 u^{m+1}\|_{L^{N^*}}  \le   {m}  (R-r)^{-\beta}\| \eta^2 u^{m+1} \|_{L^{q}}^{\frac{q}{(q-1)N}} \cdot\| \eta^2 u^{m+1}\|_{L^1} ^{\frac{(q-1)N - q}{N(q-1)}} \nonumber\\
 &\le & \label{eqn:4.17} m   (R-r)^{-\beta} \epsilon \| \eta^2 u^{m+1} \|_{L^{q}} + C m   (R-r)^{-\beta} \epsilon^{-\frac{q}{(q-1)N - q}} \| \eta^2 u^{m+1} \|_{L^{1}},
\eea
}
where in \nref{eqn:4.17} we make use of Young's inequality.  We choose $\epsilon = \frac{1}{2C m   (R-r)^{-\beta}}>0$ for $C>0$ the constant  on the right-hand side of \nref{eqn:4.15}. Then combining \nref{eqn:4.17} and \nref{eqn:4.15}, we get
\begin{equation}\label{eqn:4.18}
\Big( \frac{1}{V_\omega}  \int \eta^{2q} u^{(m+1) q} \Big)^{1/q}\le \frac{C }{V_\omega}\Big(  \int  ( \frac{m^{\alpha}   \eta^2 u^{m+1}}{(R-r)^2}+ u^{m+1} |\nabla \eta|^2)\Big)
\end{equation}
where we denote $\alpha = \frac{(q-1)N}{(q-1)N - q}>0$. By the choice of the cut-off function $\eta$ and the inequality $\frac{1}{R-r} \le \frac{1}{R_1 - R_2}$, this yields
\begin{equation}\label{eqn:4.19}
\Big( \frac{1}{V_\omega}  \int_{B_{R_2}} u^{(m+1) q} \omega^n \Big)^{1/q(m+1)}\le C^{\frac{1}{m+1}} \frac{(m+1) ^{\alpha/(m+1)}}{(R_2-R_1)^{2/(m+1)}}\Big( \frac{1}{V_\omega}\int_{B_{R_1}} u^{m+1} \omega^n\Big)^{\frac{1}{m+1}}.
\end{equation}
We will iteratively apply \nref{eqn:4.19} to a sequence of numbers $m_j = 2 q^j$ and radii
$$R_j =  r + \frac{R-r}{2^{j}}, \quad \mbox{for } j = 0, 1, 2,\ldots$$
It is clear that $R_0 = R$, $R_\infty = r$, $R_j - R_{j+1}  = {2^{-j-1}} (R-r)$. Iterating the ineuquality \nref{eqn:4.19} with $R_1 = R_{j}$ and $R_2 = R_{j+1}$ for $j = 0,\ldots, \ell$ gives 
\begin{equation}\label{eqn:4.20}
\Big( \frac{1}{V_\omega}  \int_{B_{R_{\ell+1}}} u^{m_{\ell + 1}} \omega^n \Big)^{\frac{1}{m_{\ell+ 1}}}\le  \frac{ (2^\alpha C)^{\sum_j \frac{1}{2q^{j}}}  q ^{\alpha \sum_j  j q^{-j}}   2^{  \sum_j q^{-j}   }    }{     (R-r)^{\sum_{j=0}^\ell  q^{-j}}   }\Big( \frac{1}{V_\omega}\int_{B_{R}} u^{2} \omega^n\Big)^{\frac{1}{2}}.
\end{equation}
Letting $\ell\to \infty$, we get from \nref{eqn:4.20} that 
\begin{equation}\label{eqn:4.21}
\sup_{B_r} u \le \frac{C}{(R- r) ^{\frac{q}{q-1}}} \Big( \frac{1}{V_\omega}\int_{B_{R}} u^{2} \omega^n\Big)^{\frac{1}{2}},
\end{equation}
which implies by our choice of $u: = u^+ + \| f\|_{L^N} $ that 
\begin{equation}\label{eqn:4.22}
\sup_{B_r} u^+ \le \frac{C}{(R- r) ^{\frac{q}{q-1}}} \Big( \| u^+\|_{L^2(B_R)} + \frac{V_\omega(B_R)^{1/2}}{V_\omega^{1/2}} (R-r)^\beta \| f\|_{L^N(B_R)} \Big).
\end{equation}
This proves \nref{eqn:4.14} with $s = 2$, and hence for all $s\ge 2$ by H\"older's inequality. Next let $s\in (0,2)$ be a given number. Note that \nref{eqn:4.22} tells that 
for any $0< r < R'\le R${\small
\bea \nonumber 
\| u^+ \|_{L^\infty(B_{ r})}& \le &  \frac{C}{(R'-r)^{\frac{1}{q-1}}} \Big( \| u^+\|_{L^2(B_{R'})} + \| f\|_{L^N(B_{R'})} \Big)\\
&\le &\frac{C}{(R'-r)^{\frac{1}{q-1}}} \Big( \| u^+\|_{L^\infty(B_{R'})}^{1- \frac{s}{2}}   \| u^+\|_{L^s(B_{R'})}^{s/2} + \| f\|_{L^N(B_{R'})} \Big) \nonumber\\
&\le & \frac{1}{2} \| u^+\|_{L^\infty(B_{R'})} +  \frac{C}{(R'-r)^{\frac{2}{s(q-1)}}} \| u^+ \|_{L^s(B_{R'})} +  \frac{C}{(R'-r)^{\frac{1}{q-1}}}   \| f\|_{L^N(B_{R'})}.\label{eqn:4.23}
\eea }
If we denote $\phi(r) : =\| u^+ \|_{L^\infty(B_{ r})}$ which is monotonically increasing and by \nref{eqn:4.23} it satisfies the inequality
$$\phi(r) \le \frac{1}{2} \phi(R') + \frac{C}{(R'-r)^{\frac{2}{s(q-1)}}} A + \frac{C}{(R'-r)^{\frac{1}{q-1}}} B$$
with $A :=  \| u^+ \|_{L^s(B_R)}$ and $B: =  \| f\|_{L^N(B_R)}$. We can now apply the standard iteration lemma (see e.g. Lemma 4.3,  Chap. 4, in \cite{HL}) to conclude that 
$$
\phi(r) \le \frac{C}{(R-r)^{\frac{2}{s(q-1)}}} \| u^+ \|_{L^s(B_R)} + \frac{C}{(R-r)^{\frac{1}{q-1}}}\| f\|_{L^N(B_R)},
$$
which is \nref{eqn:4.14}.
\end{proof}

\smallskip

As a corollary of Lemma \ref{lemma 6}, we have the following {\em pointwise} upper bound of the Green's function $G_\omega(x,y)$, which we do not expect to be sharp.
\begin{corollary}\label{cor Green}
The Green's function $G_\omega(x,y)$ satisfies the upper bound: for any $x\neq y$
$$G_\omega(x,y)\le \frac{C}{V_\omega} d_\omega(x,y)^{-\frac{2}{q-1}}.$$
\end{corollary}
\begin{proof} For any two different points $x, y\in X$, we let $0<R = d_\omega(x,y)/2$, and consider the function $u(x) = G_\omega(x,y)$ as a {\em smooth} function on $B_R: = B_\omega(y, R)$, which satisfies the differential equation
$$\Delta_\omega u =  \frac{1}{V_\omega},\quad \mbox{in }B_R.$$ We apply Lemma \ref{lemma 6} with $r = R/2$ and $s=1$ to conclude that
$$\sup_{B_{R/2}} u^+ \le C \Big( V_\omega^{-1} R^{-\frac{1}{q-1}} +  V_\omega^{-1} R^{-\frac{2}{q-1}}   \Big)\le C  V_\omega^{-1} R^{-\frac{2}{q-1}}.$$
This completes the proof. 
\end{proof}

\section{Smooth approximation for $\mathcal{AK}$-currents} \label{maregsec}
\setcounter{equation}{0}

\subsection{K\"ahler varieties}

We first recall the definition of a smooth K\"ahler metric on a  K\"ahler space. 

\begin{definition}\label{defn:6.1}
Let $X$ be a complex space of dimension $n$. We say $\theta$ is a smooth K\"ahler metric on $X$, if for any point $x\in X$ there is an open subset $U$ of $X$ together with a local holomorphic embedding $\iota_x: U\hookrightarrow {\mathbb C}^{N_x}$ such that $\theta|_U$ is the pullback of some smooth K\"ahler metric on ${\mathbb C}^{N_x}$ under $\iota_x$.
\end{definition}

If $\theta$ is a smooth K\"ahler metric in the sense of Definition \ref{defn:6.1}, then $\theta|_{X^{reg}}$ is a {\em genuine} K\"ahler metric on the smooth manifold $X^\circ$, the regular part of $X$. $X$ is said to be a K\"ahler space or variety if there exists a smooth K\"ahler metric on $X$. If $X$ is a compact normal K\"ahler space, then there exists  a resolution of singularities $\pi: Y\to X$ such that the smooth model $Y$ of $X$ is also K\"ahler by \cite{Kj, CMM}. {\iffalse To avoid technical subtleties, we will assume the K\"ahler condition for $Y$ as in the Definition \ref{defak}.\fi}It is straightforward to check that $[\pi^*\theta]$ is  big and nef on $Y$ since one can construct a K\"ahler current on $Y$ by \cite[Theorem 2.12]{DP}. Then there exists an effective divisor $E$ on $Y$ such that  $[\pi^*\theta] - \delta [E]$ is a K\"ahler class for any sufficiently small $\delta>0$ by Demailly's regularization techniques (cf. \cite{DP, CMM}). 

Starting from this section, we will assume $X$ is an $n$-dimensional compact normal K\"ahler variety. Let $\pi: Y\rightarrow X$ be a log resolution of $X$ with a fixed smooth K\"ahler metric  $\theta_Y$  on $Y$.

\begin{definition} \label{logtype} A nonnegative function $f$ on $Y$ is said to have log type analytic singularities if the following hold.

\begin{enumerate}

\item There exist holomorphic effective smooth divisors $D_j$ on $Y$  with simple normal crossings, for $j=1, ...., N$. Let $\sigma_j$ be the defining section of $D_j$ and $h_j$ be a smooth hermitian metric associated to $D_j$.

\item $$f= \sum_{k=1}^K  a_k(-\log)^k \Big( \prod_{j=1}^N  e^{f_{k, j}} |\sigma_j|^{2b_{k, j}}_{h_j} \Big),$$
where $(-\log)^k$ is the $k^{th}$ composition of $(-\log)$, $a_k\in \mathbb{R}$, $b_{k, j}\geq 0$, $f_{k, j}\in C^\infty(Y)$ and $K\in \mathbb{Z}^+$.

\end{enumerate}  

\end{definition}
One can easily expand Definition \ref{logtype} to a class of more general functions. However, Definition \ref{logtype} is sufficient to study most canonical metrics such as singular K\"ahler-Einstein metrics.

 For any semi-K\"ahler current $\omega \in \mathcal{AK}(X, \theta_Y, n,  p,A, K, \gamma)$, we choose a smooth closed $(1,1)$-form $\omega_0 \in [\omega]$. Then there exists a unique $\varphi\in PSH(X, \omega_0)\cap C^0(X)$ such that
$$\omega= \omega_0+\ddbar \varphi, ~\sup_X \varphi=0. $$
Let $$F = \log \Big(\frac{1}{V_\omega}\frac{(\pi^*\omega)^n}{(\theta_Y)^n}\Big).$$
Then $F$ has log type analytic singularities and
$$\mathcal{N}_p(\omega) = \frac{1}{V_\omega} \int_Y |F|^p \omega^n =  \frac{1}{V_\omega} \int_Y |F|^p e^F (\theta_Y)^n \leq K .$$

As in the definition, let 
\begin{equation}\label{logres}
\pi: Y \rightarrow X
\end{equation}
 be a log resolution of $X$ and $\theta_Y$ be a K\"ahler metric on $Y$. We can find a sequence of $\{ F_j \}_{j=1}^\infty$ that approximate $F$ satisfying the following. 

\begin{enumerate}

\item $F_j\in C^\infty(X)$ and $e^{F_j} \geq 0.5 \pi^*\gamma$, for all sufficiently large $j$. 

\medskip

\item There exists $K'>0$ such that  for all $j\geq 1$
\begin{equation}\label{appr6}
\int_Y |F_j|^p e^{F_j} (\theta_Y)^n \leq K'.
\end{equation}
In the case when $\|e^F\|_{L^{1+\epsilon'}(Y, (\theta_Y)^n)} <\infty$ for some $\epsilon'>0$, there exists $K''>0$ such that 
\begin{equation}\label{appr7}
\|e^{F_j}\|_{L^{1+\epsilon'}(Y, (\theta_Y)^n)} \leq K''.
\end{equation}

\medskip

\item Let $D$ be an effective divisor of $Y$ such that the support of $D$ contains  the exceptional locus of $\pi$ and the singular locus of $F$. We can assume that 
$$\pi^*[\omega] - \delta [D]$$
is a K\"ahler class for some sufficiently small $\delta>0$. 
Let $\sigma_D$ be the defining section of $D$ and $h_D$ be a smooth hermitian metric associated to $D$. Then there exists $N>0$ and $C>0$ such that 
$$ \sup_Y |\ddbar F_j |_{\theta_Y} \leq C |\sigma_D|^{-2N}_{h_{\redn{D}}}. $$

\end{enumerate}
The simplest choice for $F_j$ will be 
$$F_j= \sum_{k=1}^K  a_k(-\log)^k \Big( \prod_{\redn{i}=1}^N  e^{f_{k, \redn{i}}} \left( |\sigma_{\redn{i}}|^2_{h_{\redn{i}}} + j^{-1} \right)^{b_{k, \redn{i}}} \Big),$$

We now consider the following complex Monge-Amp\`ere equation
\begin{equation}\label{appeq6}
 (\pi^*\omega_0 + \epsilon_j \theta_Y + \ddbar \varphi_j)^n = e^{F_j + c_j} (\theta_Y)^n, ,~~\sup_X \varphi_j = 0,
 \end{equation}
where $c_j$ is the normalizing constant satisfying 
$$\int_Y e^{F_j+ c_j} (\theta_Y)^n = \left(\pi^* [\omega_0] + \epsilon_{\redn{j}} [\theta_Y] \right)^n. $$
We define 
\begin{equation}\label{omj}
\omega_j = \pi^*\omega_0 + \epsilon_j \theta_Y + \ddbar \varphi_j. 
\end{equation}

\begin{lemma} \label{02nd} There exist $N'>0$, $C>0$ such that 
$$\|\varphi_j\|_{L^\infty(Y)} \leq C, ~ \Delta_{\theta_Y} \varphi_j \leq C|\sigma_D|_{h_D}^{-2N'}. $$

\end{lemma}

\begin{proof} 

The $L^\infty$-estimate directly follows from \cite{K, EGZ, GPT} by the uniform bound (\ref{appr6}). 
For the second order estimate for $\varphi_j$, we let $\omega_j = \pi^*\omega_0 + \epsilon_j \theta_Y + \ddbar\theta_j$ and let  $h_D$ be a smooth hermtian metric associated to $D$ such that 
$$\pi^*\omega_0  - \delta {\mathrm{Ric}}(h_D) >0.$$
We recall that $\omega_j$ satisfies  the complex Monge-Amp\`ere equation \nref{appeq6}. 
It follows from standard calculations \cite{Y} that 
\begin{equation}\label{eqn:yau1}
\Delta_{\omega_j} \tr_{\theta_Y} \omega_j \ge \Delta_{\theta_Y} F_j - C_1 (\tr_{\omega_j} \theta_Y)\,(\tr_{\theta_Y} \omega_j )+ \frac{|\nabla \tr_{\theta_Y} \omega_j|^2_{\omega_j}}{\tr_{\theta_Y} \omega_j},
\end{equation}
where $-C_1$ is a lower bound of the bisectional curvature of $(Y, \theta_Y)$.  Let $\hat\varphi_j = \varphi_j - \delta \log |\sigma_D|^2_{h_D}$. For some $\mu>0$ to be determined, we define $Q = e^{-\mu \hat\varphi_j} \tr_{\theta_Y} \omega_j$. We will choose $\mu>0$ sufficiently large so that $Q$ is $C^2$ in $Y$. Using \nref{eqn:yau1}, we calculate the Laplacian of $Q$ as follows:
{\small
\bea\nonumber
\Delta_{\omega_j} Q & \ge & e^{-\mu \hat\varphi_j} \big(  \Delta_{\theta_Y} F_j - C_1 (\tr_{\omega_j} \theta_Y)\,(\tr_{\theta_Y} \omega_j )+ \frac{|\nabla \tr_{\theta_Y} \omega_j|^2_{\omega_j}}{\tr_{\theta_Y} \omega_j} \big)+ \mu^2 e^{-\mu \hat\varphi_j} \tr_{\theta_Y} \omega_j |\nabla \hat\varphi_j|^2_{\omega_j} \\
&& \label{eqn:yau2}
+ \mu Q \big( -n + \epsilon_j \tr_{\omega_j} \theta_Y + \tr_{\omega_j} (\omega_0 - \delta \ric(h_D))   \big)  - 2 \mu e^{-\mu \hat\varphi_j} {\mathrm{Re}} \langle \nabla \hat\varphi_j, \bar \nabla \tr_{\theta_Y} \omega_j\rangle_{\omega_j}.
\eea
}
By Cauchy-Schwarz inequality, the last term in \nref{eqn:yau2} satisfies 
$$- 2 \mu e^{-\mu \hat\varphi_j} {\mathrm{Re}} \langle \nabla \hat\varphi_j, \bar \nabla \tr_{\theta_Y} \omega_j\rangle_{\omega_j} \ge  - e^{-\mu\hat\varphi_j} \frac{|\nabla \tr_{\theta_Y} \omega_j|^2_{\omega_j}}{\tr_{\theta_Y} \omega_j} -  \mu^2 e^{-\mu \hat\varphi_j} \tr_{\theta_Y} \omega_j |\nabla \hat\varphi_j|^2_{\omega_j}.$$ We fix a constant $\mu_1>0$ such that 
$$\mu_1 (\omega_0 - \delta\ric(h_D)) \ge 2 \theta_Y.$$ The constant $\mu$ is chosen so that $\mu\ge \mu_1$. Then \nref{eqn:yau2} yields that
\begin{equation}\label{eqn:yau3}
\Delta_{\omega_j} Q \ge e^{-\mu \hat\varphi_j} \Delta_{\theta_Y} F_j - n \mu Q.
\end{equation}
Note that $|\ddbar F_j|_{\theta_Y} \le C |\sigma_D|_{h_D}^{-2N}$. If additionally we take $\mu>0$ such that $\mu \delta \ge N$, from \nref{eqn:yau3} we obtain
\begin{equation}\label{eqn:yau4}
\Delta_{\omega_j} Q \ge -C |\sigma_D|_{h_D}^{2(\mu\delta - N)}  - n \mu Q\ge - n\mu Q - C,
\end{equation}
where we have assumed $|\sigma_D|_{h_D}^2\le 1$. Recall that the Green's function $G_j$ associated to $\omega_j$ satisfies the uniform integral estimates as in \cite{GPS, GPSS}. Applying Green's formula at a point $x\in Y$ where $Q$ achieves its maximum $\bar Q$, we have
\bea \nonumber
\bar Q & \le  & \frac{1}{V_{\omega_j}} \int_Y Q \omega_j^n + \int_Y {\mathcal G}_j(x,\cdot) (n\mu Q + C ) \omega_j^n\\
&\le &\label{eqn:yau5} \frac{\bar Q^{1-\epsilon}}{V_{\omega_j}} \int_Y Q^{\epsilon} \omega_j^n + C + \bar Q^{1-\frac{\epsilon}{2n}} \big(  \int_Y Q^\epsilon \omega_j^n\big)^{1/2n},
\eea
where the inequality \nref{eqn:yau5} follows from the H\"older's inequality and $\epsilon = \frac{\epsilon'}{1+\epsilon'}>0$ is a constant chosen so that the following holds
$$  \int_Y Q^\epsilon \omega_j^n\le C \int_Y (\tr_{\theta_Y} \omega_j) ^{\epsilon} e^{F_j} \theta_Y^n\le C \big( \int_Y ( \tr_{\theta_Y} \omega_j )\theta_Y^n\big)^{\frac{\epsilon'}{1+\epsilon'}}\le C $$ where the constant $C>0$ depends on $\| e^{F_j}\|_{L^{1+\epsilon'}}$ which is uniformly bounded. 
Then \nref{eqn:yau5} gives that
$$\bar Q \le C  \bar Q^{1-\epsilon} + C \bar Q^{1-\frac{\epsilon}{2n}} + C,$$
which implies the upper bound of $\bar Q$, hence that of $Q$.
%
\end{proof}

Immediately, we can obtain local higher order estimate for $\varphi_j$ uniformly away from $Y\setminus D$. Let $\cS_{X, \omega}$ be the union of the singular set of $X$ and the singular set of $\pi_*F$. Since $\pi^{-1}(\cS_{X ,\omega})$ is contained in the ample locus of $[\pi^*\omega]$, for every point $p\in X\setminus \cS_{X, \omega}$, we can choose $D$ such that $\pi^{-1}(p) \in Y \setminus D$.  The standard Schauder estimates give the following regularity result for $\omega\in \mathcal{AK}(X, \theta_Y, n, p, A, K, \gamma)$.

\begin{proposition} \label{omapp} The smooth K\"ahler metrics $\omega_j$ on $Y$ satisfy the following conditions. Then for $\mathcal{K} \subset \subset X\setminus \cS_{X, \omega}$ and $k>0$, we have 
$$\lim_{j\rightarrow \infty} \|\varphi_j - \pi^*\varphi\|_{L^\infty(\mathcal{K})} =0, $$
$$\lim_{j\rightarrow \infty}  \|\omega_j - \omega\|_{C^k(\mathcal{K})} =0.$$
Here we identify $\omega_j$ and $\pi_*\omega_j$ on $X\setminus \cS_{X, \omega}$. In particular, $\omega$ is a smooth K\"ahler metric on $X\setminus \cS_{X, \omega}$.

\end{proposition}


\section{Sobolev inequalities for $\mathcal{AK}$-currents}\label{singsobsec}\label{section 7}
\setcounter{equation}{0}

Let $\omega\in \mathcal{AK}(X, \theta_Y, n, p, A, K, \gamma)$ as in Definition \ref{defak} and $\cS_{X, \omega}$ be the union of the singular set of $X$ and \redn{the image of the} singular set of $\log \frac{\pi^*\omega^n}{(\theta_Y)^n}$ \redn{under $\pi$}. {We define the distance function $d$ on $(X\backslash \cS_{X,\omega}, \omega)$ as follows:  for any $x,y\in X\backslash \cS_{X,\omega}$
\begin{equation}\label{eqn:dist}d(x,y) = \inf\{L_\omega(c): ~ c:[0,1]\to X\backslash \cS_{X,\omega},\, c(0) = x,\, c(1) = y\},\end{equation} where $c$ is a piecewise differentiable curve in $X\backslash \cS_{X,\omega}$, and $L_\omega(c) = \int_0^1 |\dot c(t)|_\omega dt$ is the length of $c$ under the smooth metric $\omega$. It is not hard  to see $d$ is a distance function on $X\backslash \cS_{X,\omega}$.   Let $(\hat X, \hat d)$ be the metric completion of $(X\backslash \cS_{X,\omega}, d)$. For notational simplicity we will write $\hat d$ as $d$.}

We will first set up the Sobolev spaces on the metric measure space $(\hat X, d, \omega^n)$ associated to $(X, \omega)$. It is straightforward to define $L^q(\hat X, \omega^n)$ by identifying it with $L^q(X, \omega^n)$  because the singular set $\cS_{X, \omega}$ is a proper analytic subvariety \redn{with real codimension at least $2$.} 

\begin{definition} \label{sobdef} The Sobolev space $W^{1,2}(\hat X, d, \omega^n)$ is defined by for all $u: X \rightarrow \mathbb{R}$ such that $u|_{X^\circ} \in W^{1,2}_{loc}(X^\circ, \omega)$ and 
$$\sup_{\cK \subset\subset X\setminus \cS_{X,\omega}} \|u\|_{W^{1,2}(\cK, \omega)}  {= \sup_{\cK \subset\subset X\setminus \cS_{X,\omega}} \Big(\int_{{\mathcal K}} (u^2 + |\nabla u|^2_\omega)\omega^n \Big)^{1/2}}< \infty. $$
We then define the $W^{1,2}$-norm of $u$ by 
$$\|u\|_{W^{1,2}(\hat X, d, \omega^n)} =  \sup_{\cK \subset\subset X\setminus \cS_{X,\omega}} \|u\|_{W^{1,2}(\cK, \omega)}. $$

\end{definition}

We can also identify $W^{1,2}(X,\omega)$ with $W^{1,2}(\hat X, d, \omega^n)$ as the integrals are calculated on $X\setminus \cS_{X, \omega}$. We will first prove a Sobolev inequality for {$u \in W^{1,2}(X, \omega)=W^{1,2}(\hat X, d, \omega^n)$ with compact support in $X\backslash \cS_{X,\omega}$}. 

\begin{lemma} \label{sobcps} There exists $C=C (X, \theta_Y, n, p, A, K, \gamma)>0$ such that 
$$\Big( \int_X |u|^{2q} \omega^n\Big)^{\frac{1}{q}} \leq C \Big( \int_X |\nabla u|^2_\omega \omega^n + \int_X u^2  \omega^n \Big)$$
for all $u\in W^{1,2}(\hat X, d, \omega^n)$ with compact support in $X \setminus \cS_{X, \omega}$, where $q>1$ is the Sobolev exponent in Theorem \ref{thm:Sob}.

\end{lemma}

\begin{proof} We will approximate $\omega$ by $\omega_j$ as in Proposition \ref{omapp} on $Y$, the blow-up of $X$. By construction, there exist $A', K'>0$ and a nonnegative $\gamma'=\pi^*\gamma \in C^0(Y)$ 
such that for all $j>0$,  
$$\omega_j \in \mathcal{W}(n, p, \gamma', A', K').$$
 By Theorem \ref{thm:Sob}, there exists $C>0$ such that for all $f \in W^{1, 2}(Y, \omega_j)$, we have
$$\Big( \int_Y | f |^{2q} (\omega_j)^n\Big)^{\frac{1}{q}} \leq C \Big( \int_Y |\nabla f|_{\omega_j}^2 (\omega_j)^n + \int_Y  f^2 (\omega_j)^n \Big).$$
For any $u\in W^{1,2}(X, \omega)$ with compact support in $X\setminus  \cS_{X, \omega} $, since $\omega_j$ converges to $\omega$ \redn{locally smoothly} on $X\setminus \cS_{X, \omega}$, we have $\pi^* u \in W^{1,2}(Y, \omega_j)$ and 
$$\Big( \int_Y | \pi^* u |^{2q} (\omega_j)^n\Big)^{\frac{1}{q}} \leq C \Big( \int_Y |\nabla \pi^* u|_{\omega_j}^2 (\omega_j)^n + \int_Y  (\pi^*u) ^2 (\omega_j)^n \Big).$$
The lemma is proved by letting $j\rightarrow \infty$.
\end{proof}

The following cut-off functions are constructed in \cite{Stu} (c.f. Lemma 3.7 in \cite{S1}).

\begin{lemma} \label{cutoff} Let $X$ be an $n$-dimensional K\"ahler \redn{variety} and $E$ is a subvariety of $X$. Let $\omega$ be a semi-K\"ahler current on $X$ with bounded local potentials. Then for any $\epsilon>0$ and $\mathcal{K}\subset\subset X\setminus E$, there exists a cut-off function $\eta\in C^\infty(X)$ satisfying the following.

\begin{enumerate}

\item $0 \leq \eta\leq 1$, $\eta|_{\mathcal{K}}=1$ and $\eta=0$ on an open neighborhood of $E$.

\medskip

\item $$\int_X |\nabla \eta|_{\redn{\omega}}^2  \omega^n =  \redn{n}\int_X \sqrt{-1}\partial \eta \wedge \dbar \eta \wedge \omega^{n-\redn{1}} < \epsilon. $$

\end{enumerate}

\end{lemma}

\begin{lemma} \label{sobapp2} There exists $C=C(X, \theta_Y, n, A, K, \gamma)>0$ such that 
$$\left( \int_{\hat X} |u|^{2q} \omega^n\right)^{\frac{1}{q}} \leq C \left( \int_{\hat X} |\nabla u|^2_\omega \omega^n + \int_{\hat X} u^2  \omega^n \right)$$
for all $u\in W^{1,2}(\hat X, d, \omega^n) \cap L^\infty(\hat X) $. Here all the integrals are calculated on $X\setminus \cS_{X, \omega}$. 

\end{lemma}

\begin{proof} For any $\epsilon>0$, we choose a cut-off function $\eta_\epsilon$ as in Lemma \ref{cutoff} with support $\mathcal{K}_\epsilon\subset\subset X\setminus \cS_{X, \omega}$. We can assume that $\mathcal{K}_\epsilon \rightarrow X\setminus \cS_{X, \omega}$ increasingly as $\epsilon \rightarrow 0$. 
Then 
\begin{eqnarray} \label{w12app}
\int_X |\nabla (\eta_\epsilon   u)|^2_\omega \omega^n  &= & \int_X |\nabla\eta_\epsilon|^2_\omega u^2 \omega^n + \int_X (\eta_\epsilon)^2 |\nabla u|^2_\omega \omega^n + 2Re \int_X \eta_\epsilon u \redn{\langle} \nabla\eta_\epsilon,  \overline{\nabla} u \redn{\rangle_\omega} \omega^n \nonumber \\
&\leq &\epsilon \|u\|_{L^\infty}^2   + \int_X |\nabla u|^2_\omega \omega^n + 2\epsilon^{1/2} \|u\|_{L^\infty(X)} \Big( \int_X  |\nabla u|^2_\omega \omega^n  \Big)^{1/2} \nonumber \\
&\rightarrow& \int_X  |\nabla u|^2_\omega \omega^n, 
\end{eqnarray}
as $\epsilon \rightarrow 0$ since $\|u\|_{L^\infty(X)} < \infty$. 
Applying Lemma \ref{sobcps},  we have  
\begin{eqnarray*}
\left( \int_X |\eta_\epsilon u|^{2q} \omega^n\right)^{\frac{1}{q}} \leq C \left( \int_X |\nabla (\eta_\epsilon u)|^2_\omega \omega^n + \int_X (\eta_\epsilon u)^2  \omega^n \right) 
\end{eqnarray*}
 for some uniform $C>0$.  The lemma is proved by letting $\epsilon \rightarrow 0$. 
\end{proof}

\begin{theorem} \label{sobana} There exists $C=C(X, \theta_Y, n, A, K, \gamma)>0$ such that 
\begin{equation}\label{anasob2}
\left( \int_{\hat X} |u|^{2q} \omega^n\right)^{\frac{1}{q}} \leq C \left( \int_{\hat X} |\nabla u|^2_\omega \omega^n + \int_{\hat X} u^2  \omega^n \right)
\end{equation}
for all $u\in W^{1,2}(\hat X, d, \omega^n) $, where $q>1$ is the Sobolev exponent in Theorem \ref{thm:Sob}. 

\end{theorem}

\begin{proof} We first assume that $u\geq 0$. We define
$$u_A = \min(u, A)= \frac{u + A - |u-A| }{2} $$
for a constant $A\geq 0$.  Obviously, \redn{by monotone convergence theorem} we have 
$$\lim_{A\rightarrow \infty} \|u_A \|_{L^{2q}(X, \omega^n)} = \|u \|_{L^{2q}(X, \omega^n)}, ~\lim_{A\rightarrow \infty} \|u_A \|_{L^{2}(X, \omega^n)} = \|u \|_{L^{2}(X, \omega^n)} .$$
Also for any $f\in W^{1,2}(X, \omega)$, we have 
\begin{equation}\label{eqn:Kato}\int_{\hat X} |\nabla |f| |^2_\omega \omega^n \leq \int_{\hat X} |\nabla f|^2_\omega \omega^n. \end{equation}
Therefore 
$$
 \int_{\hat X} |\nabla u_A|^2 \omega^n \leq  \int_{\hat X} \left(  |\nabla u|^2_\omega + |\nabla |u-A\|^2_\omega \right) \omega^n 
\leq 2  \int_{\hat X}  |\nabla u|^2_\omega    \omega^n $$
by Kato's inequality \nref{eqn:Kato}.
Applying Lemma \ref{sobapp2} {to the bounded function $u_A$}, we have
$$\Big( \int_{\hat X} |u_A|^{2q} \omega^n\Big)^{\frac{1}{q}} \leq C \Big( \int_{\hat X} |\nabla u_A|^2_\omega \omega^n + \int_{\hat X} u_A^2  \omega^n \Big) \leq 2C \Big( \int_{\hat X} |\nabla u|^2_\omega \omega^n + \int_{\hat X} u_A^2  \omega^n \Big).$$
The Sobolev inequality is then proved by letting $A\rightarrow \infty$. \redn{
For general $u\in W^{1,2}(\hat X, \omega)$, we apply the Sobolev inequality just proved to $u_+$ and $u_-$, the positive and negative parts of $u$, respectively.} \redn{Therefore, }we have 
$$\Big( \int_{\hat X} |u|^{2q} \omega^n\Big)^{\frac{1}{q}}  \leq 2C \Big( \int_{\hat X} |\nabla u|^2_\omega \omega^n + \int_{\hat X} u^2  \omega^n \Big)$$ for all $u\in W^{1,2}(\hat X, d, \omega^n)$
and the theorem is proved.
\end{proof}

\noindent{{\bf Remark.} With the same proof as that in Theorem \ref{sobana}, using the Sobolev inequality \nref{eqn:Sob}, we also have the following version of Sobolev inequality for some constant $C>0$ that
\begin{equation}\label{eqn:Sob2}
\Big(\frac{1}{V_\omega}\int_{\hat X} | u - \bar u  |^{2q} \Big)^{\frac{1}{q}} \le \frac{C}{V_\omega} \int_{\hat X} |\nabla u|^2_\omega \omega^n
\end{equation}
for all $u \in W^{1,2}(\hat X, d, \omega^n)$, where $\bar u = \frac{1}{V_{\omega}} \int_{\hat X} u\, \omega^n$ is the average of $u$.
}
{\begin{lemma}
The following local Sobolev inequality holds on $(X,\omega)$. For any open subset $\Omega\subset X$, there is a constant $C>0$ such that
\begin{equation}\label{eqn:localSob1}
\Big( \frac{1}{V_\omega}\int_\Omega | u|^{2q} \omega^n  \Big)^{1/q} \le C\Big(1+ \frac{V_\omega(\Omega)}{V_\omega(\Omega^c)}\Big)  \frac{1}{V_\omega} \int_{\Omega} |\nabla u|^2_\omega \omega^n,
\end{equation}
for all $u\in W^{1,2}_0(\Omega)$.
\end{lemma}
\begin{proof}
The argument is similar to that in Theorem \ref{sobana} by making use of the inequality \nref{eqn:local S1}, so we only outline the proof. Taking $u_\epsilon = u \eta_\epsilon$ if necessary, we may assume $u\in C^1_0(\Omega\backslash \cS_{X,\omega})$, and identify it with $\pi^*u\in C^1_0(\pi^{-1}(\Omega))$. We can then apply the local Sobolev inequality \nref{eqn:local S1} to $u$ over $\pi^{-1}(\Omega)\subset Y$ with the approximating K\"ahler metrics $\omega_j$. Letting $j\to\infty$ gives \nref{eqn:localSob1}.
\end{proof}
As a corollary of Theorem \ref{sobana}, we have the following Sobolev embedding theorem.
\begin{corollary}\label{cor embedding}
For any $1\le \hat q< q$, the continuous embedding $W^{1,2}(X,\omega) \hookrightarrow L^{2 \hat q}(X,\omega)$ is compact. 
\end{corollary}
\begin{proof} Take a sequence of cut-off functions $\eta_j$ such that for each $j$, $\eta_j \le \eta_{j+1}$, ${\mathcal K_j} \subset {\mathcal K_{j+1}}$, $\cup_j {\mathcal K_j} = X\backslash \cS_{X,\omega}$, where ${\mathcal K}_j = {\mathrm{supp}}(\eta_j) \subset X\backslash \cS_{X,\omega}$ is the compact support of $\eta_j$. Moreover, 
$$\int_{X} (1-\eta_j)\omega^n \le \frac{1}{j}\to 0.$$ 
Such cut-off functions always exist because the measure $\omega^n$ does not charge mass on $\cS_{X,\omega}$. To prove Corollary \ref{cor embedding}, we need to show that for any sequence of functions $\{f_a\}_{a = 1}^\infty \subset W^{1,2}(X,\omega)$ with $\sup_a \| f_a\|_{W^{1,2}(X,\omega)}<\infty$, there exists a convergent subsequence in $L^{\hat q}(X,\omega)$. We observe first that by the Sobolev inequality in Theorem \ref{sobana},
\begin{equation}\label{eqn:uniform1}
M:=\sup_a \| f_a\|_{L^{2q}(X,\omega)} \le C\sup_a \| f_a\|_{W^{1,2}(X,\omega)}<\infty. 
\end{equation}
We next consider the sequence of functions $\{f_a \eta_{1}\}_{a}$ which are compactly supported in the smooth manifold $({\mathcal K}_2^{\circ},\omega)$ and uniformly bounded\footnote{The $W^{1,2}({\mathcal K}_2^{\circ},\omega)$-bound of $\{f_a \eta_{1}\}_{a}$ may depend on $|\nabla \eta_1|$, but it is independent of $a$.} in $W^{1,2}({\mathcal K}_2^{\circ},\omega)$. By the standard Sobolev precompact embedding theorem on smooth manifolds, we see that there is a subsequence $\{f_{a_1} \eta_1\}_{a_1\in {\mathcal A}_1}$ which converges in $L^{2\hat q}(X,\omega)$. We then argue similarly on the sequence of functions $\{f_{a_1}\eta_2\}_{a_1\in{\mathcal A}_2}$ which are uniformly bounded in $W^{1,2}({\mathcal K}_3^{\circ},\omega)$, hence there is a subsequence of functions $\{f_{a_2}\eta_{2}\}_{a_2\in {\mathcal A}_2}$ which converges in  $L^{2\hat q}(X,\omega)$. Continuing this way, for any $k$ we can find a sequence of functions $\{f_{a_k} \eta_k\}_{a_k\in {\mathcal A}_k}$ which converges in $L^{2\hat q}(X,\omega)$, and the index sets fulfill ${\mathcal A}_1\supset {\mathcal A}_2\supset \cdots \supset {\mathcal A}_k$. By the classical diagonal argument, we can pick a sequence of functions $\{f_{a_k}\}_{k=1}^\infty$ such that for any fixed $\ell$, the sequence $\{f_{a_k} \eta_\ell\}_{k=1}^\infty$ is  convergent in $L^{2\hat q}(X,\omega)$. We claim that the sequence $\{f_{a_k}\}_{k}$ is also convergent in $L^{2\hat q}(X,\omega)$. It suffices to show $\{f_{a_k}\}$ is a Cauchy sequence in $L^{2\hat q}(X,\omega)$. To see this, for any $\epsilon>0$ we take $\ell = \big[(4M \epsilon^{-1})^{2q\hat q/(q-\hat q)}\big]+ 1$. The sequence $\{f_{a_k} \eta_\ell\}_{k=1}^\infty$ is Cauchy in $L^{2\hat q}(X,\omega)$. Hence there is a $K = K(\epsilon,\ell)>0$ such that for any $k', k\ge K$, $\|f_{a_k} \eta_\ell - f_{a_{k'}} \eta_\ell\|_{\scriptscriptstyle L^{2\hat q}}\le \frac{\epsilon}{4} $. Then we have
\bea\nonumber
\| f_{a_k} - f_{a_{k'}}\|_{\scriptscriptstyle  L^{2\hat q}}&\le & \| f_{a_k} - f_{a_k}\eta_\ell\|_{\scriptscriptstyle L^{2\hat q}} + \| f_{a_k} \eta_\ell - f_{a_{k'}} \eta_\ell\|_{\scriptscriptstyle L^{2\hat q}} + \| f_{a_{k'}} - f_{a_{k'}}\eta_\ell\|_{\scriptscriptstyle L^{2\hat q}}\\
&\le & \| f_{a_k}\|_{\scriptscriptstyle L^{2q}} \| 1- \eta_\ell\|_{\scriptscriptstyle L^{\frac{q}{q-\hat q}}}^{1/2\hat q} +  \| f_{a_{k'}}\|_{\scriptscriptstyle L^{2q}} \| 1- \eta_\ell\|_{\scriptscriptstyle L^{\frac{q}{q-\hat q}}}^{1/2\hat q} + \frac{\epsilon}{4} \nonumber\\
&\le &\frac{\epsilon}{4} + \frac{\epsilon}{4} + \frac{\epsilon}{4} < \epsilon. \nonumber
\eea
This proves the claim and hence the Corollary.
\end{proof}

\begin{proposition} \label{proposition 7.1} The space $C^\infty_0(X\setminus \cS_{X, \omega})$ is  dense in $W^{1, 2}(\hat X, d, \omega^n)$.  In particular, $C^\infty(X)\cap W^{1,2}$ is dense in  $W^{1, 2}(\hat X, d, \omega^n)$ if we identify every $u\in C^\infty(X)$ as a smooth function on $\hat X$ by trivial extension of $u|_{X\setminus \cS_{X, \omega}}$ . 
\end{proposition}

\begin{proof} Let $u\in W^{1,2}(\hat X, d, \omega^n)$. For any $\epsilon>0$, we choose a cut-off function $\eta_\epsilon$ as in Lemma \ref{cutoff} with support $\mathcal{K}_\epsilon\subset\subset X\setminus \cS_{X, \omega}$. We can assume that $\mathcal{K}_\epsilon \rightarrow X\setminus \cS_{X, \omega}$ increasingly as $\epsilon \rightarrow 0$. 

We will first assume $u\in L^\infty(\hat X)$.  Let $u_\epsilon = \eta_\epsilon u$. Since $\omega$ is smooth on $X\setminus\cS_{X, \omega}$, there exists   $v_\epsilon \in C_0^\infty(X\setminus \cS_{X, \omega})$ such that
$$\|v_\epsilon - u_\epsilon \|_{W^{1,2}(\hat X, d, \omega^n)} < \epsilon.$$
Obviously, $$\lim_{\epsilon \rightarrow 0} \|u_\epsilon - u\|_{L^2(X, \omega^n)} = 0. $$
By the calculations in (\ref{w12app}), we have 
{\bea\nonumber
\lim_{\epsilon\rightarrow 0} \int_X |\nabla (u_\epsilon -u) |^2_\omega \omega^n \le \lim_{\epsilon\to 0} 2 \int_X (1-\eta_\epsilon)^2 |\nabla u|^2_\omega \omega^n + 2\| u\|_{L^\infty} \int_X |\nabla \eta_\epsilon|^2_\omega \omega^n = 0
\eea
}
%
Hence, we have
$$
{\lim_{\epsilon \rightarrow 0} \int_X ( |\nabla (u_\epsilon -u) |^2 + (u - u_\epsilon)^2) \omega^n = 0.}
$$
Immediately, we have $\lim_{\epsilon \rightarrow 0}\|v_\epsilon - u\|_{L^2(X, \omega^n)} = 0$ by the triangle inequality. 

In general, we can always replace $u$ by {$\tilde u_A= (u_+)_A - (u_-)_A$, where $f_A := \min(u, f)$ for a function $f$ and constant $A\geq 0$}. Then
\begin{equation}\label{eqn:uA}\int_X \chi_{\scriptscriptstyle \{|u|\ge A\}}\omega^n  = \int_{X \cap \{ |u|\geq A\} } \omega^n\le  \frac{1}{A^2} \int_X u^2 \omega^n \rightarrow 0, \quad \mbox{as }A\to\infty.
\end{equation} Hence $\chi_{\scriptscriptstyle \{|u|\ge A\}} \to 0$ in the a.e. sense on $(X,\omega^n)$. 
It then follows from dominated convergence theorem that for $u\in W^{1,2}(X,\omega)$
$$\int_X |u - \tilde u_A|^2 \omega^n = \int_X |u|^2 \chi_{\scriptscriptstyle \{|u|\ge A\}} \omega^n\to 0, $$ and 
$$\int_X |\nabla(u - \tilde u_A)|^2_\omega \omega^n \le \int_X |\nabla u|^2_\omega \chi_{\scriptscriptstyle \{|u|\ge A\}} \omega^n\to 0,$$
as $A\to \infty$. We then can apply the same argument earlier for $\tilde u_A\in L^\infty(\hat X)$. This completes the proof of the proposition.
%
\end{proof}

By applying similar arguments, we can also show that $W^{1,2}\cap C^\infty(X\setminus \cS_{X, \omega}) \cap L^\infty(\hat X)$ is also dense in $W^{1,2}(\hat X, d, \omega^n)$.  If we denote $W^{1,2}_0(X,\omega)$ the Hilbert space obtained by the completion of $C^{\infty}_0(X\backslash \cS_{X,\omega})$ under the norm $\| \cdot\|_{W^{1,2}(X,\omega)}$, then Proposition \ref{proposition 7.1} implies that 
$$W^{1,2}(X,\omega) = W^{1,2}_0(X,\omega). $$
%


\section{Noncollapsing and diameter bounds for $ \mathcal{AK}$-currents} \label{singdiasec}
\setcounter{equation}{0}

Let $(X, \theta_Y)$ be an $n$-dimensional projective normal variety equipped with a smooth K\"ahler metric $\theta_Y$. Let $\cS_{X, \omega}$ be the union of the singular set of $X$ and  $\log \frac{(\pi^*\omega)^n}{(\theta_Y)^n}$.  We let $(\hat X, d)$ be the metric completion of $(X\setminus \cS_{X, \omega}, \omega)$ since $\omega$ is a smooth K\"ahler metric on $X\setminus \cS_{X, \omega}$. The natural question is  \redn{whether} $(\hat X, d)$  is a compact metric space.

 We first verify that the distance function on $\hat X$ with a fixed base point  is Lipschitz. 
 
\begin{lemma} \label{badis} For any \redn{fixed point} $p\in \hat X$, the function 
$$f(x) = d(p, x)$$
is a $1$-Lipschitz function on $(\hat X, d)$. 

\end{lemma}

\begin{proof} We can assume $p\in X\setminus \cS_{X, \omega}$ since any point of $\hat X$ can always be approximated by \redn{a sequence of points in $X\setminus \cS_{X,\omega}$}. For any $x, y \in X\setminus \cS$, we have by the triangle inequality
$$  - d(x, y) \leq    f(x) - f(y) \leq d(x, y). $$
Then lemma immediately follows. 
\end{proof}


The volume measure $\omega^n = e^F (\theta_Y)^n$ can be uniquely extended to a volume measure to $\hat X$ by natural extension. Hence we will consider the metric measure space
$$(\hat X, d, \omega^n). $$
We also define the Sobolev space 
$$W^{1,2}(\hat X, d, \omega^n) =W^{1, 2}(X, \omega). $$
by extending $f\in W^{1,2}(X, \omega)$ from \redn{$X\setminus \mathcal{S}_{X,\omega}$} to $\hat X$ trivially. By Lemma \ref{badis}, for any $p\in \hat X$, the distance function $d(p, \cdot) \in \redn{W^{1, \infty}_{loc}(\hat X, d, \omega^n)}$. The Sobolev inequality in Theorem \ref{sobana} also holds for $\redn{W^{1,2}}(\hat X, d, \omega^n)$ with the integration calculated on \redn{$X\setminus \cS_{X,\omega}$}.

We will prove the following non-collapsing result for $(\hat X, d, \omega^n)$.

\begin{proposition}  \label{noncolloc}There exists $\kappa=\kappa(n, A, K, \gamma)>0$ such that for any point $p\in \hat X$ and $r\in (0, 1)$,  
we have
\begin{equation}\label{eqn:noncol}\frac{{\mathrm{Vol}}_\omega(B_d (p, r))}{V_\omega} \geq \min\{\frac{1}{2}, \kappa {r^{\frac{2q}{q-1}}}\}, \end{equation}
where $q>1$ is the Sobolev exponent in Theorem \ref{anasob} and  $B_d(p, r)$ is the geodesic ball of radius $r$ centered at $p$ with respect to $d$ and $\omega^n$. 

\end{proposition}
We remark that if $q=\frac{n}{n-1}$, then \nref{eqn:noncol} would be the expected
$$\frac{\mbox{Vol}_\omega(B_d (p, r))}{V_\omega} \geq \kappa {r^{2n}}. $$
\begin{proof} First we assume {$p\in X\setminus \cS_{X,\omega}$. For a given radius $r>0$, either $\frac{{\mathrm{Vol}}_\omega(B_d (p, r))}{V_\omega} \ge 1/2$ or $\frac{{\mathrm{Vol}}_\omega(B_d (p, r))}{V_\omega}\le 1/2$. We assume the latter inequality holds, so $\frac{{\mathrm{Vol}}_\omega(B_d (p, r))}{{\mathrm{Vol}}_\omega(X\backslash B_d (p, r))} \le 1 $. So by \nref{eqn:localSob1} the following local Sobolev inequality holds
\begin{equation}\label{eqn:localSob2}
\Big( \frac{1}{V_\omega}\int_{B_d(p,r)} | u|^{2q} \omega^n  \Big)^{1/q} \le  \frac{C}{V_\omega} \int_{B_d(p,r)} |\nabla u|^2_\omega \omega^n\end{equation}
for all $ u\in W^{1,2}_0(B_d(p,r))$.
We let $\rho: \mathbb{R} \rightarrow [0, \infty) $ be the standard cut-off function with $\rho=1$ on $(-\infty, 1/2)$ and $\rho=0$ on $[1, \infty)$. We let 
$$u(x)= \rho( d(p, x)/r).$$   
Then $u$ is a Lipschitz function with compact support in $B_d(p,r)$, and we also have
$\sup_{\hat X} |\nabla u|\leq 4 r^{-1}. $ Applying \nref{eqn:localSob2} to this function $u$, we obtain
$$
\Big(\frac{{\mathrm{Vol}}_\omega(B_d (p, r/2))}{V_\omega} \Big)^{1/q} \le C \frac{{\mathrm{Vol}}_\omega(B_d (p, r))}{r^2 V_\omega},
$$
for some uniform constant $C>0$ independent of $r$. This is equivalent to 
\begin{equation}\label{eqn:iteration1}
c_0\Big(\frac{{\mathrm{Vol}}_\omega(B_d (p, r/2))}{(r/2)^{\frac{2q}{q-1}} V_\omega  }\Big)^{1/q} \le  \frac{{\mathrm{Vol}}_\omega(B_d (p, r))}{r^{\frac{2q}{q-1}}V_\omega},
\end{equation} 
for some uniform constant $c_0>0$. We now apply \nref{eqn:iteration1} to the radii $r_m = 2^{-m} r$ with $m=0,1,2,\cdots,\ell$ and iterate the resulted inequalities to obtain 
\begin{equation}\label{eqn:iteration2}
c_0^{\sum_{m=0}^\ell q^{-m}}\Big(\frac{{\mathrm{Vol}}_\omega(B_d (p, 2^{-\ell-1} r))}{(2^{-\ell-1} r)^{\frac{2q}{q-1}} V_\omega  }\Big)^{1/q^{\ell + 1}} \le  \frac{{\mathrm{Vol}}_\omega(B_d (p, r))}{r^{\frac{2q}{q-1}}V_\omega}.
\end{equation} 
We observe that the left-hand side of \nref{eqn:iteration2} satisfies $\lim_{\ell\to\infty} c_0^{\sum_{m=0}^\ell q^{-m}} = c_0^{\frac{q}{q-1}}$, and
{\small
\bea\nonumber
&& \lim_{\ell \to \infty} \Big(\frac{{\mathrm{Vol}}_\omega(B_d (p, 2^{-\ell-1} r))}{(2^{-\ell-1} r)^{\frac{2q}{q-1}} V_\omega  }\Big)^{\frac{1}{q^{\ell+1}}}\\
 & = \nonumber& \lim_{\ell\to\infty} \Big(\frac{{\mathrm{Vol}}_\omega(B_d (p, 2^{-\ell-1} r))}{(2^{-\ell-1} r)^{2n} V_\omega  }\Big)^{\frac{1}{q^{1+\ell}}}  r^{(2n - \frac{2q}{q-1})/q^{\ell+1}} 2^{-(\ell+1)(2n - \frac{2q}{q-1})/q^{\ell+1}} = 1,
\eea
since $p$ is chosen to be a smooth point where $\omega$ is locally smooth, hence when $\ell$ is sufficiently large the metric ball $B_d(p,2^{-\ell + 1} r)$ is approximately Euclidean, so 
$$ \lim_{\ell\to\infty} \frac{{\mathrm{Vol}}_\omega(B_d (p, 2^{-\ell-1} r))}{(2^{-\ell-1} r)^{2n} V_\omega  } = \frac{c_n}{V_\omega},$$ where $c_n$ is the volume of unit ball in ${\mathbb C}^n$.
}
}

In general, if $p\in \hat X$, for any $r>0$ and $0<\epsilon<< r$, there exists $p_\epsilon\in X\setminus \cS_{X,\omega}$ such that
$$d(p, p_\epsilon)< \epsilon$$ and 
$$ \frac{Vol(B(p, r)) }{ r^{2q/(q-1)}} \geq \frac{Vol(B(p_\epsilon, r-2\epsilon)) }{ r^{2q/(q-1)}}\geq \kappa \Big( \frac{r-2\epsilon}{r} \Big)^{2q/(q-1)} \rightarrow \kappa$$
as we take $\epsilon \rightarrow 0$. We have completed the proof of the proposition. 
\end{proof}

The non-collapsing estimate in Proposition \ref{noncolloc} immediately implies the following diameter estimate.

\begin{corollary} There exists $C=C(n, A,p,K, \gamma)>0$ such that if $\omega \in \mathcal{AK}(n, A,p,K, \gamma)$, 
$$diam(\hat X, d) \leq C. $$
In particular,  $(\hat X, d, \omega^n)$ is a compact metric space.

\end{corollary}


\section{Heat kernel estimates for $ \mathcal{AK}$-currents} \label{section heat}
\setcounter{equation}{0}

Let $X$ be an $n$-dimensional normal K\"ahler variety. Let $\omega \in \mathcal{AK}(X, \theta_Y, n, A, p, K, \gamma)$. Since $\omega \in C^\infty(X\setminus \cS_{X, \omega})$, we would like to define the heat kernel of the Laplacian $\Delta_\omega$  by the following parabolic equation 
$$\ddt{} H(x, y, t) = \Delta_{\omega, y} H(x, y, t), ~ \lim_{t\rightarrow 0^+} H(x, y, t) = \delta_x(y)$$
for $x, y \in Y^\circ = \pi^{-1}(X\setminus \cS_{X, \omega})$. However, we have to show the existence and uniqueness of $H(x, y, t)$ and how it can be extended naturally to $\hat X \times \hat X \times [0, \infty)$.

We will apply Proposition \ref{omapp} with $\{ \omega_j\}_{j=1}^\infty $ as the approximating K\"ahler metrics for $\pi^*\omega$ on $Y$, the suitable blow-up of $X$ as in (\ref{logres}). 
We will then consider the heat kernel $H_j(x, y, t)$ \redn{associated to $\omega_j$} on $Y\times Y\times [0, \infty)$. Theorem \ref{thm:heat} immediately implies the following lemma. 
\begin{lemma}\label{heatapp} There exists $C>0$ such that
$$\sup_{x,y\in Y}H_j(x, y, t) \leq \frac{1}{V_{\omega_j}} + \frac{C}{V_{\omega_j}} t^{-\frac{q}{q-1}}. $$
\end{lemma}
\begin{proof}
This lemma follows from the diagonal estimates for heat kernel in Theorem \ref{thm:heat} and the elementary inequality:
$$H_j(x,y,t ) \le \frac{1}{2} H_j(x,x,t) + \frac{1}{2} H_j(y,y,t).$$
\end{proof}
%
\redn{ In particular, for any $\epsilon>0$, there exists $C=C(\epsilon)>0$ such that for all $j \geq 1$, 
$$\sup_{Y\times Y\times[\epsilon, \infty)} H_j(x, y, t) \leq C. $$
}

The following lemma immediately follows by the standard interior estimates for linear parabolic equations as $\omega_j$ converges smoothly to $\omega$ on $Y^\circ$. 

\begin{lemma} For any $\epsilon>0$, $k>0$ and $\cK\subset\subset Y^\circ$, there exists $C>0$ such that for all $j\geq 1$, 
$$ \| H_j(x, y, t)\|_{C^k(\cK \times \cK \times [\epsilon, \infty))} \leq C,$$
\redn{where the $C^k$ norm of $H_j$ is with respect to the fixed metric $\pi^*\omega|_{\mathcal K}$.}
\end{lemma}

\begin{lemma} \label{lemma 9.3} For any $x\in Y^\circ$, any open neighborhood $ U_x \subset Y^\circ$ of $x$, there exist $\delta>0$ and $C>0$ such that for all $j\geq 1$,
\begin{equation}\label{eqn:heatk1}
\sup_{y\in Y\backslash U_x}H_j(x, y, t) < C t^{-\frac{q}{q-1}} e^{-\frac{\delta}{t}} \chi_{\scriptscriptstyle (0,A]}(t) + C e^{-\frac{\delta}{t}} \chi_{\scriptscriptstyle(A, \infty)}(t).
\end{equation}
\end{lemma}
\begin{proof} Since $x\in Y^\circ$, there exists $d_0>0$ such that for all $j\geq 1$ and $y\in Y \setminus U_x$, 
$$d_{\omega_j}(x, y) > d_0,$$ by the smooth convergence of $\omega_j|_{U_x}$ to $\pi^*\omega|_{U_x}$ if $U_x\subset Y^\circ$ is chosen sufficiently small.  
The lemma then follows from the uniform estimates in Theorem \ref{thm:heat}. 
\end{proof}

By taking a subsequence, we can assume that $H_j(x, y, t)$ converges \redn{locally} smoothly  on $$\left( Y^\circ \times Y^\circ \right) \times(0, \infty).$$ 
%
%
We let $H_\infty (x, y, t)$ be the limit of $H_j(x, y, t)$ defined on $(Y^\circ \times Y^\circ)\times \redn{(}0, \infty)$, after further taking a subsequence.  \redn{$H_\infty$} can be naturally extended to a locally \redn{(in time $t$)} bounded function on $Y\times Y\times (0, \infty)$.  

Straightforward arguments will give the following lemma by the construction of $H_\infty$ and the uniform estimates of $H_j$.  
%
\begin{lemma} \label{heatrep1} Let $H_\infty (x, y, t)$ be a sub-sequential  limit of $H_j(x, y, t)$. Then the following hold:
\begin{enumerate}
\item \label{item 1} $H_\infty \in C^\infty( (Y^\circ \times Y^\circ) \times (0, \infty)). $
%
%
\item\label{item 2} $H_\infty$ satisfies the same upper bounds as $H_j$ in Lemma \ref{heatapp}, i.e. 
\begin{equation}
\label{eqn:H inf}\sup_{x,y\in Y^\circ}H_\infty(x, y, t) \leq \frac{1}{V_{\omega}} + \frac{C}{V_{\omega}} t^{-\frac{q}{q-1}}. 
\end{equation}
\item\label{item 3} For all $t>0$ and $x\in Y^\circ$, 
$$\int_Y H_\infty (x, y, t) \omega^n(y) = 1. $$ 
\item \label{item 4}For any $x, y\in Y^\circ$, 
$$\lim_{t\rightarrow 0^+} H_\infty(x, y, t) = \delta_x(y). $$
\item\label{item 5} for any $(x, y, t) \in (Y^\circ \times Y^\circ) \times (0, \infty)$, 
$$\ddt{} H_\infty (x, y, t) = \Delta_{\omega, y} H_\infty (x, y, t). $$
\item \label{item 6}For a fixed $x\in Y^\circ$ and $t>0$,  the function $y\mapsto H_\infty(x,y, t)\in W^{1,2}(Y^\circ,\omega)$ and the following hold:
\begin{equation}\label{eqn:heat doc sing}
\int_{Y^\circ} |\nabla H_\infty(x,y,t)|_\omega^2 \omega^n \le  \frac{C}{V_\omega}  t^{-1 - \frac{q}{(q-1)}},
\end{equation}
\begin{equation}\label{eqn:heat dot2}
\sup_{x,y\in Y^\circ}|\dot H_\infty(x,y,t)| \le \frac{C}{V_\omega} t^{-1 - \frac{q}{q-1}}.
\end{equation}
Furthermore, $H_\infty$ satisfies the following semi-group property 
\begin{equation}\label{eqn:semigroup}
\int_{Y^\circ} H_\infty(x,z,t) H_\infty(y,z,s) \omega^n(z) = H_\infty(x,y, t+s),
\end{equation}
for all $x$,$y\in Y^\circ$ and for all $s$, $t >0$.
\end{enumerate}
\end{lemma}
\begin{proof}
Items \nref{item 1}, \nref{item 2} and \nref{item 5} are immediate consequences of locally smooth convergence of $H_j$ to $H_\infty$ over $(Y^\circ\times Y^\circ)\times (0,\infty)$. To prove item \nref{item 3}, for any $\epsilon>0$ we take a small open neighborhood $U$ of $Y\backslash Y^\circ$ such that $\sup_j \int_{U} \omega_j^n < \epsilon.$ By the uniform convergence of $H_j(x,y,t)$ over $y\in Y^\circ\backslash U$, we have 
\bea\nonumber 1 \ge \int_{Y^\circ \backslash U} H_\infty(x,y,t) \omega^n(y) &=& \lim_{j\to\infty} \int_{Y^\circ \backslash U} H_j(x,y,t) \omega_j^n(y)\\
& = \nonumber &1 - \lim_{j\to\infty} \int_{ U} H_j(x,y,t) \omega_j^n(y)\\
&\ge & 1- \big(\frac{1}{V_{\omega}} + \frac{C}{V_{\omega}} t^{-\frac{q}{q-1}} \big) \epsilon \to 1 ,\nonumber
\eea
if we let $U$ shrink to $Y\backslash Y^\circ$ and then $\epsilon\to 0$. This proves item \nref{item 3}. For item \nref{item 4}, it suffices to show that for any test function $\rho\in C^\infty_0(Y^\circ)$, 
\begin{equation}\label{eqn:Dirac}
\lim_{t\to 0^+} \int_{Y^\circ} H_\infty(x,y,t) \rho(y) \omega^n(y) = \rho(x),
\end{equation}
for all  $ x\in Y^\circ$. Note that the function $u_j(x,t) = \int_{Y} H_j(x,y,t)\rho(y)\omega_j^n(y)$ satisfies $$\Big|\frac{\partial u_j}{\partial t}\Big| = \Big|\int_Y H_j(x,y,t) \Delta_{\omega_j} \rho(y) \omega_j^n(y)\Big| \le C,$$ where we have used $|\Delta_{\omega_j} \rho|\le C$ since $\rho$ is compactly supported in $Y^\circ$. Integrating this inequality gives $|u_j(x,t) - \rho(x)|\le C t$. It is clear that $u_j(x,t)\to \int_{Y^\circ} H_\infty(x,y,t) \rho(y)\omega^n(y)$ as $j\to\infty$. The equation in \nref{eqn:Dirac} follows easily. To see the last item \nref{item 6}, multiplying both sides of the heat equation $\partial_t H_j(x,y,t) = \Delta_{\omega_j} H_j(x,y,t)$ by $H_j(x,y,t)$ and applying integration by parts, we get
\bea \nonumber
\int_Y |\nabla H_j(x,y,t)|^2_{\omega_j} \omega_j^2 & = & - \int_Y \dot H_j(x,y,t) \tilde H_j(x,y,t) \omega^n\\
&\le \nonumber & \| \dot H_j(x,y,t)\|_{L^2(\omega_j)}\tilde H_j(x,x,2t)^{1/2}\\
&\le\label{eqn:heat dot 1} & \frac{C}{V_\omega} t^{-1 - \frac{q}{(q-1)}}.
\eea 
The inequality \nref{eqn:heat doc sing} follows by letting $j\to\infty$ in \nref{eqn:heat dot 1} and using Fatou's lemma. The semi-group property of $H_\infty$ in \nref{eqn:semigroup} follows by similar argument using the uniform upper bound of $H_j$.  The estimate \nref{eqn:heat dot2} follows from \nref{eqn:heat dot1} and locally smooth convergence of $\dot H_j$ to $\dot H_\infty$.
\end{proof}
We remark that \nref{eqn:Dirac} also holds for any $\rho\in C(Y^\circ)$. To see this, we first consider $\rho_\epsilon = \rho \eta_\epsilon\in C_0(Y^\circ)$ such that $x\in {\mathrm{supp}(\eta_\epsilon)}$ and $\eta_\epsilon(x) = 1$. Then we take a sequence of smooth functions $u_j\in C^\infty_0(Y^\circ)$ such that $\|u_j - \rho_\epsilon\|_{L^\infty} \le j^{-1}$. It then follows that
{\small
$$
u_j(x) = \lim_{t\to 0^+} \int_{Y^\circ} H_\infty(x,y,t) u_j(y) \omega^n(y) = \lim_{t\to 0^+} \int_{Y^\circ} H_\infty(x,y,t) \rho_\epsilon(y) \omega^n(y) + O(j^{-1}).
$$ 
} Letting $j\to\infty$ and then $\epsilon\to 0$ gives the equation \nref{eqn:Dirac} for $\rho\in C(Y^\circ)$.

%

\begin{corollary} Suppose there exists another heat kernel $\hat H(x,y,t)$ satisfying the conditions in Lemma \ref{heatrep1}.  Then we have 
$$\hat H(x, y, t) = H_\infty (x, y, t)$$
for all $x,y \in Y^\circ$.
Hence the heat kernel $H_\infty$ is unique in this sense.
\end{corollary}
\begin{proof}
Take  cut-off functions $\eta_\epsilon$ as in Lemma \ref{cutoff}. Fix a time $T>0$. We claim that the function $$t\in (0,T)\mapsto \phi(t): = \int_{Y^\circ} H_\infty(x,z,t) \hat H(z, y, T-t) \omega^n(z)$$ is independent of $t$. Denote $\phi_\epsilon(t): = \int_{Y^\circ} \eta_\epsilon(z)H_\infty(x,z,t) \hat H(y, z, T-t) \omega^n(z)$. By integration by parts and using the heat equation, we have for any small $\delta>0$ and $t\in [\delta, T-\delta]$, there is a constant $C>0$ depending additionally on $\delta>0$ such that
$$
|\phi_\epsilon'(t)|  =  \Big| \int_{Y^\circ} (-\langle \nabla \eta_\epsilon, \nabla H_\infty\rangle \hat H + \langle \nabla \eta_\epsilon, \nabla \hat H \rangle H_\infty)\omega^n\Big| \le C \epsilon^{1/2},
$$
by using the H\"older's inequality together with the upper bounds of $\| \nabla H_\infty\|_{L^2}$ (and $\| \nabla \hat H\|_{L^2}$) and $H_\infty$ (and $\hat H$) in \nref{eqn:heat doc sing} and \nref{eqn:H inf}, respectively. Integrating this over $[t_1, t_2] \subset [\delta, T-\delta]$ gives 
$$|\phi_\epsilon(t_1) - \phi_\epsilon(t_2)| \le C  \epsilon^{1/2}.$$ Letting $\epsilon\to 0$ yields $\phi(t_1) = \phi(t_2)$. Since $\delta>0$ is an arbitrary constant, this easily implies that $\phi(t)$ remains a constant for $t\in (0,T)$. By the Dirac property \nref{eqn:Dirac} of $H_\infty$ and $\hat H$, we see that $\phi(t)\to \hat H(x,y, T)$ as $t\to 0^+$ and $\phi(t)\to H_\infty(x,y, T)$ as $t\to T^-$. This finishes the proof of the uniqueness of the heat kernel.
\end{proof}

\section{A spectral theorem for normal K\"ahler spaces}\label{section 10}
\setcounter{equation}{0}

Let $(X, \theta_Y)$ be a normal \redn{K\"ahler} variety of $\dim X =n$, equipped with a smooth K\"ahler metric $\theta_Y$. Let $\omega\in \mathcal{AK}(X, \theta_Y, n, A, p, K, \gamma)$. We consider the Laplace operator $\Delta=\Delta_\omega$ with respect to $\omega$ on $X^\circ:=X\setminus \cS_{X,\omega}$. The following lemma is a simple consequence of the integration by parts.

\begin{lemma} The linear operator $-\Delta: C^\infty_0(X^\circ) \to C^\infty_0(X^\circ)$ is self-adjoint, i.e.  for any $u,v\in C^\infty_0(X^\circ)$, $\int_X (- \Delta  u) v \omega^n = \int_X (-\Delta  v) u \omega^n $.

\end{lemma} 
Note $-\Delta: L^2(X,\omega) \to L^2(X,\omega)$ is a linear and densely defined operator. The associated bilinear quadratic form of $ - \Delta$ is given by $Q(u,u) = \int_X |\nabla u|_\omega^2 \omega^n$, for any $u \in C^\infty_0(X^\circ)$. Let ${\mathcal D}(Q)$ be the completion of $C^\infty_0(X^\circ)$ in $L^2(X,\omega)$ under the norm $\| u\|_{L^2} + Q(u,u)^{1/2}$, which is equivalent to the $W^{1,2}(X, \omega)$-norm of $u$. Hence by Proposition \ref{proposition 7.1}, ${\mathcal {D}}(Q) = W_0^{1,2}(X,\omega)  = W^{1,2}(X,\omega)$. We then let $(-\Delta)^{F}: L^2(X,\omega)\to L^2(X,\omega)$ be the {\em Friedrich extension} of $-\Delta$, which is  a self-adjoint, densely defined and nonnegative operator. We note that the domain of $(-\Delta)^F$ is a subspace of ${\mathcal {D}}(Q)= W^{1,2}(X,\omega)$  (c.f. \cite{MM}). To ease notations, we will omit the superscript $F$ and simply write $(-\Delta)^F$ as $-\Delta$. 

\begin{theorem}
There exists a countable set $\Sigma\subset {\mathbb R}_{\ge 0}$ such that the equation
$$-\Delta u = \lambda u,\quad\mbox{with}\quad  u\in W^{1,2}(X,\omega)$$ admits a unique solution iff $\lambda \not\in \Sigma$. Furthermore, we can write $\Sigma = \{\lambda_k\}$ with \begin{equation}\label{eqn:eigen}0 = \lambda_0< \lambda_1\le \lambda_2\le\cdots.\end{equation}
\end{theorem}
\begin{proof}
We consider the linear operator $Lu = -\Delta u + u$. The associated quadratic bilinear form $Q_1(u,u)$ of $L$ satisfies
$$ Q_1(u,u) = Q(u,u) + \| u\|^2_{L^2(X,\omega)} =  \| u\|^2_{W^{1,2}(X,\omega)}.$$ Hence $Q_1$ is bounded and coercive on the Hilbert space $W^{1,2}(X,\omega)$. By the Lax-Milgram Theorem, the inverse $L^{-1}$ of $L$ exists, and it is a continuous and injective map from $(W^{1,2})^\dagger$ to   $W^{1,2}$. For any $u\in W^{1,2}$, we define $I: W^{1,2} \to (W^{1,2})^\dagger$ by $Iu(v) = \int_X u v \omega^n$, which can be decomposed as $I = I_1 \circ I_2$, where $I_2: W^{1,2} \to L^2(X,\omega)$ is the natural embedding, and $I_1: L^2(X,\omega)\to (W^{1,2})^\dagger$ is defined similarly as $I$. By the compact embedding theorem in Corollary \ref{cor embedding}, $I_2$ is compact. It is also clear that $I_1$ is continuous, hence $I$ is compact. If we denote $T = L^{-1}\circ I, $ then we have
\bea
\nonumber
\lambda \in \Sigma & \Longleftrightarrow & (-\Delta - \lambda I) \mbox{\, is not injective}\nonumber\\
& \Longleftrightarrow & (L - (\lambda + 1) I) \mbox{ is not injective}\nonumber\\
\mbox{(by left multiplication of $L^{-1}$) }& \Longleftrightarrow & (Id - (\lambda + 1) T) \mbox{ is not injective}\nonumber\\
& \Longleftrightarrow & \lambda+1 \neq 0 \, \mbox{and $\frac{1}{1+\lambda}\in \sigma_p(T)$},\nonumber
\eea
where $\sigma_p(T)$ denotes the set of eigenvalues of the {\em compact operator} $T$, which is countable and each number in $\sigma_p(T)$ has finite multiplicity by the Fredholm Alternative. This shows that $\Sigma$ is a discrete countable set, and its elements correspond to the eigenvalues of $-\Delta$, which can be listed as in \nref{eqn:eigen} since $-\Delta$ is nonnegative operator. The first eigenvalue $\lambda_0 = 0$ is simple since any harmonic function $u_0\in W^{1,2}(X,\omega)$ must be constant by Lemma \ref{lemma 10.2} below. 
\end{proof}


\begin{lemma} \label{lemma 10.2}Suppose $\lambda\ge 0$ is an eigenvalue of $-\Delta$ and $f\in W^{1,2}(X)$ is an eigenfunction associated to $-\Delta$, i.e. 
$-\Delta f = \lambda f$. Then $f\in L^\infty(X) \cap C^\infty(X^\circ)$. In particular, when $\lambda = 0$, $f$ must be  constant.

\end{lemma}
\newcommand{\innpro}[1]{\langle #1 \rangle}
\begin{proof}
Given $f\in W^{1,2}$ as above, by the standard elliptic regularity theory we know $f\in C^\infty(X\backslash\cS_{X,\omega})$. We consider two cases: $\lambda>0$ and $\lambda =0$. 

\smallskip

\noindent{\em Case 1:} $\lambda >0$. We normalize $f$ such that $\int_X f^2 \omega^n = 1$. Using the cut-off functions $\eta_\epsilon$ we have as $\epsilon\to 0$
$$\lambda |\int_X \eta_\epsilon f \omega^n |= |\int_X \innpro{\nabla \eta_\epsilon,\nabla f}\omega^n| \le \| \nabla f\|_{L^2} \| \nabla \eta_\epsilon\|_{L^2} \to 0.$$ Hence we see that $\int_X f\omega^n = 0$. 
 We consider the function $$u(x,t) = \int_X H_\infty(x,y,t) f(y) \omega^n(y)$$ for all $t>0$. 
By the boundedness of $\dot H_\infty(x,y,t)$ as in \nref{eqn:heat dot2}, we have 
\begin{equation}\label{eqn:heateigen0}\dot u(x,t) = \int_X \dot H_\infty(x,y,t) f(y) \omega^n(y) =  \int_X \Delta_y H_\infty(x,y,t) f(y) \omega^n(y). \end{equation}
By the Dirac property \nref{eqn:Dirac} of $H_\infty(x,y,0)$, we see that $u(x,0) = \lim_{t\to 0^+} u(x,t) = f(x)$. 
For the cut-off functions $\eta_\epsilon$ in Lemma \ref{cutoff}, we consider 
{\small
\bea \nonumber
 \int_X \eta_\epsilon \Delta_y H_\infty f \omega^n & = & - \int_X \innpro{\nabla \eta_\epsilon, \nabla H_\infty} f - \int_X \eta_\epsilon \innpro{\nabla H_\infty, \nabla f}\\
  & = \notag& - \int_X \innpro{\nabla \eta_\epsilon, \nabla H_\infty} f + \int_X H_\infty \innpro{\nabla \eta_\epsilon, \nabla f} +  \int_X H_\infty  \eta_\epsilon \Delta f \\
 & = \label{eqn:heateigen1}& - \int_X \innpro{\nabla \eta_\epsilon, \nabla H_\infty} f + \int_X H_\infty \innpro{\nabla \eta_\epsilon, \nabla f}  - \lambda  \int_X H_\infty  \eta_\epsilon f.
\eea 
}
The last integral in \nref{eqn:heateigen1} tends to $-\lambda \int_X H_\infty f$ since $H_\infty$ is $L^\infty$ bounded by \nref{eqn:H inf} and $f\in L^2(X,\omega)$. The second integral in \nref{eqn:heateigen1} is bounded by $\| H_\infty\|_{L^\infty} \| \nabla f\|_{L^2} \| \nabla \eta_\epsilon\|_{L^2} \to 0$ as $\epsilon\to 0$. The first integral in \nref{eqn:heateigen1} satisfies the following:
\bea\notag
\Big|- \int_X \innpro{\nabla \eta_\epsilon, \nabla H_\infty} f\Big| & = &\Big|\int_X (1-\eta_\epsilon) (\Delta H_\infty ) f + \int_X (1-\eta_\epsilon) \innpro{\nabla H_\infty, \nabla f}\Big|\\
&\le & \label{eqn:heateigen2} \| \Delta H_\infty\|_{L^\infty} \| f\|_{L^1(\Omega_\epsilon)} + \| \nabla H_\infty\|_{L^2(\Omega_\epsilon)} \| \nabla f\|_{L^2(\Omega_\epsilon)},
\eea
where $\Omega_\epsilon = X\backslash \{\eta_\epsilon = 1\}$ and by the construction $\eta_\epsilon$, $\int_{\Omega_\epsilon} \omega^n \to 0$ as $\epsilon\to 0$. Since $\| \nabla H_\infty\|_{L^2(X)}$ is finite by \nref{eqn:heat doc sing}, $\| \Delta H_\infty\|_{L^\infty}$ is also bounded by \nref{eqn:heat dot2},    and $f$ lies in $W^{1,2}(X,\omega)$ by assumption, we see that  the RHS of \nref{eqn:heateigen2} tends to zero as $\epsilon\to 0$, by the absolute continuity of integrals. Combining all the above, letting $\epsilon\to 0$ on both sides of \nref{eqn:heateigen1}, we obtain
$\int_X \Delta H_\infty f \omega^n = - \lambda \int_X H_\infty f \omega^n$. This together with \nref{eqn:heateigen0} implies that (because $\int_X f = 0$ we can replace $H_\infty$ by $\tilde H_\infty  = H_\infty - \frac{1}{V_\omega}$)
$$\dot u(x,t) = - \lambda \int_X \tilde H_\infty f \omega^n = -\lambda u(x,t)$$
for all $t>0$.
Solving this ODE, we obtain $u(x,t) = e^{-\lambda t} f(x)$. Therefore, for any $x\in X^\circ$, we have by the upper bound \nref{eqn:H inf} of $H_\infty$
\begin{equation}\label{eqn:heateigen4}
|f(x)|\le e^{\lambda t} \tilde H_\infty(x,x,2t) ^{1/2}\| f\|_{L^2} \le e^{\lambda t} C t^{-\frac{q}{2(q-1)}}.
\end{equation}
We take $t = 1/\lambda$ and then we get
\begin{equation}\label{eqn:L heat}\| f\|_{L^\infty} \le C \lambda^{\frac{q}{2(q-1)}}.\end{equation}

\smallskip

\noindent {\em Case 2:} $\lambda = 0$. Replacing $f$ by $f - \frac{1}{V_\omega}\int_X f$ if necessary, we assume $\int_X f \omega^n = 0$. The estimate \nref{eqn:heateigen4} still holds in this case, and thus we  obtain 
$$\| f\|_{L^\infty} \le C t^{-\frac{q}{2(q-1)} }\| f\|_{L^2}$$
for all $t>0.$ Letting $t\to \infty$ yields that $f\equiv 0$. Hence any harmonic function in $W^{1,2}(X,\omega)$ must be constant.
\end{proof}
If we let $\{\phi_k\}_{k=0}^\infty$ be the eigenfunctions of $-\Delta $ associated to the eigenvalues $\{\lambda_k\}$ as in \nref{eqn:eigen}. We also assume $\{\phi_k\}$ are {\em orthonormal}, i.e. $(\phi_k, \phi_{k'})_{L^2} = \delta_{k k'}$. Lemma \ref{lemma 10.2} implies $\phi_k\in L^\infty(X)$ for each $k$. The following is a well-known fact about the set of eigenfunctions.

\begin{lemma} \label{lemma 10.3}
The eigenfunctions $\{\phi_k\}_{k}$ form an orthonormal basis for the Hilbert space $L^2(X,\omega)$. 
\end{lemma}
\begin{proof}
We let $W$ be the closure of the subspace spanned by $\{\phi_k\}_{k=0}^\infty$ in $L^2(X,\omega)$. If $W$ is a proper subspace of $L^2(X,\omega)$, then its orthogonal complement $W^{\perp}$ is a nonempty closed subspace, and it is a simple fact that $\Delta|_{W^\perp}: W^{\perp} \to W^{\perp}$ is also densely defined and  self-adjoint. Indeed, if $u\in W^{\perp}$ is in the domain of $\Delta|_{W^\perp}$, then by definition $u\in W^{1,2}(X,\omega)$, and for any $k$ 
\begin{equation}\label{eqn:orth1}\int_X (\Delta u) \eta_\epsilon \phi_k  =  \int_X u (\eta_\epsilon \Delta \phi + \phi_k \Delta \eta_\epsilon + 2 \langle \nabla \phi_k, \nabla \eta_\epsilon\rangle  ).
\end{equation}
We observe that 
\bea\nonumber
\Big|\int_X u  \langle \nabla \phi_k, \nabla \eta_\epsilon\rangle  )\Big| & =& \Big| \int_X (1-\eta_\epsilon) \big( u\Delta\phi_k + \langle \nabla u, \nabla \phi_k\rangle \big)\Big|\\
&\le & \lambda_k \| u\|_{L^2(\Omega_\epsilon)} \| \phi_k\|_{L^2(\Omega_\epsilon)} +\| \nabla u\|_{L^2(\Omega_\epsilon)} \|\nabla \phi_k\|_{L^2(\Omega_\epsilon)}\to 0, \mbox{ as } \epsilon\to 0, \label{eqn:10.34}
\eea
where, as in the proof of Lemma \ref{lemma 10.2}, $\Omega_\epsilon$ is the complement of the set $\{\eta_\epsilon  = 1\}$ and by the choice of $\eta_\epsilon$, $\int_{\Omega_\epsilon} \omega^n\to 0 $ as $\epsilon \to 0$. We also have 
$$\int_X u \phi_k \Delta\eta_\epsilon = - \int_X u \langle\nabla\phi_k, \nabla \eta_\epsilon \rangle - \int_X \phi_k \langle \nabla u, \nabla \eta_\epsilon \rangle.$$ The first integral on the RHS tends to zero by \nref{eqn:10.34} and so does the second integral by the fact that $\phi_k\in L^\infty$, $u\in W^{1,2}$ and $\| \nabla \eta_\epsilon\|_{L^2} \to 0$. Hence letting $\epsilon\to 0$ on both sides of \nref{eqn:orth1} yields that $\int_X( \Delta u) \phi_k = \int_X u (\Delta \phi_k) = - \lambda_k \int_X u \phi_k = 0$ since $u\in W^\perp$. This shows that $\Delta u\in W^\perp$.

This will imply that $ - \Delta|_{W^\perp} $ has at least an eigenvalue $\lambda'$, which is also an eigenvalue of $- \Delta$.  This is a contradiction, and therefore, $W = L^2(X,\omega)$, in other words, the set of eigenfunctions $\{\phi_k\}_k$ forms an orthonormal basis of $L^2(X,\omega)$.
\end{proof}

We are now ready to prove the eigenvalue estimates for the operator $-\Delta$.
\begin{theorem}
The following expansion formula holds for the heat kernel $H_\infty(x,y,t)$: 
\begin{equation}\label{eqn:heat last expansion}
H_\infty(x,y,t) = \sum_{k=0}^\infty e^{-\lambda_k t} \phi_k(x) \phi_k(y),\quad \mbox{ on }X^\circ \times X^\circ \times (0,\infty).
\end{equation}
As a consequence, the eigenvalues $\lambda_k$ satisfy
\begin{equation}
\lambda_k \ge C k^{\frac{q-1}{q}} 
\end{equation}
for all $k\ge 0$.
\end{theorem}
\begin{proof}
For each fixed point $(x,t)\in X^\circ \times (0,\infty)$, we view the heat kernel $H(x,y,t)$ as a function of $y$. It follows from the semi-group property \nref{eqn:semigroup} of $H_\infty$ that $H_\infty(x,y,t)\in L^2(X,\omega)$. Hence by Lemma \ref{lemma 10.3}, there exists a sequence of real numbers $\{c_k(x,t)\}_{k}$ such that 
$$H_\infty(x,y,t) = \sum_{k=0}^\infty c_k(x,t) \phi_k(y),\quad \mbox{ in the $L^2$ sense.}$$
Note that here $c_k(x,t) = \int_X H_\infty (x,y,t) \phi_k(y) \omega^n(y)$ is the coefficient of $H_\infty(x,\cdot, t)$ associated to $\phi_k$. By Parseval's Identity, we have 
\begin{equation}\label{eqn:pasv}\sum_{k=0}^\infty |c_k(x,t)|^2 = \| H_\infty(x,\cdot, t)\|_{L^2}^2 = H_\infty(x,x,2t).\end{equation} On the other hand, in the proof of Lemma \ref{lemma 10.2}, we have already shown that $c_k(x,t) = e^{-\lambda _k t} \phi_k(x)$ since $\Delta \phi_k = -\lambda_k \phi_k$. This implies that we have the expansion formula \nref{eqn:heat last expansion} for $H_\infty$ {\em in the $L^2$ sense}. From the equation \nref{eqn:pasv}, we see that by \nref{eqn:H inf}
\begin{equation}\label{eqn:pasv1}
\sum_{k=0}^\infty e^{-2\lambda_k t} \phi_k(x)^2 = H_\infty(x,x,2t)\le \frac{1}{V_\omega} + C t^{-\frac{q}{q-1}}.
\end{equation} 
Integrating \nref{eqn:pasv1} over $x\in X^\circ$ we get
\begin{equation}\label{eqn:pasv2}
\sum_{j=1}^\infty e^{-2\lambda_j t} + 1\le 1 + C t^{-\frac{q}{q-1}},
\end{equation} 
since $\lambda_0 = 0$ and $\phi_0 = 1/\sqrt{V_\omega}$. For $k\ge 1$, taking $t = 1/\lambda_k$ in \nref{eqn:pasv2}, we get $\lambda_k \ge C k^{\frac{q-1}{q}}$. 

We now argue that the equation \nref{eqn:heat last expansion} holds smoothly. From the $L^\infty$ estimate \nref{eqn:L heat} of $\phi_k$, by elliptic regularity theory the eigenfunctions $\phi_k(x)$ are uniformly bounded (with any derivatives) by $O(\lambda_k^{\frac{q}{2(q-1)}})$ on any compact subset ${\mathcal K}\subset X^\circ$. This, together with the fast decaying of the factors $e^{-\lambda_k t} \le \exp(-C k^{\frac{q-1}{q}} t)$ for any $t>0$, immediately implies that the series on the RHS of \nref{eqn:heat last expansion} converges smoothly on ${\mathcal K}$. Hence the two sides of \nref{eqn:heat last expansion} coincide as smooth functions. 
\end{proof}


\section{K\"ahler-Einstein spaces}\label{section 11}
\setcounter{equation}{0}

In this section, we will prove Theorem \ref{kecy}, Theorem \ref{kegt} and Theorem \ref{kefano}. 

Let $X$ be a projective normal variety of $\dim X =n$. We let $X^\circ$ and $\cS_X$ be the regular and singular parts of $X$ respectively.  An adapted volume measure $\Omega$ on $X$ is introduced in \cite{EGZ}. For any $p\in X$, there exists an open neighborhood $U_p$ of $p$ in $X$ such that $\Omega= e^f |\eta|^{2/m}$ in $U_p$, where $f$ is a smooth function on $U_p$ and $\eta$ is a local generator of $m K_X$ in $U_p$ for some $m\in \mathbb{Z}$.  We consider a log resolution $\pi: Y \rightarrow X$ for $X$ and let $\theta_Y$ be a smooth K\"ahler metric on $Y$. Our goal is to show that the pullback of K\"ahler-Einstein currents will lie in an $\mathcal{AK}$-space on $Y$.

\medskip

\noindent{\bf Proof of Theorem \ref{kecy}}  There exists a unique Ricci-flat K\"ahler current $\omega_{CY}$ in any K\"ahler class of $X$. Let $\Omega$ be an adapted volume measure on $X$. Then there exists a unique $F\in L^\infty(X)\cap C^\infty(X^\circ)$ such that  
$$\omega^n = e^F\Omega.$$ 
Then there exist $p>n$ and $A, K>0$, such that $\omega\in \mathcal{AK}(X, \theta_Y, n, A, p, K, \gamma)$ and Theorem \ref{kecy} can be proved by directly applying  Theorem \ref{singsob}.   

\bigskip

\noindent{\bf Proof of Theorem \ref{kegt}} Since $K_X>0$, there exists a smooth adapted volume measure $\Omega$ on $X$ such that $\omega_0 = \ddbar\log \Omega$. The K\"ahler-Einstein equation is equivalent to the following complex Monge-Amp\`ere  equation
\begin{equation}\label{modeqn1}
(\pi^*\omega)^n =\left( \pi^*(\omega_0 + \ddbar \varphi) \right)^n = e^{\pi^* \varphi+F} \theta_Y^n = e^{\pi^*\varphi} \pi^*\Omega
\end{equation}
on $Y$, where $F = \log \frac{\pi^*\Omega}{(\theta_Y)^n}$. Equation (\ref{modeqn1}) can be solved if and only $(X, -K_X)$ is K-stable \cite{LC, LXZ}.  We cannot directly apply Proposition \ref{omapp} since $\varphi$ does not have log type analytic singularities in general. However,  since $F$ has log type analytic singularities, we will then modify the approximating equation (\ref{appeq6}) by 
$$ (\pi^* \omega_0 + \epsilon_j \theta_Y + \ddbar \varphi_j)^n = e^{F_j + \varphi_j}(\pi^* \theta_Y)^n, ~ \sup_X \varphi_j=0,$$
where   $\{F_j\}_{j=1}^\infty $ are chosen similarly in Section \ref{maregsec}. In particular, the approximation $F_j$ of $F$ satisfies (\ref{appr7}). Thus Proposition (\ref{omapp}) will still hold for the approximations of $\omega$ by 
$$\omega_j=\pi^*\omega_0 + \epsilon_j \theta_Y + \ddbar \varphi_j.$$ We can apply Theorem \ref{singsob} to $\omega$ since $\omega\in \mathcal{AK}(X, \theta_Y, n, A, p, K, \gamma)$. In particular,  there exists $\epsilon'>0$ and $C>0$ such that for all $j\geq 1$, 
$$\big\|\frac{(\omega_j)^n}{(\theta_Y)^n}\big\|_{L^{1+\epsilon'}(Y, (\theta_Y)^n)} \leq C. $$
The estimates of Theorem \ref{kegt} will follow by passing the uniform estimates for $\omega_j$ after letting $j\rightarrow \infty$  by applying the same arguments in Section \ref{singsobsec} and Section \ref{singdiasec}.   \qed


\medskip

Now we will consider the case when $X$ is a $\mathbb{Q}$-Fano variety with log terminal singularities. Let $\Omega$ be an adapted volume measure on $X$ such that
$$\omega_0 = -\ddbar \log \Omega \in [-K_X]$$
is a smooth K\"ahler metric on $X$. If $(X, -K_X)$ is K-stable, then there exists a unique   K\"ahler-Einstein current $\omega \in [-K_X]$ on $X$ with bounded local potentials. In particular,  $\omega=\omega_0 + \ddbar \varphi$ for some $\varphi\in PSH(X, \omega_0) \cap L^\infty(X).$ Furthermore, $\varphi\in C^\infty(X^\circ)$ by applying the smoothing properties of the K\"ahler-Ricci flow \cite{ST2, SzT}. The K\"ahler-Einstein equation is equivalent to the following complex Monge-Amp\`ere equation
$$(\omega_0 + \ddbar \varphi)^n = e^{-\varphi} \Omega$$
after adding a constant to $\varphi$ for normalization. Let  $\pi: Y \rightarrow X$ be the log resolution of $X$ and let $\theta_Y$ be a K\"ahler metric on $Y$. If we let
$$F= \log \frac{\pi^*\Omega}{(\theta_Y)^n}, $$ $F$ has log type singularities whose support coincide with the exceptional locus  of $\pi$.  In particular, there exists $\epsilon'>0$ such that
$\|e^F\|_{L^{1+\epsilon'}(Y, (\theta_Y)^n)} < \infty.$ We will choose  $\epsilon_j$ and $F_j$ as in Section \ref{maregsec} and further assume  that there exists $C>0$ such that for all $j\geq 1$,  
\begin{equation}\label{fjap11}
 \|e^{F_j}\|_{L^{1+\epsilon'}(Y, (\theta_Y)^n)} \leq C. 
 \end{equation}
The following follows from Demailly's regularization and the fact that $\varphi\in C^\infty(X^\circ)$. 

\begin{lemma} There exist a decreasing sequence of $\psi_j\in PSH(Y, \pi^*\omega_0 +  \theta_Y)\cap C^\infty(Y)$ such that $\psi_j$ converges to $\varphi$ pointwise on $X$. Furthermore, $\psi_j$ converges smoothly to $\varphi$ on any compact subset of $X^\circ$.   In particular, 
$$\sup_{j} \|\psi_j\|_{L^\infty(Y)} < \infty. $$

\end{lemma}

We will now consider the following complex Monge-Amp\`ere equation
$$(\pi^*\omega_0 + \epsilon_j \theta_Y+ \ddbar \varphi_j)^n = e^{-\psi_j + F_j + c_j}(\theta_Y)^n, ~\sup_Y \varphi_j =0, $$
where $c_j$ is the normalization constant. In particular, $\lim_{j\rightarrow \infty} c_j=0$.  We let 
$$\omega_j =\pi^*\omega_0 + \epsilon_j \theta_Y+ \ddbar \varphi_j$$
be the smooth K\"ahler metric on $Y$ for $j\geq 1$.

\begin{lemma} There exists $C>0$ such that for all $j\geq 1$, 
$$\|\varphi_j\|_{L^\infty(Y)}\leq C. $$

\end{lemma}

\begin{proof} The lemma follows by combining the $L^\infty$-estimate for complex Monge-Amp\`ere equation and  (\ref{fjap11}) as well as  the fact that $\psi_j$ is uniformly bounded below. 
\end{proof}

Let $D$ be an effective divisor of $Y$ such that its support conicides with the exceptional locus of $\pi$. Let $\sigma_D$ be a defining section of $D$ and $h_D$ be a smooth hermitian metric of the line bundle associated to $D$. We can assume that for any sufficiently small $\delta$, $\pi^*\omega_0 - \delta {\mathrm{Ric}}(h_D)>0$ is a K\"ahler metric.  The following lemma follows by the same argument in the proof of Lemma \ref{02nd}. 

\begin{lemma} There exists $K>0$, $C>0$ such that 
$$ \Delta_{\theta_Y} \varphi_j \leq C|\sigma_D|^{-2K}_{h_D}. $$

\end{lemma}

\begin{lemma} $\omega_j$ converges to $\omega$ locally smoothly on $C^\infty(X^\circ)$.

\end{lemma}

\begin{proof} The standard Schauder estimates and linear estimates can be obtain locally due to the second order estimates in Lemma. More precisely, for any $k>0$ and $\cK\subset \subset X\setminus \cS_X$,  there exists $C>0$ such that $$\|\varphi_j\|_{C^k(\cK)} \leq C. $$ Therefore $\varphi_j$ converges to $\varphi_\infty \in PSH(Y, \pi^*\omega_0) \cap C^\infty(\pi^*(X^\circ))$ satisfying
$$(\pi^*\omega+ \ddbar \varphi_\infty)^n = e^{-\varphi}\pi^*\Omega.$$
Then $\varphi_\infty =\varphi$ by the uniqueness of the bounded solution to the complex Monge-Amp\`ere equation. The lemma is then proved. 
\end{proof}

\medskip
\noindent{\bf Proof of Theorem \ref{kefano}.} Since $\omega_j$ converges to $\omega$ smoothly on $X^\circ$ and $\lim_{j\rightarrow \infty} \|\varphi_j- \varphi\|_{L^\infty(\cK)} = 0$ for any $\cK\subset\subset X^\circ$, the estimates in Theorem \ref{kefano} hold uniformly for $\omega_j$ and the theorem is proved by letting $j \rightarrow \infty$ by applying the same arguments in Section \ref{singsobsec} and Section \ref{singdiasec}.  \qed
 

\section{Finite time solutions of the K\"ahler-Ricci flow}\label{section 12}
\setcounter{equation}{0}


Let $(X, g_0)$ be a  projective manifold equipped with a K\"ahler metric  $g_0\in H^{1,1}(X, \mathbb{R})\cap H^2(X, \mathbb{Q})$. We will consider the K\"ahler-Ricci flow (\ref{krflow}) starting with $g_0$. Let $T>0$ be the maximal time for the flow (\ref{krflow}). Let $\omega_0$ and $\omega(t)$ be the K\"ahler forms associated to $g_0$ and $g(t)$, the solutions of (\ref{krflow}) for $t\in [0, T)$.  If $T<\infty$,   the limiting class $\alpha_T = \lim_{t\rightarrow T} ([\omega_0]+T[K_X])$ induces a unique surjective holomorphic map $\Phi: X \rightarrow Y$ as in (\ref{kawam}).  If $\alpha_T$ is big, $\dim Y= n$ and we let $\cS$ be the exceptional locus of $\Phi$. $\cS$ is then an analytic subvariety of $X$. 

The following lemma follows directly from the maximum principle (c.f. \cite{TZ, ST2}). 
\begin{lemma} Suppose $T<\infty$, then there exists $C>0$ such that for all $t\in [0, T)$, 
\begin{equation}\label{volrat0}
\left( \omega(t) \right)^n \leq C \left(\omega_0 \right)^n.
\end{equation}
If we further suppose  the limiting class $\alpha_T$ is big. Then $\omega(t)$ converges to the pullback of a closed positive $(1,1)$-current $\omega_T$ on $Y$ and the convergence is smooth on $X \setminus \cS$.

\end{lemma}

The limiting current $\omega(T)$ has bounded local potentials and there exists $C>0$ such that 
\begin{equation}\label{volrat}
\omega^n (T) \leq C(\omega_0)^n
\end{equation}
 is uniformly bounded above. Theorem \ref{thm:main2} will then follow by applying Theorem \ref{singsob} with estimates (\ref{volrat0}) and (\ref{volrat0}).

\bigskip

\noindent {\bf Acknowledgement:} We would like to thank H. Guenancia and M. Paun for their interest and helpful comments on this paper.

\setcounter{equation}{0}

\end{document}